\documentclass[11pt]{amsart}

\makeatletter

\renewcommand\subsubsection{\@secnumfont}{\bfseries\itshape}%
\renewcommand\subsubsection{\@startsection{subsubsection}{3}
	\z@{.5\linespacing\@plus.7\linespacing}{-.5em}%
	{\normalfont\bfseries\itshape}}

\makeatother

\usepackage{etoolbox}

\makeatletter
\patchcmd{\@setaddresses}{\indent}{\noindent}{}{}
\patchcmd{\@setaddresses}{\indent}{\noindent}{}{}
\patchcmd{\@setaddresses}{\indent}{\noindent}{}{}
\patchcmd{\@setaddresses}{\indent}{\noindent}{}{}
\makeatother

\usepackage{graphicx}
\usepackage{amssymb}
\usepackage{amsmath}
\usepackage{amsthm}
\usepackage{xcolor}
\usepackage{bm}
\usepackage{placeins}
\usepackage{multirow}
\usepackage{float}
\usepackage{enumitem}
\usepackage[a4paper,left=1.3in,right=1.3in,top=1.3in,bottom=1.3in]{geometry}
\usepackage{url}
\usepackage{enumitem}
\usepackage{tikz-cd}
\usepackage[all]{xy}
\usepackage{url}
\usepackage{array,booktabs}
\usepackage{algorithmicx}
\usepackage{algpseudocode}
\usepackage{algorithm}
\usepackage[hidelinks]{hyperref}
\usepackage{verbatim}
\usepackage{etex}
\usepackage[width=0.9\textwidth]{caption}
\usetikzlibrary{decorations.pathmorphing} 
\makeatletter
\DeclareRobustCommand*{\bfseries}{%
	\not@math@alphabet\bfseries\mathbf
	\fontseries\bfdefault\selectfont
	\boldmath
}
\makeatother
\usepackage{xargs}                      
\usepackage[colorinlistoftodos,prependcaption,textsize=tiny]{todonotes}
\newcommandx{\change}[2][1=]{\todo[linecolor=red,backgroundcolor=red!25,bordercolor=red,#1]{#2}}
\newcommandx{\unsure}[2][1=]{\todo[linecolor=blue,backgroundcolor=blue!25,bordercolor=blue,#1]{#2}}
\newcommandx{\info}[2][1=]{\todo[linecolor=OliveGreen,backgroundcolor=OliveGreen!25,bordercolor=OliveGreen,#1]{#2}}
\newcommandx{\improvement}[2][1=]{\todo[linecolor=Plum,backgroundcolor=Plum!25,bordercolor=Plum,#1]{#2}}
\newcommandx{\thiswillnotshow}[2][1=]{\todo[disable,#1]{#2}}
\makeatletter
\providecommand\@dotsep{5}
\renewcommand{\listoftodos}[1][\@todonotes@todolistname]{%
	\@starttoc{tdo}{#1}}
\makeatother

\numberwithin{equation}{section} 

\newtheorem{theorem}[equation]{Theorem}

\newtheorem*{theorem:correct:manifold}{Theorem \ref{thm:mutant:correct:manifold}}
\newtheorem*{theorem:veering:mutant}{Theorem \ref{thm:mutant:veering}}
\newtheorem*{theorem:veering:general}{Theorem \ref{thm:generalized}}
\newtheorem*{fact:summary}{Fact \ref{thm:mutation:properties}}
\newtheorem*{fact:monodromies}{Fact \ref{fact:monodromies:not:conjugate}}
\newtheorem*{theorem:many:flows}{Theorem \ref{thm:many:flows}}
\newtheorem{fact}[equation]{Fact}
\newtheorem{proposition}[equation]{Proposition}
\newtheorem*{prop:tautness}{Proposition \ref{prop:taut:sufficient:necessary}}
\newtheorem*{fact:polys}{Fact \ref{fact:different:polys}}
\newtheorem{lemma}[equation]{Lemma}
\newtheorem{corollary}[equation]{Corollary}

\theoremstyle{definition}
\newtheorem{definition}[equation]{Definition}

\newtheorem{remark}[equation]{Remark}

\newtheorem*{remark*}{Remark}
\newtheorem{question}{Question}

\newcommand{\bez}{-}
\newcommand{\ro}{\varrho}
\newcommand{\zz}{\mathbb{Z}}
\newcommand{\rr}{\mathbb{R}}

\newcommand{\face}[1]{\raisebox{.03em}{\large{$\mathtt{#1}$}}}

\newcommand{\Thnorm}[1]{\left\lVert #1 \right\rVert_{\mathrm{Th}}}

\newcommand{\Aut}{\mathrm{Aut}}
\newcommand{\V}{\mathcal{V}}
\newcommand{\B}{\mathcal{B}}
\newcommand{\T}{\mathcal{T}}

\newcommand{\C}{\mathcal{C}}

\newcommand{\ab}{\mathfrak{a}}

\newcommand{\Q}{\mathcal{Q}}
\newcommand{\D}{\mathcal{D}}
\newcommand{\HH}{\mathcal{H}}
\newcommand{\Semb}{S_w^{\epsilon}}
\newcommand{\Sembp}{S_w^{\epsilon +}}
\newcommand{\Sembm}{S_w^{\epsilon -}}
\newcommand{\QVw}{\mathcal{Q}_{\V,w}}
\newcommand{\cone}[1]{\rr_+\hspace{-0.1cm}\cdot \raisebox{.05em}{$\mathtt{#1}$}} 

\newcommand{\coll}{\mathrm{coll}}

\algrenewcommand{\algorithmiccomment}[1]{\hspace*{\fill}
	\color{gray}\small $\#$  #1 \color{black}\normalsize}

\makeatletter
\@namedef{subjclassname@2020}{%
	\textup{2020} Mathematics Subject Classification}
\makeatother
\usepackage{doc}
\makeatletter
\def\maketitle{\par
	\begingroup
	\def\thefootnote{\fnsymbol{footnote}}%
	\setcounter{footnote}\z@
	\def\@makefnmark{\hbox to\z@{$\m@th^{\@thefnmark}$\hss}}%
	\long\def\@makefntext##1{\noindent
		\ifnum\c@footnote>\z@\relax
		\hbox to1.8em{\hss$\m@th^{\@thefnmark}$}##1%
		\else
		\hbox to1.8em{\hfill}%
		\parbox{\dimexpr\linewidth-1.8em}{\raggedright ##1}%
		\fi}
	\if@twocolumn\twocolumn[\@maketitle]%
	\else\newpage\global\@topnum\z@\@maketitle\fi
	\thispagestyle{titlepage}\@thanks\endgroup
	\setcounter{footnote}\z@
	\gdef\@date{\today}\gdef\@thanks{}%
	\gdef\@author{}\gdef\@title{}}

\makeatother
\begin{document}

	\numberwithin{equation}{section}
	\title[Mutations and faces of the Thurston norm ball dynamically represented by multiple distinct flows]{Mutations and faces of the Thurston norm ball \\dynamically represented by multiple \\distinct flows}
	
	\author{Anna Parlak}
	\address{Mathematical Sciences Building \\ University of California, One Shields Avenue, Davis, CA 95616, United States}
	\email{anna.parlak@gmail.com}

	\thanks{This work was partially supported by the Simons Investigator Award No. 409745 of Vladimir Marković.}


	\begin{abstract}
		A pseudo-Anosov flow on a hyperbolic 3-manifold dynamically represents a face~$\face{F}$ of the Thurston norm ball if the cone on $\face{F}$ is dual to the cone spanned by homology classes of closed orbits of the flow. Fried showed that for every fibered face of the Thurston norm ball there is a unique, up to isotopy and reparametrization, flow which dynamically represents the face. Using veering triangulations we have found that there are non-fibered faces of the Thurston norm ball which are dynamically represented by multiple topologically inequivalent flows. This raises a question of how distinct flows representing the same face are related. 
		
		We define combinatorial mutations of veering triangulations along surfaces that they carry. We give sufficient and necessary conditions for the mutant triangulation to be veering. After appropriate Dehn filling these veering mutations correspond to transforming one 3-manifold $M$ with a pseudo-Anosov flow transverse to an embedded surface $S$ into another 3-manifold admitting a pseudo-Anosov flow  transverse to a surface homeomorphic to $S$. We show that a non-fibered face of the Thurston norm ball can be dynamically represented by two distinct flows that differ by a veering mutation.
		
	\end{abstract}
	
	\maketitle%
	\setcounter{tocdepth}{1}
	
	\vspace{2cm}
	\tableofcontents
	
		\newpage 
	\section{Introduction}\label{sec:intro}
	Let $M$ be a compact, oriented 3-manifold $M$ whose interior admits complete hyperbolic structure. The \emph{Thurston norm} on $H_2(M, \partial M;\rr)$ measures the minimal topological complexity of  surfaces that represent a homology class \cite{Thur_norm}. It has been intensively studied in various different contexts. It is related to finite depth foliations~\cite{Gabai-foliations83}, the Alexander polynomial \cite{Dunfield_Alex, McMullen_Alex},  the $L^2$-torsion function \cite{L2torsion}, Floer homology \cite{Floer_Thurston},  and many other aspects of 3-dimensional topology. In this paper we focus on the connection between the Thurston norm on $H_2(M, \partial M;\rr)$ and nonsingular flows on $M$. Originally this connection was drawn by  Fried \cite{Fried_cross-sections, Fried_suspension} and Mosher \cite{Mosher-dynamical1992, Mosher_manuscipt}. The topic has reemerged recently in the work of Landry \cite{Landry_stable, Landry_branched, Landry_homology_isotopy}, and Landry-Minsky-Taylor \cite{LMT_flow, LMT}, where they relate the Thurston norm with veering triangulations.
	
	Since the unit norm ball $\mathbb{B}_{\mathrm{Th}}$ of the Thurston norm is a compact polytope \cite[Theorem~2]{Thur_norm}, we can speak about its \emph{faces}. Thurston proved that all ways in which $M$ fibers over the circle are encoded by finitely many, potentially zero, top-dimensional faces of $\mathbb{B}_{\mathrm{Th}}$, called \emph{fibered faces} \cite[Theorem~3]{Thur_norm}. The first known  connection between pseudo-Anosov flows and the Thurston norm concerned only these faces. Assuming that~$M$ is closed, Fried proved that associated to a fibered face $\face{F}$ there is a unique, up to isotopy and reparametrization, pseudo-Anosov flow $\Psi$ on $M$ with the property that a class $\eta \in H_2(M;\zz)$ can be represented by a cross-section to $\Psi$ if and only if $\eta$ is in the interior of the cone $\cone{\face{F}}$ \cite[Theorem~7]{Fried_cross-sections}. 
	Mosher extended Fried's result by showing that $\eta \in H_2(M; \zz)$ can be represented by a surface that is \emph{almost transverse} to $\Psi$ if and only if $\eta$ is  in $\cone{\face{F}}$ \cite[Theorem 1.4]{Mosher_branched}. 
	Results of Fried and Mosher are stated for closed manifolds, but they can be generalized to the case of flows on 3-manifolds with toroidal boundary whose interior admits a complete hyperbolic structure; see \cite[Theorem 3.5]{Landry_stable}. In this case, the relevant flows are obtained from pseudo-Anosov flows by \emph{blowing-up} finitely many closed orbits into toroidal boundary components; see \cite[Section~3.2]{Mosher_manuscipt} and \cite[Section 3.6]{Bonatti-surgeries}.
	
	The relationship between pseudo-Anosov flows and the Thurston norm extends beyond the fibered case. In this more general setup we consider flows which do not admit cross-sections. Such flows are called \emph{non-circular}.
	Given a potentially non-circular flow $\Psi$ on $M$ denote by $\C(\Psi)$ the cone in $H_2(M, \partial M;\rr)$ spanned by  homology classes whose algebraic intersection with homology classes of closed orbits of $\Psi$ is nonnegative. We say that~$\Psi$ \emph{dynamically represents} a  (not necessarily fibered, not necessarily top-dimensional) face  $\face{F}$ of the Thurston norm ball in $H_2(M, \partial M;\rr)$ if $\mathcal{C}(\Psi) = \cone{F}$. This is a slight modification of Mosher's terminology from ~\cite{Mosher-dynamical1992}; see Definition \ref{defn:dynamical_rep} and the discussion below it.
	From the results of Fried and Mosher mentioned in the last paragraph it follows that every fibered face is dynamically represented by a flow which is unique up to isotopy and reparametrization. In the non-fibered case, Mosher found sufficient conditions on a non-circular flow to dynamically represent a face of the Thurston norm ball \cite[Theorem~2.7]{Mosher-dynamical1992} and showed that there are non-circular flows representing non-fibered faces \cite[Section 4]{Mosher-dynamical1992}.  However, the question whether for every non-fibered face $\face{F}$ of the Thurston norm ball in $H_2(M, \partial M;\rr)$ there is a (blown-up) pseudo-Anosov flow $\Psi$ which dynamically represents $\face{F}$ remains open.
	
	In this paper we answer two  closely related questions. First, if there is a flow which dynamically represents~a non-fibered face, is this flow necessarily unique, up to isotopy and reparametrization?  In Section \ref{sec:distinct:flows} we give explicit examples of flows which represent the same non-fibered face of the Thurston norm ball but are not even topologically equivalent, thus showing that the answer to this question is negative; see Theorem \ref{thm:many:flows}. These examples have been found using \emph{veering triangulations}, a combinatorial tool to study pseudo-Anosov flows. We refer the reader to Subsection \ref{sec:veering:and:flows} for an outline of the connection between veering triangulations and pseudo-Anosov flows.
	
	Once we know that a non-fibered face can be dynamically represented by two topologically inequivalent flows we may ask how the two distinct flows which dynamically represent the same face are related. Veering triangulations can be helpful in solving this problem as well. In particular, the  Veering Census \cite{VeeringCensus} and computational tools to study  triangulations \cite{regina, snappea, VeeringGitHub} can be used to find many examples of veering triangulations that \emph{combinatorially represent} the same face of the Thurston norm ball. At the beginning of Section \ref{sec:faces} and in  Subsection \ref{subsubsec:higher:betti} we briefly outline what the search for appropriate examples boils down to. Furthermore, veering triangulations are finite objects that satisfy very restrictive conditions; see Definition~\ref{def:veering}. This makes comparing two veering triangulations an easier task than comparing their underlying flows. An analysis of certain examples of veering triangulations which combinatorially represent the same face of the Thurston norm ball led us to define \emph{combinatorial mutations of veering triangulations} along surfaces that they carry. The main goal of this paper is to carefully study these operations and demonstrate that in special cases they can yield distinct flows representing the same face of the Thurston norm ball.

	\subsection{Combinatorial mutations of veering triangulations}
	A veering triangulation $\V$ is determined by three pieces of combinatorial data: an ideal triangulation $\T$, a taut structure $\alpha$ on $\T$, and a smoothening of the dual spine of~$\T$ into a branched surface~$\B$ with certain properties; see Definitions \ref{def:taut} and \ref{def:veering}. 
	Associated to $(\T, \alpha)$ there is a finite system of \emph{branch equations} such that if $w$ is a nonzero, nonnegative, integral solution to this system then $w$ determines a surface $S_w$ which is \emph{carried} by $(\T, \alpha)$; see Subsection~\ref{subsec:carried:surfaces}. We also say that $S_w$ is carried by $\V$. This surface is naturally equipped with an ideal triangulation $\QVw$ induced from $\V$. Let $\Aut^+(\QVw)$ be the group of orientation-preserving combinatorial automorphisms of $\QVw$. Associated to  $\varphi \in \Aut^+(\QVw)$ there always is a \emph{mutant manifold} $M^\varphi$, obtained from~$M$ by cutting it open along a properly embedded surface~$\Semb$ isotopic to $S_w$ and then identifying the two copies of~$\Semb$ in the boundary of the resulting sutured manifold $M|\Semb$ via $\varphi$. Our goal is to mimic this construction in the combinatorial setup of triangulations.  Unfortunately, it is not as straightforward as it may sound. The main difficulty is the fact that~$S_w$ is often not embedded. Thus we may view cutting $\T$ along~$S_w$ as equivalent to cutting it along a certain branched surface~$F_w$ which fully carries $S_w$; see Subsection \ref{subsec:cut:tri:versus:cut:manifold}. This in turn causes the problem of not being able to use~$\varphi$ directly to reglue the top boundary $F_w^+$ of $\T|F_w$ to its bottom boundary~$F_w^-$. 
	In Subsection~\ref {subsec:mutating} we define a \emph{regluing map} $r(\varphi):F_w^+ \rightarrow F_w^-$ determined by
	$\varphi$  and use it to define a \emph{mutant triangulation} $\T^\varphi$. Without further assumptions on $\varphi$ not only can this triangulation fail to be veering, but it also might not be a triangulation of~$M^\varphi$.  
	We deal with these issues in Subsections~\ref{subsec:mutant:homeomorphism} and \ref{subsec:veeringness:of:mutant}.

	Studying mutations has a long history, particularly in knot theory; see for instance \cite{Dunfield_mutation, Kirk-Livingston-concordance, Millichap-mutations, Morton-Traczyk}. 
	Mutant knots share many properties, and much work on mutations concentrates on establishing which knot invariants distinguish mutants. Another thread in the theory is finding sufficient conditions on a surface $S$ and its homeomorphism $\varphi$ so that $M$ and $M^\varphi$ share some property. For instance, in \cite[Theorem 4.4]{Ruberman-mutations} Ruberman considered mutations of hyperbolic 3-manifolds and found sufficient conditions for the mutant manifold $M^\varphi$ to be hyperbolic and have the same hyperbolic volume as $M$. 	 Our goal to find conditions under which $\T^\varphi$ is a veering triangulation of $M^\varphi$ work fits into this second framework. 
	
	\subsection{Properties of the mutant triangulation}
	To analyze the homeomorphism type of the manifold underlying $\T^\varphi$ we introduce the notion of  \emph{edge product disks}, a special type of product disks in the sutured manifold~$M|\Semb$; see Subsection \ref{subsec:cut:tri:versus:cut:manifold}.  Then we define what does it mean for $\varphi \in \Aut^+(\QVw)$ to \emph{misalign edge product disks}; see Definition \ref{defn:aligns:product:disks}. 
	Using this we prove:	
	\begin{theorem:correct:manifold}
		The mutant triangulation $\T^\varphi$ is an ideal triangulation of $M^\varphi$ if and only if $\varphi$ misaligns  edge product disks.
	\end{theorem:correct:manifold}
To find sufficient conditions for the mutant triangulation to be veering we first need to  ensure that it admits a taut structure. It turns out that for this it also  suffices to assume that $\varphi$ misaligns edge product disks. However, we prove a slightly stronger result.
	
	\begin{prop:tautness}
		Ideal triangulation $\T^\varphi$ admits a taut structure if and only if every vertical annulus or M\"obius band in $M^\varphi$ lies in a prismatic region of $M^\varphi$. 
	\end{prop:tautness}

The backward direction of Proposition \ref{prop:taut:sufficient:necessary} is proved by explicitly constructing a taut structure $\alpha^\varphi$ on $\T^\varphi$ from the taut structure $\alpha$ on $\T$. We say that $(\T, \alpha)$ and $(\T^\varphi, \alpha^\varphi)$ are \emph{taut mutants}. 	
	
	Intuitively, the condition that appears in the above proposition means that $\varphi$ might align edge product disks but it does so in a way which is not visible from the perspective of~$\T^\varphi$; see Lemma \ref{lem:prismatic:vertical}. Nonetheless, in light of Theorem \ref{thm:mutant:correct:manifold} it is convenient to assume that $\varphi$ misalign edge product disks, so that we deal only with triangulations of~$M^\varphi$. This assumption is further justified by the fact that in Proposition \ref{prop:not:hyperbolic} we prove that when $\varphi$ aligns edge product disks $M^\varphi$ cannot admit a veering triangulation.

	To obtain sufficient conditions on the taut triangulation $(\T^\varphi, \alpha^\varphi)$ to admit a veering structure, we  make use of the branched surface $\B$ defining the veering structure on $\V = (\T, \alpha, \B)$. This branched surface intersects the 2-skeleton of $\T$ in a train track; see  Figure \ref{fig:stable:track}. Therefore any surface $S_w$ carried by $\V$ inherits a train track $\tau_{\V, w}$ which is dual to its ideal triangulation $\QVw$. By $\Aut^+(\Q_{\V, w} \ | \ \tau_{\V, w})$ we denote the subgroup of  $\Aut^+(\QVw)$  consisting of orientation-preserving combinatorial automorphisms of $\QVw$ which preserve $\tau_{\V, w}$. 
	
	\begin{theorem:veering:mutant}
		Let $S_w$ be a surface carried by a veering triangulation $\V = (\T, \alpha, \B)$ of $M$.  Suppose that $\varphi \in \Aut^+(\QVw)$ misaligns edge product disks. If additionally \linebreak $\varphi \in \Aut^+(\QVw \ | \ \tau_{\V, w})$ then $(\T^\varphi, \alpha^\varphi)$ admits a veering structure.
	\end{theorem:veering:mutant}

	Under the assumptions of this theorem,  the branched surface $\B$ dual to $\T$ mutates into a branched surface $\B^\varphi$ that is dual to $\T^\varphi$ and satisfies Definition \ref{def:veering}. We say that $\V^\varphi = (\T^\varphi, \alpha^\varphi, \B^\varphi)$ is obtained from $\V = (\T, \alpha, \B)$ by a \emph{veering mutation} or that $\V^\varphi, \V$ are \emph{veering mutants}.
	
	Observe that Theorem \ref{thm:mutant:veering} gives a sufficient condition  for a taut mutant $(T^\varphi, \alpha^\varphi)$ to be veering. It is, however, possible that $(\T^\varphi, \alpha^\varphi)$ admits a veering structure even when $\varphi \notin \Aut^+(\QVw \ | \ \tau_{\V, w})$. This can happen whenever after cutting~$\T$ along  $F_w$  the cut triangulation $\T|F_w$ admits a veering structure $\B^\ast|F_w$ which mutates into a branched surface that is dual to $\T^\varphi$ and satisfies Definition~\ref{def:veering}. If $\B^\ast|F_w \neq \B|F_w$ we do not consider such triangulations to be veering mutants. This construction can be used to prove a generalization of Theorem~\ref{thm:mutant:veering} giving a sufficient and necessary conditions on a taut mutant of a veering triangulation to be veering.
	
	\begin{theorem:veering:general}
		Let $S_w$ be a surface carried by a veering triangulation $\V = (\T, \alpha, \B)$ of~$M$.  Suppose that $\varphi \in \Aut^+(\QVw)$ misaligns edge product disks. The taut triangulation $(\T^\varphi, \alpha^\varphi)$ admits a veering structure if and only if there is a veering structure $\B^\ast|F_w$ on $(\T|F_w, \alpha|F_w)$ such that the isomorphism $\varphi: \QVw^+ \rightarrow \QVw^-$ sends $\tau^{\ast \ +}_{\V, w}$ to $\tau_{\V, w}^{\ast \ -}$.
		
	\end{theorem:veering:general}

	In Subsection~\ref{subsec:generalization} we give an example of a pair of veering triangulations which are taut mutants but not veering mutants. This proves that the generalization appearing in Theorem \ref{thm:generalized} is not just theoretical, but actually arises in practice. In the same subsection we also define a \emph{veering mutation with insertion}, a certain generalization of a veering mutation where the related triangulations have different number of tetrahedra.

		\subsection{Homeomorphic veering mutants}
	In Section \ref{sec:faces} we analyze a few examples of homeomorphic veering mutants. Apart from illustrating the constructions that are defined in earlier parts of the paper, we use these examples to  establish the following facts connecting veering mutations and faces of the Thurston norm ball.
	\begin{fact:summary} \emph{(Veering mutations and faces of the Thurston norm ball)}
		\begin{enumerate}
		\item A non-fibered face $\face{F}$ of the Thurston norm ball of a compact, oriented, hyperbolic 3-manifold with boundary can be represented by two combinatorially non-isomorphic veering mutants.
		\item Performing a veering mutation along a surface representing a class lying at the boundary of a fibered face may yield a veering triangulation representing a non-fibered face of the Thurston norm ball of the mutant manifold.
		\end{enumerate}
	\end{fact:summary}
	
Analyzing two veering mutants of the complement of the $10^3_{12}$ link leads to the following discovery.
	\begin{fact:monodromies}
		The complement of the $10^3_{12}$ link admits two fibrations over the circle such that
		\begin{itemize} \item The fiber is a genus two surface with four punctures.
			\item The monodromy of one fibration is obtained from the monodromy of the other fibration by postcomposing it with an involution. In particular, the stretch factors of monodromies are equal.
			\item The monodromies are not conjugate in the mapping class group of a genus two surface with four punctures.
		\end{itemize}
	\end{fact:monodromies}
	The last part of Fact \ref{fact:monodromies:not:conjugate} follows from an observation that the Euler classes of the two fibrations lie in different orbits under the action of $\mathrm{Homeo}(M)$ on $H^2(M,\partial M;\rr)$. 
	Examples of such fibrations of the same manifold have been known before; see for instance \cite[Theorem 1.2]{McMullen-Taubes}. What is new here is that we get fiber bundles which are not isomorphic, even though both their total spaces and fibers are homeomorphic, and the stretch factors of monodromies are the same. 
	
	\subsection{Multiple distinct flows dynamically representing the same face of the Thurston norm ball}
	
Given a veering triangulation $\V$ of $M$ it is possible to construct a transitive pseudo-Anosov flow $\Psi$ on a closed Dehn filling $N$ of $M$, provided that a certain natural condition on the Dehn filling slopes is satisfied \cite[Theorem 5.1 (stated here as Theorem~\ref{thm:AT})]{Tsang-Agol}. Let $\Psi^\circ$ be the blown-up flow on $M$. If $\V$ combinatorially represents a face $\face{F}$ of the Thurston norm ball in $H_2(M, \partial M;\rr)$ then $\Psi^\circ$ dynamically represents~$\face{F}$  \cite[Theorem 6.1 (stated here as Theorem \ref{thm:LMT:cones})]{LMT_flow}. Under additional assumptions on the Dehn filling slopes, there is also a face $\face{F}_N$ of the Thurston norm ball in $H_2(N;\rr)$ which is dynamically represented by $\Psi$ \cite[Theorem A (stated here as Theorem \ref{thm:L:cones})]{Landry_homology_isotopy}. These results are the main ingredients to  prove the following theorem.
	
\begin{theorem:many:flows}
		A non-fibered face $\face{F}$ of the Thurston norm ball can be dynamically represented by two topologically inequivalent flows.
\end{theorem:many:flows}

In the case of manifolds with nonempty boundary we show that a non-fibered face can be dynamically represented by two topologically inequivalent blown-up Anosov flows constructed from a pair of homeomorphic veering mutants. Unfortunately, the corresponding Anosov flows on a closed manifold cannot be used to prove the theorem in the closed case because the manifold is toroidal. For this reason, we refer to a different pair of veering triangulations which represent the same face of the Thurston norm ball and after appropriate Dehn filling yield transitive pseudo-Anosov flows on a hyperbolic 3-manifold.
In particular, it is important to note that even though the focus of this paper is on veering mutants, not all pairs of veering triangulations combinatorially representing the same face of the Thurston norm ball are related by a veering mutation or even a veering mutation with insertion; see Fact \ref{fact:not:mutant}.
	
	\begin{remark*}
		Although Anosov flows underlying homeomorphic veering mutants $\V$, $\V^{\ro\sigma}$ discussed in Subsection~\ref{subsec:same:face} cannot be used to prove Theorem \ref{thm:many:flows} in the closed case, they have another interesting feature. The closed manifold $N$ obtained by Dehn filling $M \cong M^{\ro\sigma}$ along the boundary of the mutating surface  is a graph manifold constructed from the orientable circle bundle over a 2-holed $\rr P^2$ by identifying its two toroidal boundary components. Such manifolds are called BL-manifolds in \cite{Barbot-BL}. The Anosov flows $\Psi$, $\Psi^{\ro\sigma}$ on $N$ built from $\V, \V^{\ro \sigma}$, respectively, are counterexamples to the claim, which appears as Theorem B(2) of \cite{Barbot-BL}, that all non $\rr$-covered Anosov flows on a  fixed BL-manifold are topologically equivalent; see Remark \ref{remark:counterexamples:to:Barbot}.
	\end{remark*}

	\subsection{Polynomial invariants of veering triangulations representing the same face of the Thurston norm ball}
	
	In \cite{LMT} Landry, Minsky, and Taylor introduced two polynomial invariants of veering triangulations: the \emph{taut polynomial} and the \emph{veering polynomial}. They proved that the taut polynomial generalizes the \emph{Teichm\"uller polynomial}, an invariant of a fibered face of the Thurston norm ball defined by McMullen in \cite{McMullen_Teich}, to  faces of the Thurston norm ball combinatorially represented by veering triangulations \cite[Theorem 7.1]{LMT}. In Section \ref{sec:polys} we compute the taut and veering polynomials of veering triangulations representing the same face of the Thurston norm ball discussed in Sections \ref{sec:faces} and \ref{sec:distinct:flows}. We deduce that in the non-fibered case the taut and veering polynomials are not invariants of faces of the Thurston norm ball combinatorially represented by veering triangulations.
	\begin{fact:polys}
		A non-fibered face of the Thurston norm ball can be combinatorially represented by two distinct veering triangulations with different taut polynomials, and different veering polynomials. 
	\end{fact:polys}

	\subsection{Further questions}
	In Section \ref{sec:questions} we speculate about what happens on the level of flows when we perform a veering mutation. We also ask a few questions concerning veering mutations, faces of the Thurston norm ball dynamically represented by multiple distinct flows, connections between this work and a recent result of Barthelm\'e-Frankel-Mann \cite{pA-classification} characterizing topologically inequivalent pseudo-Anosov flows on a fixed manifold, 
	and hyperbolic volumes of veering mutants.

	\subsection*{Acknowledgements}
	I am grateful to Michael Landry and Chi Cheuk Tsang for answering  my questions about  their work connecting veering triangulations and pseudo-Anosov flows.  I thank Saul Schleimer for discussions on veering triangulations analyzed in Section \ref{subsec:same:face}. I also thank Thomas Barthelm\'e and Chi Cheuk Tsang for discussions that led to Remark \ref{remark:counterexamples:to:Barbot}. 
	
	\noindent This project was partially supported by the Simons Investigator Award No. 409745 of Vladimir Marković.
	
	\section{Veering triangulations and pseudo-Anosov flows}\label{sec:veering}
	
	Let $M$ be a compact, oriented 3-manifold. By an \emph{ideal triangulation} of $M$ we mean an expression of $M \bez \partial M$ as a collection of finitely many ideal tetrahedra with triangular faces identified in pairs by homeomorphisms which send vertices to vertices. 
	Links of ideal vertices of the triangulation correspond to boundary components of $M$. 

	Let $\T$ be a finite ideal triangulation of $M$. Every triangular face of $\T$ has two \emph{embeddings} into two, not necessarily distinct, tetrahedra. 
	Every edge of $\T$ has finitely many  embeddings into tetrahedra of $\T$ and the same number of embeddings into faces of $\T$.
	By \emph{edges of a triangle/tetrahedron} or \emph{triangles of a tetrahedron} we mean embeddings of these ideal simplices into the boundary of a higher dimensional ideal simplex.  Similarly, by \emph{triangles/tetrahedra attached to an edge} we mean triangles/tetrahedra in which the edge is embedded, together with this embedding. Observe that triangles/tetrahedra attached to an edge can be circularly ordered, hence we can speak about consecutive triangles/tetrahedra attached to an edge.
	
	
	Every ideal triangulation  $\T$ determines a 2-dimensional complex $\mathcal{D}$ dual to $\T$, called the \emph{dual spine} of $\T$. 
	For every tetrahedron $t$ of $\T$ there is a vertex $v = v(t)$ of $\mathcal{D}$. If tetrahedra $t_1, t_2$ of $\T$ admit faces $f_1, f_2$, respectively, which are identified in $\mathcal{T}$, then in~$\D$ there is an edge  joining their dual vertices $v_1, v_2$.
	Finally, each edge  $e$ of $\T$ gives a 2-cell of $\D$ which is glued along the edges of $\D$ which are dual to the consecutive triangles attached to~$e$. Since there are no higher dimensional cells in $\D$, and 0- and 1-cells have special names, we will often refer to the 2-cells of $\D$ as just `cells'.
	
	Translating between properties of an ideal triangulation and properties of its dual spine is  straightforward. 
	Throughout the paper we freely alternate between these two perspectives depending on which one is more useful in a given context.
	
	\subsection{Taut triangulations}\label{subsec:taut}
	In \cite[Introduction]{Lack_taut} Lackenby introduced \emph{taut ideal triangulations} of 3-manifolds. Using the duality between an ideal triangulation and its dual spine we define tautness of an ideal triangulation in terms of properties of its dual spine.

	\begin{definition}\label{def:taut}
		A \emph{taut structure} $\alpha$ on an ideal triangulation $\T$ is a choice of orientations on the edges of its dual spine $\D$ such that 
		\begin{enumerate}
			\item every vertex $v$ of $\D$ has two incoming edges and two outgoing edges,
			\item every cell $s$ of $\D$ has exactly one vertex $b_s$ such that the two edges of $s$ adjacent to~$v$ both point out of $b_s$,
			\item every cell $s$ of $\D$ has exactly one vertex $t_s$ such that the two edges of~$s$ adjacent to~$v$ both point into $t_s$.
		\end{enumerate}
	\end{definition}
	A \emph{taut triangulation} is a pair $(\T, \alpha)$, where $\T$ is an ideal triangulation, and $\alpha$ is a taut structure on $\T$. If $(\T, \alpha)$ is taut then for every cell $s$ of the dual spine of $\T$ the vertex from Definition \ref{def:taut}(2) is called the \emph{bottom vertex} of~$s$, and the vertex from Definition \ref{def:taut}(3) is called the \emph{top vertex} of $s$.
	
	\begin{remark}
		Taut triangulations are often called \emph{transverse taut triangulations}; see for instance \cite{Parlak-computation, taut_alex, SchleimSegLinks}.
	\end{remark}

	Intuitively, tautness of an ideal triangulation gives an upwards direction which is consistent throughout the whole triangulation. 
	Under the duality, orientations on the edges of $\D$ translate into coorientations on the faces of $\T$. If $(\T, \alpha)$ is taut then, by Definition \ref{def:taut}(1), every  tetrahedron $t$ of $\T$ has two faces whose coorientations point into $t$, and two faces whose coorientations point out of $t$. We call the pair of faces whose coorientations point out of $t$ the \emph{top faces} of~$t$ and the pair of faces whose coorientations point into~$t$ the \emph{bottom faces} of $t$. 
	We also define the \emph{top diagonal} of $t$ to be the common edge of the two top faces of~$t$ and the \emph{bottom diagonal} of $t$ to be the common edge of the two bottom faces of $t$.  By Definition \ref{def:taut}(2), every edge of $\T$ is embedded as the top diagonal in precisely one tetrahedron of $\T$. Similarly, Definition~\ref{def:taut}(3) implies that every edge of $\T$ is embedded as the bottom diagonal in precisely one tetrahedron of $\T$.
	We encode a taut structure on a tetrahedron by drawing it as a quadrilateral with two diagonals --- one on top of the other; see Figure~\ref{fig:taut_tet}. Then the convention is that coorientations on all faces point towards the reader. In other words, we view the tetrahedron \emph{from above}.
	
	\begin{figure}[h]
		\includegraphics[scale=1]{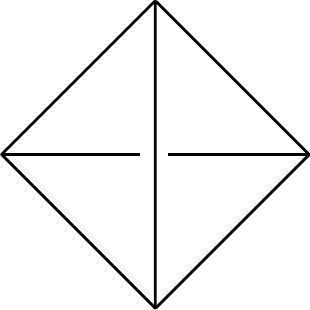}
		\caption{Taut tetrahedron.} 
	\label{fig:taut_tet}
\end{figure}

Let $e$ be an edge of a taut triangulation $(\T, \alpha)$. To every embedding $\epsilon(e)$ of $e$ into a tetrahedron $t$ of $\T$ we assign a dihedral angle 0 or $\pi$ in the following way. If $\epsilon(e)$ is either the top or the bottom diagonal of $t$ we assign to $\epsilon(e)$ angle $\pi$. Otherwise we assign to~$\epsilon(e)$ angle $0$. 
This equips the 2-skeleton of $\T$ with a structure of a branched surface with branch locus equal to the 1-skeleton of $\T$; see Figure \ref{fig:horizontal}. 
We call it the \emph{horizontal branched surface} associated to $\T$, and denote it by $\HH$. 

\begin{figure}[h]
	\includegraphics[scale=0.7]{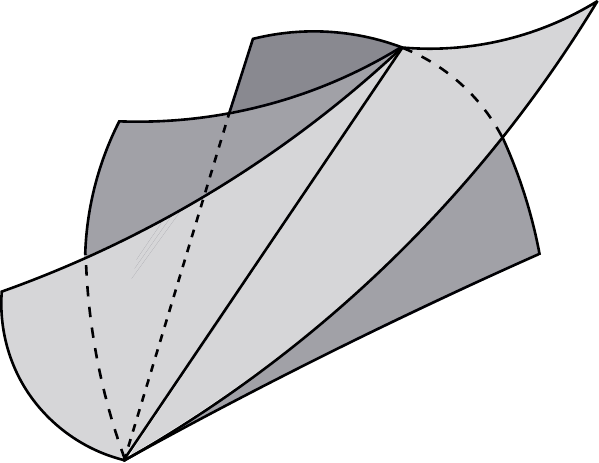}
	\caption{The horizontal branched surface associated to a taut triangulation.}
	\label{fig:horizontal}
\end{figure}

Recall that boundary components of $M$ correspond to links of vertices of $\T$. An ideal vertex of tetrahedron $t$ of $\T$ meets three faces of $t$. Thus an ideal triangulation $\T$ of $M$ determines a triangulation $\partial \T$ of  $\partial M$. If $\T$ is additionally taut, the smoothening of the 2-skeleton $\T^{(2)}$ into the horizontal branched surface determines a smoothening of $\partial \T$ into a train track. We call this train track the \emph{boundary track} of~$\T$ and denote it by $\beta$. 
Suppose that an ideal vertex of $t$ meets faces $f_1, f_2, f_3$ of $t$. Exactly one pair $(f_i, f_j)$, $i\neq j$, is adjacent either along the top or along the bottom diagonal of~$t$. In the construction of the horizontal branched surface of $\T$ we assign to such a pair dihedral angle $\pi$, and to the remaining pairs we assign dihedral angle 0. Thus every complementary region of $\beta$ is a bigon. This has important implications for the topology of $\partial M$. Recall that if $\tau$ is a train track in a surface~$S$ without boundary then the Euler characteristic of $S$ is equal to the sum of indices of all complementary regions of $\tau$ in $S$, where the index of a complementary region $C$ is the quantity
\[\mathrm{index}(C) = 2\chi(C) - \#\text{cusps in $\partial C$}. \] 
It follows that any surface admitting a bigon train track has zero Euler characteristic. Among closed orientable surfaces only torus satisfies this condition. Below we state this observation as a lemma, because we will refer to it in the proof of Proposition~\ref{prop:taut:sufficient:necessary} to show that in some situations the mutant triangulation does not admit a taut structure.

\begin{lemma}\label{lem:taut:torus:boundary}
	Suppose that an oriented 3-manifold $M$ admits a taut ideal triangulation. Then the boundary of $M$ is nonempty and consists of tori.\qed
\end{lemma}

Suppose that $(\T, \alpha)$ is a taut triangulation. Let $-\alpha$ denote the  taut structure on~$\T$ obtained by reversing orientations of all edges of the dual spine $\D$ of $\T$. Taut triangulations $(\T, \alpha)$ and $(\T, -\alpha)$ determine the same dihedral angles between consecutive faces attached to edges of $\T$ and thus the same horizontal branched surface. We call $(\T, \pm \alpha)$ a \emph{taut angle structure} on $\T$.
\subsection{Veering triangulations}\label{subsec:veering}
Taut triangulations are abundant in 3-manifolds \cite[Theorem 1]{Lack_taut}. In contrast, veering triangulations, a subclass of taut triangulations defined below, are very rare. It is conjectured that any hyperbolic 3-manifold with toroidal boundary admits only finitely many, potentially zero. 
\begin{definition}\label{def:veering}
	A \emph{veering structure} on a taut ideal triangulation is a smoothening of its dual spine into a branched surface $\B$ which locally around every vertex looks either as in Figure \ref{fig:veering_branched_surface} (a) or Figure \ref{fig:veering_branched_surface} (b). 
\end{definition} 
\begin{figure}[h]
	\includegraphics[scale=0.6]{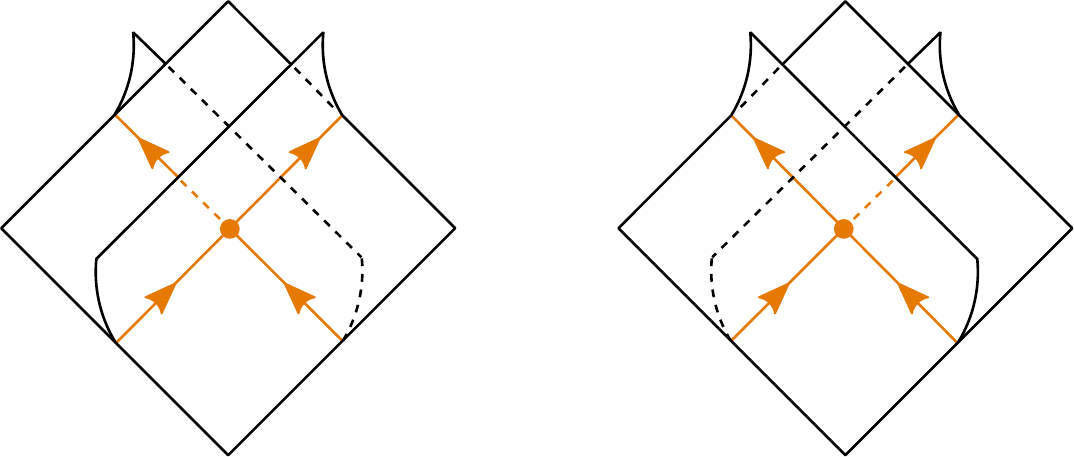}
	\put(-195,0){(a)}
	\put(-90, 0){(b)}
	\caption{The dual spine of a veering triangulation can be smoothened into a branched surface which around every vertex looks either as in (a) or as in~(b). Orientation on the edges of the branch locus is relevant.}
	\label{fig:veering_branched_surface}
\end{figure}

A \emph{veering triangulation} is a taut ideal triangulation with a veering structure. We call the branched surface from Definition \ref{def:veering} the \emph{stable branched surface} of a veering triangulation. We emphasize that its branch locus is oriented by the taut structure on the triangulation. 
\begin{remark}
	In \cite{Tsang-branched-surfaces} a branched surface which locally around every vertex looks like in Figure \ref{fig:veering_branched_surface}, and has only solid tori or solid shells as complementary regions, is called a \emph{veering branched surface}. However, the author of \cite{Tsang-branched-surfaces} orients the branch locus of this branched surface in the opposite direction.
\end{remark} 

From Definition \ref{def:veering} it follows that a veering triangulation is determined by three pieces of combinatorial data:
\begin{enumerate}
	\item ideal triangulation $\T$, 
	\item taut structure $\alpha$; see Definition \ref{def:taut},
	\item veering structure $\B$; see Definition \ref{def:veering}.
\end{enumerate}
For brevity, we typically denote a veering triangulation by a caligraphic letter $\V$, potentially with some sub- or superscript, by which we mean $\V = (\T, \alpha, \B)$.

Definition \ref{def:veering} is `dual' to the classical definition of a veering triangulation; see \cite[Definition 5.1]{SchleimSegLinks}. Below we explain how to translate between the two definitions. When viewing the dual spine of a veering triangulation as a branched surface we call its 2-cells \emph{sectors}.
Every edge $d$ of $\B$ is adjacent to three sectors of $\B$. The structure of a branched surface on $\B$  determines the \emph{one-sheeted side} of $d$ and the \emph{two-sheeted side} of $d$; see Figure \ref{fig:branched_surface_sides}. 
We say that a sector $s$ adjacent to $d$ is \emph{large} relative to $d$ if it is on the one-sheeted side of $d$. Otherwise, we say that $s$ is \emph{small} relative to $d$. Thus, two out of three sectors adjacent to $d$ are small relative to $d$. Since $d$ is oriented, and the manifold underlying~$\V$ is oriented, we can detect in which direction (right/left) each of these small sectors veers. One of them veers to the right of $d$, and the other --- to the left of~$d$. These directions are marked in Figure \ref{fig:branched_surface_sides}. 

\begin{figure}[h]
	\includegraphics[scale=1]{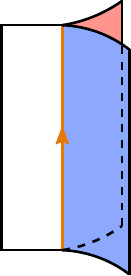}
	\put(-130,35){one-sheeted side}
	\put(20, 35){two-sheeted side}
	\put(5, 0){right}
	\put(5, 75){left}
	\caption{One of the sectors on the two-sheeted side veers to the right, the other veers to the left.} 
\label{fig:branched_surface_sides}
\end{figure}
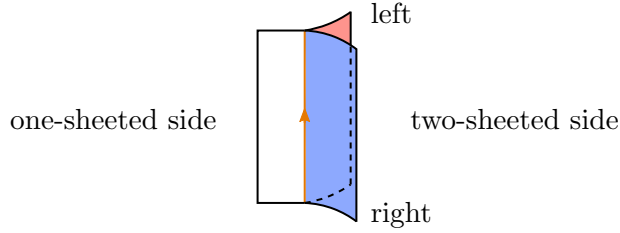

\begin{lemma}\label{lem:stable:branched:surface}
Let $s$ be a sector of the stable branched surface of a veering triangulation. Let $d_1, d_2$ be two consecutive edges of $s$. Let $v$ be the common vertex of $d_1, d_2$.
\begin{enumerate}
	\item If orientations of $d_1, d_2$ both point into $v$, then $s$ is
	large relative to both $d_1$ and~$d_2$. 
	\item If orientations of $d_1, d_2$ both point out of $v$, then $s$ is small relative to both~$d_1$ and $d_2$, and if it veers right (left) of $d_1$, then it veers right (left) of $d_2$. 
	\item If orientation of $d_1$ points into $v$ and orientation of $d_2$ points out of $v$ then either $s$ is small relative to $d_1$ and  large relative to $d_2$, or~$s$ is small relative to both $d_1$ and $d_2$ in which case if it veers right (left) of $d_1$, then it veers right (left) of $d_2$. 
\end{enumerate}
In particular, $s$ has at least four edges.
\end{lemma}
\begin{proof}
The statement of this lemma is a verbalization of the local picture of $\B$ presented in Figure \ref{fig:veering_branched_surface}.
\end{proof}

Lemma \ref{lem:stable:branched:surface} says that if $s$ veers to the right (left) of $d$ then for every other edge~$d'$ of $s$ such that $s$ is small relative to $d'$, $s$ veers to the right (left) of $d'$. Since, by  Lemma~\ref{lem:stable:branched:surface}(2), $s$ is a small relative to at least two of its edges, the \emph{veering direction} of~$s$ is well-defined. We can therefore assign colors, red and blue, to the sectors of $\B$ so that right-veering sectors are colored blue and left-veering sectors are colored red; see Figure \ref{fig:branched_surface_sides}. We call them the \emph{veering colors} on $\B$. Dually, we obtain a coloring on the edges of $\V$.
\begin{corollary}\label{cor:veering:equivalence}
Let $\V$ be a veering triangulation. The veering colors on sectors of the stable branched surface of $\V$ determine colors on edges of $\V$ such that for every tetrahedron $t$ of $\V$ the following two conditions hold.
\begin{itemize}
	\item Let  $e_0, e_1, e_2$ be edges of a top face of $t$, ordered counter-clockwise as viewed from above and so that $e_0$ is the top diagonal of t. Then $e_1$ is red and $e_2$ is blue.
	\item Let  $e_0, e_1, e_2$ be edges of a bottom face of $t$, ordered counter-clockwise as viewed from above and so that $e_0$ is the bottom diagonal of $t$. Then $e_1$ is blue and $e_2$ is red.
	$\pushQED{\qed} 
	\hfill \qedhere
	\popQED$
\end{itemize}
\end{corollary}
The conditions from Corollary \ref{cor:veering:equivalence} are exactly the classical veeringness conditions; see \cite[Definition 5.1]{SchleimSegLinks}. Therefore if a triangulation is veering in the sense of Definition~\ref{def:veering}, then it is veering in the sense of \cite[Definition 5.1]{SchleimSegLinks}. The converse also holds. This can be seen by observing that colors on edges of a veering tetrahedron~$t$ that satisfy conditions listed in Corollary \ref{cor:veering:equivalence} determine how to smoothen the dual spine of $t$ into a branched surface which locally around its vertices looks like the one presented in Figure \ref{fig:veering_branched_surface}: blue edges are dual to right veering sectors, and red edges are dual to left veering sectors; see Figure \ref{fig:branched_surface_sides}.

Let $\B$ be the stable branched surface of a veering triangulation $\V$. For every face $f$ of $\V$ the intersection of $\B$ with $f$ is a train track with one switch $v_f$ in the interior of $f$ and three branches, each joining $v_f$ with the mid-point of an edge of~$f$. The union of all these train tracks in faces of $\V$ gives a train track in the horizontal branched surface of $\V$. We call it the \emph{stable train track} of $\V$, and denote it by $\tau$.
A picture of~$\tau$ restricted to the faces of one veering tetrahedron is presented in Figure~\ref{fig:stable:track}.

\begin{figure}[h]
\begin{center}
	\includegraphics[width=0.35\textwidth]{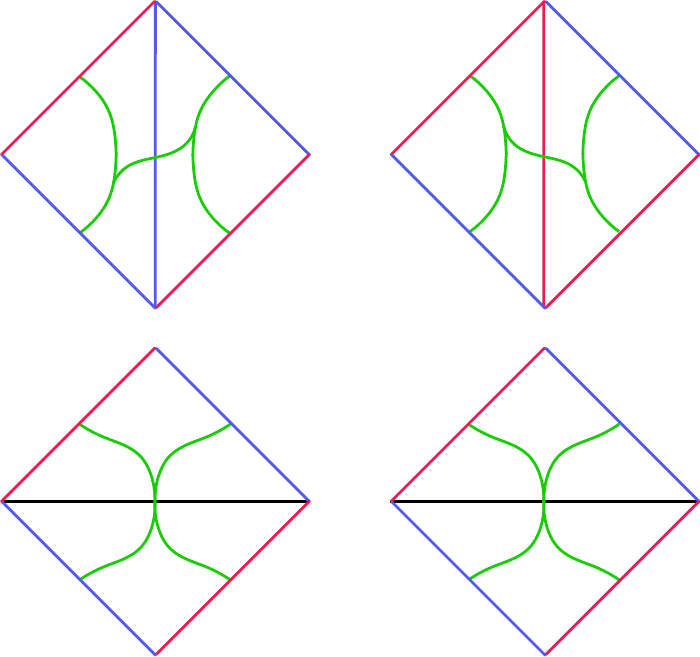} 
\end{center}
\caption{The stable track in the faces of a veering tetrahedron $t$. There are two cases, depending on the structure of the stable branched surface dual to~$t$ (equivalently: depending on the color of the top diagonal of $t$). For each column, the square in the top row represents the top faces of $t$, and the square in the bottom row represents bottom faces of $t$.} 
\label{fig:stable:track}
\end{figure}

Let $\tau_f = \tau\cap f$. It is a trivalent train track with one large branch and two small branches. We say that an edge $e$ of $f$ is the \emph{large edge} of $f$ if it is dual to the large branch of $\tau_f$. Otherwise we say that $e$ is a small edge of $f$. The key property of the stable train track is stated in the following lemma.
\begin{lemma}\label{lem:large:edges}
Let $f$ be a top face of a tetrahedron $t$ of a veering triangulation $\V$. Then the large edge of $f$ is identified with the bottom diagonal of the tetrahedron immediately above $f$.
\end{lemma}
\begin{proof} Let $t'$ be the tetrahedron immediately above $f$. There is a bottom face $f'$ of~$t'$ such that the large edge of $f$ is identified with an edge $e'$ of $f'$. The picture of the stable train track in the bottom faces of a veering tetrahedron (Figure \ref{fig:stable:track}) indicates that~$e'$ must be the bottom diagonal of $t'$.
\end{proof}

The stable branched surface and the stable train track both carry the same combinatorial information. However, it is beneficial to have both these perspectives on the same object, as they have different applications in this paper. In Subsection \ref{subsec:veeringness:of:mutant} we use the stable train track to define a certain subgroup of the group of orientation-preserving combinatorial automorphisms of a surface carried by a veering triangulation. The whole branched surface is more natural to use in the proof of Theorem \ref{thm:mutant:veering} which says that under certain conditions a mutant of a veering triangulation is veering.

The stable branched surface $\B$ can be used to divide all veering triangulations into two classes, depending on whether $\B$ is transversely orientable or not.

\begin{definition}
A veering triangulation $\V$ is \emph{edge-orientable} if its stable branched surface is transversely orientable. Otherwise we say that $\V$ is \emph{not edge-orientable}.
\end{definition}

We refer the reader to \cite{taut_alex} for more information about edge-orientability and how it affects certain polynomial invariants of veering triangulations.

\begin{remark}\label{remark:unstable:branched:surface}
The dual spine of a veering triangulation $\V$ can be smoothened into another branched surface, called  the \emph{unstable branched surface} of $\V$.  We denote it by~$\B^u$; see \cite[Section~6.1]{SchleimSegLinks}. It is also encoded by the colors on edges of~$\V$. If $t$ is a tetrahedron of $\V$ whose bottom diagonal is blue (respectively, red) then $\B^u_t = \B^u\cap t$ is obtained from Figure \ref{fig:veering_branched_surface}(a) (respectively, Figure \ref{fig:veering_branched_surface}(b)) by rotating it by $\pi$ in the plane of the page and then reversing orientations of all edges in the branch locus.
\end{remark}

\begin{remark}
\label{remark:two:veering:tris}
If $M$ admits a veering triangulation $\V = (\T, \alpha, \B)$ then it also admits a veering triangulation $-\V = (\T, -\alpha, -\B^u)$, where $-\alpha$ is obtained from $\alpha$ by reversing orientations of all edges of the dual spine of $\T$, and $-\B^u$ is the unstable branched surface of $\V$ with orientation on the branch locus given by $-\alpha$.
\end{remark}

In Proposition \ref{prop:not:hyperbolic} we will  use the following crucial fact about veering triangulations.
\begin{theorem}\label{thm:HRST}
\emph{(Hodgson-Rubinstein-Segerman-Tillmann, \cite[Theorem 1.5]{veer_strict-angles})}\nopagebreak

\noindent Suppose that $M$ is a compact oriented 3-manifold that admits a veering triangulation. Then the interior of $M$ admits complete hyperbolic metric. \qed
\end{theorem}

\subsubsection{The veering census}\label{sec:veering:signature}
Data on veering triangulations of orientable 3-manifolds consisting of up to 16 tetrahedra is available in the Veering Census \cite{VeeringCensus}. A veering triangulation in the census is described by a string of the form
\begin{equation}\label{string}
\texttt{[isoSig]\underline{ }[taut angle structure]}.
\end{equation}
The first part of this string is the isomorphism signature of the triangulation. It identifies a triangulation uniquely up to combinatorial isomorphism \cite[Section 3]{Burton_isoSig}. The second part of the string records a taut angle structure, that is a taut structure up to reversing orientation of all dual edges. 
A string of the form \eqref{string} is called a \emph{taut signature} and we use it whenever we refer to any particular veering triangulation from the Veering Census. 

The following lemma is well-known since the development of the Veering Census. It explains why there is at most one veering triangulation with a fixed underlying taut ideal triangulation. We include its proof here because we will use it in Subsection \ref{subsec:generalization}.

\begin{lemma}\label{lem:uniqueness:of:veering}
Suppose that $(\T, \alpha)$ is a taut triangulation. If $(\T,\alpha)$ admits a veering structure then this structure  is unique.
\end{lemma}
\begin{proof}
Suppose that there are two veering veering triangulations $\V = (\T, \alpha, \B)$, \mbox{$\V' = (\T, \alpha, \B')$}. If they are distinct then there is a tetrahedron $t$ of $\T$ such that $\B_t = \B\cap t$ and $\B'_t = \B'\cap t$ are different. Let $f$ be a top face of $t$. Let $\tau_f$, $\tau'_f$ be the stable train tracks in $f$ determined by $\B$, $\B'$, respectively. Definition \ref{def:veering} and the assumption that $\B_t\neq\B'_t$ imply that $\tau_f \neq \tau'_f$. In particular, there is an edge $e_1$ of $f$ which is dual to the large branch of $\tau_f$ and a distinct edge $e_2$ of $f$, of a different color than $e_1$,  which is dual to the large branch of $\tau_f'$. Applying Lemma \ref{lem:large:edges} to $\V$ yields that $e_1$ is identified with the bottom diagonal of the tetrahedron immediately above $f$, and applying it to $\V'$ yields that $e_2$ is identified with the bottom diagonal of the tetrahedron immediately above $f$. Since $e_1, e_2$ cannot be identified in $\T$, this is a contradiction to the assumption that the taut ideal triangulations underlying $\V, \V'$ are the same.
\end{proof}

\subsection{Surfaces carried by a taut triangulation} \label{subsec:carried:surfaces}
Let $(\T, \alpha)$ be a taut triangulation of an oriented 3-manifold $M$, with the set $T$ of tetrahedra, the set $F$ of faces, and the set $E$ of edges. Recall that $\alpha$ determines a branched surface structure on $\T^{(2)}$ which we call the horizontal branched surface and denote by $\HH$; see Figure~\ref{fig:horizontal}. 
The 1-skeleton of $\T$ is the branch locus of $\HH$. Thus each (oriented) edge $e$ of $(\T, \alpha)$ determines a~\emph{branch equation} defined as follows. Let $f_{1}, f_{2}, \ldots, f_{k}$ be triangles attached to~$e$ on the left side, ordered from the bottom to the top. Let $f'_1, f'_{2}, \ldots, f'_{l}$ be triangles attached to $e$ on the right side, also ordered from the bottom to the top. 
Then the branch equation determined by $e$ is given by
\begin{equation} \label{eqgn:branch_eqn}
f_{1} + f_{2} + \ldots + f_{k} = f'_1 + f'_2 + \ldots + f'_{l}.
\end{equation}
Let $w = (w_f)_{f\in F}$ be a nonzero, nonnegative, integral solution to the system of branch equations of $(\T, \alpha)$.  We call the number $w_f$ the \emph{weight} of $f$ and $w$ a \emph{weight system} on $(\T, \alpha)$. Using weights of triangles we can define the weight $w_e$ of an edge $e \in E$ as the sum of weights of faces attached to $e$ on one of its sides. For edge satisfying the branch equation \eqref{eqgn:branch_eqn} we have
\[w_e = \sum\limits_{i=1}^k w_{f_i} = \sum\limits_{j=1}^l w_{f_j'}.\]

Equip the triangles of $\T$ with an orientation determined by their coorientation via the right hand rule. Then the (relative) 2-chain \[S_w = \sum\limits_{f \in F} w_f \cdot f.\]
is a 2-cycle giving an oriented surface properly immersed in $M$. It is embedded if and only if $w_e \leq 1$ for every $e \in E$. Let $x \in E \cup F$. If $w_x> 1$  then multiple copies of $x$  are pinched together. Pulling these overlapping regions of $S_w$ slightly apart yields an oriented surface $\Semb$ which is properly embedded in $M$; see Figure \ref{fig:embedded_copy}. We say that $\Semb$ is \emph{carried} by $(\T, \alpha)$.

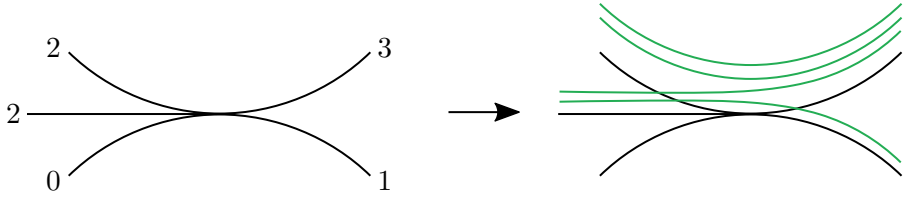
\begin{figure}[h]
\includegraphics[scale=1]{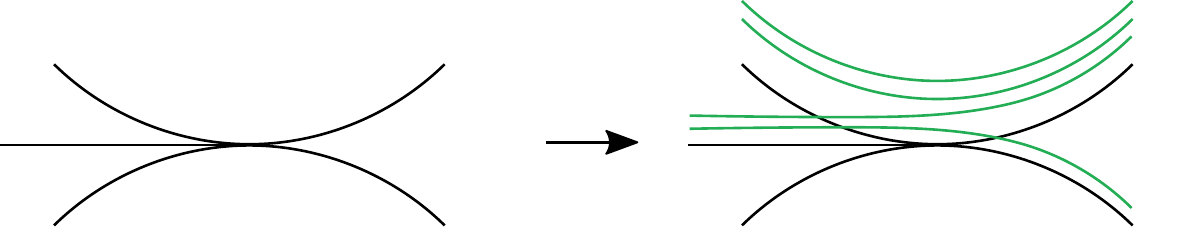}
\put(-335,45){2}
\put(-350,20){2}
\put(-335,-5){0}
\put(-211, 45){3}
\put(-211, -5){1}
\caption{From a solution to branch equations to an embedded surface.}
\label{fig:embedded_copy}
\end{figure}
More generally, we say that a surface $S$ properly embedded in $M$ is \emph{carried} by $(\T, \alpha)$ if there exists a nonzero, nonnegative, integral solution $w = (w_f)_{f\in F}$ to the system of branch equations of $(\T, \alpha)$ such that $S$ is homotopic to  the relative 2-cycle $S_w$.
Note that the same properly embedded surface $S$ can be carried by $(\T, \alpha)$ in multiple different ways.  When we write $S_w$ or $\Semb$ we always mean a surface in a fixed carried position corresponding to the weight system $w$.

If there exists a strictly positive integral solution $w$ to the system of branch equations of $(\T, \alpha)$, we say that $(\T, \alpha)$ is \emph{layered}. 
If there exists a nonnegative, nonzero integral solution, but no strictly positive integral solution, then we say that~$(\T, \alpha)$ is \emph{measurable}. If there is no nonnegative nonzero solution to the system of branch equations of $(\T, \alpha)$ then we say that $(\T,\alpha)$ is \emph{nonmeasurable}.

\subsection{Connection with pseudo-Anosov flows}\label{sec:veering:and:flows}
Recall from the Introduction that veering triangulations are combinatorial tools to study pseudo-Anosov flows. In this subsection we will make this statement more precise.

\begin{definition}\label{defn:pA}
A continuous flow $\Psi: N \times \rr \rightarrow N$ on a closed 3-manifold $N$ is \emph{pseudo-Anosov} if there are 2-dimensional singular foliations $\mathcal{F}^s$, $\mathcal{F}^u$ on $N$ with the following properties:
\begin{itemize}
	\item $\mathcal{F}^s$, $\mathcal{F}^u$ intersect along the flow lines of $\Psi$.
	\item $\Psi$ admits finitely many (potentially zero) isolated closed orbits $\ell_1, \dots, \ell_k$ such that for \mbox{$i=1,2 \ldots k$} in a sufficiently small tubular neighborhood of $\ell_i$ the foliation $\mathcal{F}^{s/u}$ is isotopic to the mapping torus of the $\frac{2\pi m_i}{p_i}$-rotation of the $p_i$-pronged foliation of a disk (with the prong singularity in the center), for some $p_i \geq 3$, $m_i \in \zz$. These orbits are called the \emph{singular orbits} of $\Psi$.
	\item Away from the singular orbits of $\Psi$, foliations $\mathcal{F}^s$, $\mathcal{F}^u$ are nonsingular and transverse to each other. 
	\item Two flow lines contained in the same leaf of $\mathcal{F}^s$ are forward asymptotic, and two flow lines contained in the same leaf of $\mathcal{F}^u$ are backward asymptotic.
\end{itemize}
If $\Psi$ has a dense orbit then we say that $\Psi$ is \emph{transitive}.
If $\Psi$ does not have any singular orbits then it is called an \emph{Anosov flow}. 
\end{definition}

\begin{definition}\label{defn:top:equiv}
Two flows $\Psi, \Psi'$ on $N$ are \emph{topologically equivalent} (or \emph{orbit equivalent}) if there is a homeomorphism $h: N\rightarrow N$ which takes oriented orbits of~$\Psi$  to oriented orbits of~$\Psi'$.
\end{definition}

In the literature there is also a notion of a smooth pseudo-Anosov flow; see \cite[Definition 5.8]{Tsang-Agol}. Recently Shannon proved that any continuous transitive Anosov flow is topologically equivalent to a smooth Anosov flow \cite[Section 5]{Shannon-equivalence}; see Definition \ref{defn:top:equiv}. His methods generalize to transitive continuous pseudo-Anosov flows \cite[Theorem~5.10]{Tsang-Agol}. Thus up to topological equivalence we may assume that our flows are smooth. The reason we prefer the above definition is that it focuses on topological properties of pseudo-Anosov flows which are crucial for our purposes --- namely, the existence of foliations $\mathcal{F}^s$, $\mathcal{F}^u$ with prescribed behavior. These foliations are called the  \emph{stable} and \emph{unstable foliations} of the flow, respectively. Splitting $\mathcal{F}^s$, $\mathcal{F}^u$ open along their singular leaves yields a pair of laminations $\mathcal{L}^s, \mathcal{L}^u$ in $N$, called the \emph{stable} and \emph{unstable laminations} of $\Psi$. Leaves of $\mathcal{L}^{s/u}$ are either open annuli, open M\"obius bands or planes. M\"obius bands appear if and only if $\mathcal{L}^{s/u}$ is not transversely orientable. We will not pay much attention to the parametrization of a flow. In fact, we will consider two flows which are topologically equivalent  to be the same.

The simplest examples of pseudo-Anosov flows are the suspension flows of pseudo-Anosov homeomorphisms of closed surfaces on their mapping tori.
Their stable/unstable foliations are formed by the mapping tori of the the stable/unstable foliations of the monodromy of the fibration. Suspension flows have an additional property: an embedded surface which intersects every flow line with positive sign. Such a surface is called a \emph{cross-section} to the flow. Any flow which admits a cross-section is called \emph{circular}. From the suspension flow of a pseudo-Anosov homeomorphism one can construct infinitely many other pseudo-Anosov flows via the Goodman-Fried surgery \cite{Fried_drill, Goodman}. In particular, Goodman-Fried surgery can be used to construct non-circular pseudo-Anosov flows, including pseudo-Anosov flows on non-fibered 3-manifolds.

Let $\Psi$ be a pseudo-Anosov flow on a closed 3-manifold $N$. Fix a finite collection $\Lambda$  of closed orbits of $\Psi$ which includes all singular orbits of $\Psi$. To be able to construct a veering triangulation of $N \bez \Lambda$ encoding $\Psi$ a technical condition, called \emph{no perfect fits relative to $\Lambda$}, has to be satisfied. We refer the reader to \cite[Definition 5.12]{Tsang-Agol} for a precise definition of this term. In this paper we will work combinatorially with veering triangulations, and deduce appropriate statements concerning flows using the following two theorems. 

\begin{theorem}\label{thm:AG} \emph{(Agol-Gu\'eritaud, unpublished)}

\noindent
Let $\Psi$ be a pseudo-Anosov flow on a closed 3-manifold $N$. Suppose that $\Lambda$  is a finite collection of closed orbits of $\Psi$ which includes all singular orbits of $\Psi$ and such that $\Psi$ has no perfect fits relative to $\Lambda$. Then $N \bez \Lambda$ admits a veering triangulation $\V$ such that the stable (unstable) branched surface of $\V$, when embedded in $N$ via the inclusion $N \bez \Lambda \hookrightarrow N$, fully carries the stable (unstable) lamination of $\Psi$.  \qed 
\end{theorem}

The outline of the proof of the above theorem  when $\Lambda$ consists only of the singular orbits of $\Psi$ and is nonempty is presented in \cite[Section~4]{LMT_flow}. In this case we obtain a  veering triangulation that is canonical for the flow. However, we are interested also in Anosov flows and pseudo-Anosov flows which do have perfect fits relative to their singular orbits. The latter are called \emph{pseudo-Anosov flows with perfect fits}. To construct a veering triangulation encoding the stable/unstable laminations of an Anosov flow or of a pseudo-Anosov flow with perfect fits, we need to add more orbits to the set $\Lambda$. It was proved by Tsang that any (pseudo)-Anosov flow is without perfect fits relative to some~$\Lambda$ containing all singular orbits and one additional orbit \cite[Proposition]{Tsang-Birkhoff}. The choice of this one additional orbit is not canonical though, so if $\Psi$ has perfect fits there is no canonical veering triangulation associated to it. Nonetheless,  even in this case there are  veering triangulations that can be used to study~$\Psi$; it is only the canonicity that is lost. If $\V$ arises from $\Psi$ via the Agol-Gu\'eritaud's construction we will say that $\V$ \emph{encodes} $\Psi$.

Another theorem that we will use in this paper says that one can go also in the other direction: use veering triangulations to construct pseudo-Anosov flows. To state it, we need to mention another connection between a pseudo-Anosov flow $\Psi$ and a veering triangulation $\V$ encoding it. Let $\ell \in \Lambda$ be a $p$-pronged orbit of $\Psi$, where $p \geq 2$ and $p=2$ corresponds to a non-singular orbit. 
If we view the compact core of $N - \Lambda$ inside of~$N$ it is clear that its boundary torus $T_\ell$ around $\ell$ meets the singular leaves of $\mathcal{F}^{s/u}$ along $p$ parallel simple closed curves. We call them the \emph{prong curves} of $\Psi$ in~$T_\ell$.
These curves are encoded by a special family of curves carried by the boundary track of $\V$ (defined in Subsection \ref{subsec:taut} for any taut triangulation). Namely, when a taut triangulation is veering, its boundary track on each boundary torus arranges into a collection of \emph{ladders} separated by \emph{ladderpole curves} \cite[Section 1.2]{Gueritaud_CT}. The ladderpole curves of $\V$ correspond precisely to the prong curves of~$\Psi$.

\begin{theorem}\emph{(Agol-Tsang, \cite[Theorem~5.1]{Tsang-Agol})} \label{thm:AT}\nopagebreak

\noindent
Let $\V$ be a veering triangulation of a \mbox{3-manifold}~$M$. Suppose that $M$ has $k$ boundary components $T_1, \ldots, T_k$. Let $l_i$ be the collection of ladderpole curves of $\V$ on $T_i$, and let $s_i$ be a connected simple closed curve on $T_i$.
If for every $i=1,2, \ldots, k$ the algebraic intersection $p_i = \langle \ell_i, s_i \rangle$ is greater than one, then the Dehn filled manifold $M(s_1, \ldots, s_k)$ admits a transitive pseudo-Anosov flow $\Psi$ with the following properties:
\begin{itemize}
	\item $\Psi$ is without perfect fits relative to a collection $\Lambda = \lbrace \ell_1, \ldots, \ell_k \rbrace$ of closed orbits  isotopic to the cores of the filling solid tori.
	\item The orbit $\ell_i$ is $p_i$-pronged for $i=1,2, \ldots,k$.
	\item The stable (unstable) lamination of $\Psi$ is fully carried by the stable (unstable) branched surface of~$\V$.\qed
\end{itemize}
\end{theorem}

The fact that the stable lamination of $\Psi$ is fully carried by the stable branched surface of the veering triangulation is  not explicitly stated in \cite[Theorem~5.1]{Tsang-Agol} but if follows from  \cite[Proposition 5.13]{Tsang-Agol}. (Note that the authors call the stable branched surface from Definition \ref{def:veering} the unstable branched surface, and orient the edges of its branch locus in the opposite direction). 
If $\V$ is a veering triangulation of $M$ and~$\Psi$ is a pseudo-Anosov flow on some closed Dehn filling $N$ of $M$ constructed by the Agol-Tsang's construction we will say that~$\Psi$ is \emph{built from} $\V$.
\begin{remark}\label{remark:inverse}

In \cite{Tsang-Agol} there is no claim that if one applies the Agol-Gu\'eritaud's construction to a pseudo-Anosov flow $\Psi$ built from $\V$ one gets $\V$ back. However,  Chi Cheuk Tsang recently told the author that this fact can be proved using \cite[Theorem~1.1]{pA-classification}.
There is also a program of Schleimer-Segerman which aims to draw a one-to-one correspondence between pairs (veering triangulation, appropriate Dehn filling slopes) and pseudo-Anosov flows. The outline of this program can be found in \cite[Section 1.2]{SchleimSegLinks}, and the details of individual steps will appear in \cite{SchleimSeg-dynamic-pairs, SchleimSeg-flows}.

To simplify our discussion we will often write about pseudo-Anosov flows and veering triangulations as if the Agol-Gu\'eritaud's construction and the Agol-Tsang's construction were inverses of each other. However, since there is currently no written proof of this fact, in Theorem \ref{thm:many:flows} we will show that certain different veering triangulations of the same manifold  encode topologically inequivalent flows independently of that convenient assumption.
\end{remark}

Let $\V$ be a veering triangulation of $M$ encoding a pseudo-Anosov flow $\Psi$ on some closed  Dehn filling $N$ of $M$. 
If we view $M$ as a cusped 3-manifold $N \bez \Lambda$, then it is naturally equipped with a flow $\Psi_{N\bez \Lambda}$, the restriction of  $\Psi$ to the complement of~$\Lambda$.  The flow $\Psi_{N \bez \Lambda}$ is not pseudo-Anosov: orbits of $\Psi$ that are asymptotic to an element of $\Lambda$ become orbits of $\Psi_{N\bez \Lambda}$ that escape to infinity. Alternatively, we can view $M$ as a manifold obtained from $N$ by \emph{blowing up} elements of $\Lambda$ into toroidal boundary components. There is a notion of the \emph{blown-up flow} $\Psi^\circ$ on $M$ \cite[Section~3.2]{Mosher_manuscipt}. See also \cite[Section 3.6]{Bonatti-surgeries} for the construction of $\Psi^\circ$ in the context of smooth flows. The orbits of $\Psi$ that are asymptotic to an element of $\Lambda$ become orbits of $\Psi^\circ$ that are asymptotic to a prong curve on some boundary component of $M$ (which is an orbit of $\Psi^\circ$). By splitting the stable/unstable foliations of $\Psi$ open along the leaves through each orbit of $\Lambda$ we obtain a pair of laminations  in $M$. We will call these laminations  the stable/unstable laminations of~$\Psi^\circ$, respectively. Theorem \ref{thm:AT} immediately implies the following statement.

\begin{corollary}\label{cor:EO:TO}
Suppose that a pseudo-Anosov flow $\Psi$ is built from a veering triangulation  $\V$ of $M$. The following statements are equivalent.
\begin{itemize}
	\item $\V$ is edge-orientable. 
	\item The stable lamination of~$\Psi$ is transversely orientable. 
	\item The stable lamination of~$\Psi^\circ$ is transversely orientable. \qed
\end{itemize}
\end{corollary}

Given a flow $\Psi$ on an oriented 3-manifold $N$ and an oriented surface $S$ properly embedded in $N$ we will say that $S$ is \emph{transverse} to $\Psi$ if $S$ is transverse to the orbits of~$\Psi$ and the orientation on $TS\oplus T\Psi$ agrees with the orientation of $M$. For instance, if a veering triangulation $\V$ is built from a pseudo-Anosov flow $\Psi$ then all surfaces carried by $\V$ are transverse to $\Psi^\circ$.

\begin{theorem}\label{thm:LMT:transverse}\emph{(Landry-Minsky-Taylor, \cite[Theorem 5.1]{LMT_flow})}

\noindent Suppose that $\V$ is a veering triangulation of $M$ encoding a pseudo-Anosov flow $\Psi$ on some closed Dehn filling $N$ of $M$. Then any surface carried by $\V$ is transverse to the blown-up flow $\Psi^\circ$ on $M$. 
\end{theorem}

\subsection{Veering triangulations and the Thurston norm}\label{veering:vs:thurston}

For a compact, oriented 3-manifold $M$  Thurston defined a  semi-norm $\Thnorm{\cdot}$ on $H_2(M, \partial M;\rr)$  as follows. Every integral class $\eta \in H_2(M, \partial M; \zz)$ can be represented by a properly embedded surface $S \subset M$ \cite[Lemma 1]{Thur_norm}. If~$S$ is connected we set
\begin{equation*}
\chi_-(S) = \max \lbrace 0, -\chi(S) \rbrace,
\end{equation*}
where $\chi(S)$ denotes the Euler characteristic of $S$.
Otherwise, denote by $S_1, S_2, \ldots, S_k$ connected components of $S$ and set 
\begin{equation*}
\chi_-(S) = \sum\limits_{i=1}^{k} \chi_-(S_i).
\end{equation*}
\label{notation:Thurston}
We define a quantity $\Thnorm{\eta}$   as the infimum of $\chi_-(S)$ over all surfaces~$S$ which are properly embedded in $M$ and represent $\eta$. 
The function
$\Thnorm{\cdot}$  can be extended to $H_2(M, \partial M; \mathbb{Q})$ by requiring linearity on each ray through the origin in $H_2(M, \partial M;\rr)$, and then to  $H_2(M, \partial M; \rr)$  by requiring continuity \cite[Section 1]{Thur_norm}.  
$\Thnorm{\cdot}$  is a norm  if we assume that every surface representing a nonzero class in $H_2(M, \partial M;\zz)$ has negative Euler characteristic \cite[Theorem 1]{Thur_norm}. 	Then $\Thnorm{\eta}$ is called the \emph{Thurston norm} of $\eta$. 
If a properly embedded surface $S$ satisfies  $\chi(S) = -\Thnorm{\eta}$ then it is called a \emph{taut representative} of $\eta$ or a \emph{Thurston norm minimizing representative} of~$\eta$.
The unit norm ball $\mathbb{B}_{\mathrm{Th}}$ of 	$\Thnorm{\cdot}$ is a polytope with rational vertices \cite[Theorem 2]{Thur_norm}. Thus we can speak about \emph{faces} of the Thurston norm ball.

A connection between the Thurston norm and pseudo-Anosov flows on closed hyperbolic 3-manifolds was established by Fried and Mosher in the 80's and 90's. The first result in that direction concerned only fibered faces.

\begin{theorem} \emph{(Fried, \cite[Theorem~7]{Fried_suspension})}  \label{thm:F} \nopagebreak

\noindent Let $N$ be a closed hyperbolic 3-manifold. Let $\face{F}$ be a fibered face of the Thurston norm ball in $H_2(N;\rr)$. There is a unique, up to isotopy and reparametrization, circular pseudo-Anosov flow~$\Psi$  such that a class $\eta \in H_2(N;\zz)$ can be represented by a cross-section to $\Psi$ if and only if $\eta$ is in the interior of $\cone{\face{F}}$.\qed 
\end{theorem}

The importance of this result lies in the fact that when $b_1(N)>1$  there are infinitely many fibrations lying over $\face{F}$ and thus one can construct infinitely many suspension flows on $N$: one for each fibration. Theorem \ref{thm:F} implies that all these flows are the same up to isotopy and parametrization. We will say that the unique flow $\Psi$ associated to a fibered face $\face{F}$ \emph{dynamically represents} $\face{F}$.

Mosher extended Fried's result by showing that if a circular flow $\Psi$ dynamically represents a fibered face $\face{F}$ then $\eta \in H_2(N; \zz)$ can be represented by a surface that is \emph{almost transverse} to  $\Psi$ if and only if $\eta$ is  in $\cone{\face{F}}$ \cite[Theorem~1.4]{Mosher_branched}. 
Almost transversality means that the surface is transverse to a slightly modified flow $\Psi^{\#}$ obtained by \emph{dynamically blowing up} finitely many closed orbits of~$\Psi$ into a collection of annuli; see \cite[Section~1.3]{Mosher-dynamical1992} for details. If $S$ is almost transverse to~$\Psi$ then the algebraic intersection of $\lbrack S \rbrack$ with the homology class $\lbrack \gamma \rbrack$ of every closed orbit $\gamma$ of~$\Psi$ is nonnegative. Conversely, if $\eta \in H_2(M;\zz)$ is such that $\langle \eta, \lbrack \gamma \rbrack \rangle \geq 0$ for every closed orbit~$\gamma$ of $\Psi$ then~$\eta$ can be represented by a taut surface which is almost transverse to $\Psi$ \cite[Theorem~1.3.2]{Mosher-dynamical1992}.  Using these facts, Mosher extended the notion of  dynamical representation of faces of the Thurston norm ball to non-fibered faces \cite{Mosher-dynamical1992}.
Given a  pseudo-Anosov flow $\Psi$ on a closed 3-manifold $N$ let  $\C(\Psi) \subset H_2(N;\rr)$ be the non-negative span of second homology classes whose algebraic intersection with homology class of every closed orbit of $\Psi$ is nonnegative. 
	By the aforementioned result  \cite[Theorem~1.3.2]{Mosher-dynamical1992}, we can  think of $\C(\Psi)$ as the cone of homology classes of surfaces that are almost transverse to~$\Psi$.

\begin{definition}\label{defn:dynamical_rep} We say that a pseudo-Anosov flow $\Psi$ on a closed hyperbolic \mbox{3-manifold}~$N$ \emph{dynamically represents} a face $\face{F}$ of the Thurston norm ball in $H_2(N;\rr)$ if $\C(\Psi) = \cone{\face{F}}$.\end{definition}
Definition \ref{defn:dynamical_rep} differs slightly from Mosher's original definition \cite[p. 244]{Mosher-dynamical1992}. He required equality  $\face{F}= \C(\Psi)\cap \chi_{\Psi}^{-1}(-1)$, where $\chi_{\Psi}$ denotes the  Euler class of the normal plane bundle to $\Psi$. We simplified the definition because the equality $\Thnorm{\eta} = -\chi_{\Psi}(\eta)$ always holds for $\eta \in \C(\Psi)$ \cite[p. 262]{Mosher-dynamical1992}.

By our earlier discussion, the circular flow~$\Psi$ associated to a fibered face $\face{F}$ as in Theorem~\ref{thm:F} dynamically represents $\face{F}$. Furthermore, Mosher  found sufficient conditions on a non-circular pseudo-Anosov flow to dynamically represent a non-fibered face of the Thurston norm ball in \cite[Theorem 2.7]{Mosher-dynamical1992}. In \cite[Section 4]{Mosher-dynamical1992} he presented an example of a non-circular pseudo-Anosov flow which dynamically represents a top-dimensional non-fibered face of the Thurston norm ball, as well as an example of a pseudo-Anosov flow which does not dynamically represent any face of the Thurston norm ball. In the latter case $\C(\Psi)$ is properly contained in the cone on some face the Thurston norm ball \cite[Theorem 2.8]{Mosher_branched}.
The results of Mosher raise the following two questions.
\begin{question}
Let $N$ be a closed 3-manifold. Given a non-fibered face $\face{F}$ of the Thurston norm ball in $H_2(N;\rr)$, is there a pseudo-Anosov flow $\Psi$ on $N$ which dynamically represents $\face{F}$?
\end{question}
\begin{question}
Suppose that a non-fibered face $\face{F}$ of the Thurston norm ball is dynamically represented by $\Psi$. Is the flow $\Psi$ unique, up to isotopy and reparametrization?
\end{question}

Question 1 is still open. However, using veering triangulations and their connection with pseudo-Anosov flows (outlined in Subsection \ref{sec:veering:and:flows}) it is possible to find examples of distinct flows representing the same face of the Thurston norm ball; see Section \ref{sec:distinct:flows}.

Once we know that a face of the Thurston norm ball can be represented by multiple distinct flows, we may ask another question.

\begin{question}
Suppose that a face $\face{F}$ of the Thurston norm ball is dynamically represented by two topologically inequivalent flows $\Psi, \Psi'$. How are $\Psi, \Psi'$ related?
\end{question}
This problem also can be approached by employing veering triangulations. We partially answer this question in the case of non-fibered faces of manifolds with nonempty boundary in Section \ref{sec:distinct:flows}.

Both Fried and Mosher worked in the setup of closed 3-manifolds. However, the Thurston norm can be defined for any compact oriented atoroidal 3-manifold $M$. Thus we can ask Questions 1, 2, 3 also in the context of faces of the Thurston norm ball in $H_2(M, \partial M;\rr)$ when $\partial M \neq \emptyset$. 
Since we will work with veering triangulations, considering 3-manifolds with $\partial M \neq \emptyset$ is in fact more natural, and is a necessary intermediate step when trying to answer the questions in the closed case.  Instead of pseudo-Anosov flows we then consider blown-up pseudo-Anosov flows.  As in the closed case, if $\Psi^\circ$ is a blown-up pseudo-Anosov flow on $M$ we denote by $\C(\Psi^\circ)$ the nonnegative span of integral homology classes in $H_2(M, \partial M;\rr)$ whose algebraic intersection with  closed orbits  $\Psi^\circ$ is nonnegative.
Then we say that $\Psi^\circ$ dynamically represents a face $\face{F}$ of the Thurston norm ball in $H_2(M, \partial M;\rr)$ if $\C(\Psi^\circ) = \cone{~\face{F}}$.  
The fact that suspension flows of pseudo-Anosov homeomorphisms of surfaces with boundary dynamically represent fibered faces of the Thurston norm ball of their mapping tori (i.e. an analogue of Fried's Theorem \ref{thm:F}) was proved by Landry in \cite[Theorem 3.5]{Landry_stable}.

We will rely on results of Landry-Minsky-Taylor, stated below, connecting the Thurston norm directly with veering triangulations. Recall from Subsection \ref{subsec:carried:surfaces} that a veering triangulation $\V$ of~$M$ may carry surfaces properly embedded in $M$. Each such surface is a Thurston norm minimizing representative of its homology class \cite[Theorem 3]{Lack_taut}, and is transverse to the flow encoded by $\V$ \cite[Theorem 5.1 (stated here as Theorem \ref{thm:LMT:transverse})]{LMT_flow}.
In analogy to the cone $\C(\Psi^\circ)$ of homology classes of surfaces almost transverse to a flow $\Psi^\circ$ we  define the cone $\C(\V)$ of homology classes of surfaces carried by $\V$. We then say that $\V$ \emph{combinatorially represents} a face $\face{F}$ of the Thurston norm ball in $H_2(M, \partial M;\rr)$ if $\C(\V) = \cone{\face{F}}$. 

\begin{theorem}\emph{(Landry-Minsky-Taylor, \cite[Theorems 5.12 and 5.15]{LMT})}\label{thm:LMT:faces}

\noindent If $\V$ is a layered or measurable veering triangulation of $M$, then there is a (not necessarily top-dimensional) face $\face{F}$ of the Thurston norm ball in $H_2(M, \partial M;\rr)$ with $\C(\V) = \cone(\face{F})$. Furthermore, $\face{F}$ is fibered if and only if $\V$ is layered.
\end{theorem}

The above theorem says that layered and measurable veering triangulations always combinatorially represent some face of the Thurston norm ball. This is in contrast with pseudo-Anosov flows for which it is possible that $\C(\Psi)$ is nonempty, but is a proper subset of the cone on some face of the Thurston norm ball \cite[Section 4]{Mosher-dynamical1992}. 
It is also known that if a fibered face is combinatorially represented by a veering triangulation, then this veering triangulation is unique \cite[Proposition 2.7]{MinskyTaylor}. Note that not every fibered face is combinatorially represented by some veering triangulation, because the circular flow associated to the face might have singular orbits. 

\begin{remark}\label{remark:same:carried:surfaces}
In the proof of \cite[Theorem 5.12]{LMT} the authors show that if $\cone{F} = \C(\V)$ then  $\V$ not only carries \emph{some} taut representative of every $\eta \in \cone{F}$, but it carries \emph{every} taut representative of $\eta$. Thus in particular if $\C(\V) = \C(\V')$ for some veering triangulations $\V,\V'$ of a fixed manifold, then $\V, \V'$ carry the same surfaces (although possibly in different number of ways).
\end{remark}

\begin{remark}\label{remark:faces:and:automorphisms}
For a taut triangulation $(\T, \alpha)$ let $\Aut^+(\T |  \alpha)$ denote the group of orientation-preserving combinatorial automorphism of $\T$ which preserve $\alpha$. Recall from Subsection \ref{sec:veering:signature} that in the Veering Census veering triangulations $\V = (\T, \alpha, \B)$ and $\phi(\V) = (\phi(\T), \phi(\alpha), \phi(\B))$ have the same taut signature. Thus the same entry in the Veering Census may encode multiple veering triangulations representing different faces of the Thurston norm ball which lie in the same orbit of the action of $\mathrm{Homeo}^+(M)$ on $H_2(M, \partial M;\rr)$. Furthermore, veering triangulations $\V, -\V$ also have the same taut signature. They satisfy $\C(\V) = -\C(\V)$ and represent a pair of opposite faces of the Thurston norm ball. 
\end{remark}

Landry, Minsky, and Taylor defined a \emph{flow graph} of a veering triangulation $\V$ whose oriented cycles correspond to orbits of the flow encoded by $\V$. We refer the reader to \cite[Section 4]{LMT} for the definition of the flow graph, and to \cite[Section 6]{LMT_flow} for an explanation of the relationship between the flow graph of $\V$ and orbits of the flow encoded by $\V$. From their results it follows that the blown-up pseudo-Anosov flow on the 3-manifold underlying a layered or measurable veering triangulation dynamically encodes a face of the Thurston norm ball.

\begin{theorem}\label{thm:LMT:cones}
\emph{(Landry-Minsky-Taylor \cite[Theorem 6.1]{LMT_flow} and \cite[Theorem 5.1]{LMT})}\nopagebreak

\noindent Let $\V$ be a veering triangulation of $M$. Suppose that $\V$ encodes a pseudo-Anosov flow~$\Psi$ on some closed Dehn filling $N$ of $M$. Let $\Psi^\circ$ be the associated blown-up flow on $M$. Then
\[\C(\Psi^\circ) = \C(\V).\]
\end{theorem}
Under additional assumptions, we also have an analogous theorem concerning the pseudo-Anosov flow $\Psi$ on $N$. There is a (potentially empty) subcone $\C(\V  |  N) \subset \C(\V) \subset H_2(M, \partial M;\rr)$  of homology classes of surfaces carried by $\V$ whose boundary components have slopes consistent with the Dehn filling slopes yielding $N$ out of $M$. These surfaces cap off to embedded surfaces in $N$. We denote by $\C_N(\V) \subset H_2(N;\rr)$ the nonnegative span of homology classes of these capped off surfaces. 

\begin{theorem}\label{thm:L:cones}
\emph{(Landry \cite[Theorem A]{Landry_homology_isotopy} and LMT \cite[Theorem 6.1]{LMT_flow})}

\noindent Let $\V$ be a veering triangulation of $M$. 
Suppose that $\V$ encodes a pseudo-Anosov flow~$\Psi$ on some closed Dehn filling $N$ of $M$ such that the core curves of the filling solid tori are singular orbits of $\Psi$ with at least 3 prongs. Then $N$ is hyperbolic. Furthermore, if $\C_N(\V)\neq \emptyset$, there is a face $\face{F}$ of the Thurston norm ball in $H_2(N;\rr)$  such that $\cone{F} =\C_N(\V) = \C(\Psi)$.
\end{theorem}

The face $\face{F}$ is determined by the Euler class $\chi_{\Psi}$ of the normal plane bundle to~$\Psi$ in the sense that $\chi_\Psi(\eta) = \Thnorm{\eta}$ for every $\eta \in \cone{F}$. Mosher's example of a pseudo-Anosov flow $\Psi$ on a closed hyperbolic 3-manifold~$N$ which does not dynamically represent a whole face of the Thurston norm ball in $H_2(N;\rr)$ is such that there is a class \mbox{$\eta \in H_2(N;\rr)$} for which $\chi_\Psi(\eta) = \Thnorm{\eta}$, but which pairs negatively with the homology class of some nonsingular closed orbit of $\Psi$ \cite[Section~4]{Mosher-dynamical1992}. In this case $\C(\Psi)$ is a proper subset of the cone on some face of the Thurston norm ball in $H_2(N;\rr)$. This `pathology' does not happen under the assumptions of Theorem \ref{thm:L:cones} because of Theorem \ref{thm:LMT:cones} and the fact that if $\langle \eta, \lbrack \gamma\rbrack \rangle <0$ for a singular orbit $\gamma$ with at least 3 prongs then $\chi_\Psi(\eta) < \Thnorm{\eta}$; see \cite[pp. 259 -- 261]{Mosher-dynamical1992}.

\section{Mutations of veering triangulations}
\label{sec:mutations}

For the remainder of this paper by $X|Y$ we denote the metric completion of $X \bez Y$ with respect to the path metric on $X  \bez Y$ induced from $X$.

Let $S$ be an oriented surface properly embedded in an oriented compact \mbox{3-manifold}~$M$. We say that $M|S$  is the \emph{cut manifold} obtained from $M$ by \emph{decomposing} it along $S$. The orientation on $S$ determines a transverse orientation on $S$ via the right hand rule.  By $S^+$ we denote the boundary copy of $S$ in $M|S$ which is cooriented out of $M|S$ and by $S^-$ we denote the boundary copy of $S$ in $M|S$ which is cooriented into $M|S$. 
Any homeomorphism $\varphi: S^+ \rightarrow S^-$ gives rise to a \emph{mutant manifold} $M^\varphi$ obtained from $M|S$ by gluing $S^+$ to~$S^-$ via $\varphi$. This manifold admits an embedded surface $S^\varphi$ homeomorphic to $S$. We say that $M^\varphi$ is obtained from $M$ by \emph{mutating} it \emph{along} $S$ via~$\varphi$.
We also say that $S$ is a \emph{mutating surface}.

In this section we approach the following problem: assuming that $M$ admits a veering triangulation~$\V = (\T, \alpha, \B)$, when does the mutant manifold~$M^\varphi$ admits a veering triangulation? We restrict our considerations to the case when the mutating surface is carried by~$\mathcal{V}$ with weights $w = (w_f)_{f \in F}$ and the homeomorphism $\varphi$ comes from an orientation-preserving combinatorial automorphism~$\varphi$ of the ideal triangulation $\QVw$ of~$S_w$; we discuss this triangulation in detail in Subsection \ref{subsec:surface:tri}.  
In Subsections \ref{subsec:isomorphism} and~\ref{subsec:sutured} we recall the notions of combinatorial isomorphisms of triangulations and sutured manifolds, respectively. Subsection  \ref{subsec:cut:tri:versus:cut:manifold} is devoted to analyzing a certain \emph{cut triangulation} $\T|F_w$ and its relation to the cut manifold $M|\Semb$. In Subsection \ref{subsec:mutating}, given a pair  $(S_w, \varphi)$, we define a \emph{mutant triangulation}~$\T^\varphi$ of $\T$.  
In Subsection \ref{subsec:mutant:homeomorphism} we find a sufficient and necessary condition on $\varphi$ so that $\T^\varphi$ is an ideal triangulation of $M^\varphi$. In Section \ref{subsec:veeringness:of:mutant} we find a sufficient and necessary condition for the existence of a taut structure $\alpha^\varphi$ on $\T^\varphi$ and furthermore conditions which ensure that $(\T^\varphi, \alpha^\varphi)$ admits a veering structure $\B^\varphi$. In Subsection \ref{subsec:generalization} we generalize this result to give sufficient and necessary conditions on $(\T^\varphi, \alpha^\varphi)$ to admit a veering structure. We also generalize a \emph{veering mutation} $(\T, \alpha, \B) \rightsquigarrow (\T^\varphi, \alpha^\varphi, B^\varphi)$ to a \emph{veering mutation with insertion}.

\subsection{Triangulations: face identifications and combinatorial automorphisms}\label{subsec:isomorphism}

We will use a taut ideal triangulation $(\T, \alpha)$ of a 3-manifold to construct another ideal triangulation of a, typically different, 3-manifold using a combinatorial automorphism of a surface carried by $(\T, \alpha)$. In this subsection we explain the standard conventions used to encode triangulations and their combinatorial automorphisms.

Recall from Section \ref{sec:veering} that by an ideal triangulation of a compact 3-manifold~$M$  we mean an expression of $M \bez \partial M$ as a collection of finitely many ideal tetrahedra with triangular faces identified in pairs by homeomorphisms which send vertices to vertices. Such an identification between a face $f$ of tetrahedron $t$ and a face $f'$ of tetrahedron $t'$ can be encoded by a bijection between their vertices. In our construction we will `forget' identifications between a subset of faces of the triangulation, and replace them with different ones. The whole procedure will be governed by a combinatorial automorphism of a carried surface.

Suppose that $S_w$ is a surface carried by $(\T, \alpha)$ as in Subsection \ref{subsec:carried:surfaces}. Since $S_w$ is built out of triangles and edges of $\T$, it inherits an ideal triangulation from $\T$. We discuss this triangulation in detail in Subsection \ref{subsec:surface:tri}. For now, we will denote this triangulation of $S_w$ by $\Q$. Similarly as in the case of a 3-dimensional triangulation, $\Q$ is an expression of $S_w \bez \partial S_w$ as a collection of finitely many ideal triangles with edges identified in pairs by homeomorphisms which send vertices to vertices. 

Below we recall the definition of a combinatorial isomorphism between two \mbox{2-dimensional} triangulations.

\begin{definition}\label{def:comb:aut}
Let $\Q_1$, $\Q_2$ be finite (ideal) 2-dimensional triangulations. For $i=1,2$ let $F_i, E_i$ denote the set of triangles and edges of~$\Q_i$, respectively. A \emph{combinatorial isomorphism} from $\Q_1$ to $\Q_2$ consists of 
\begin{itemize}
	\item a bijection $\varphi: F_1 \rightarrow F_2$,
	\item for each $f \in F_1$ a bijection $\varphi_f$ between the edges of $f$ and edges of $\varphi(f)$ such that if edges $e$ of $f$ and $e'$ of $f'$ are identified in $\Q_1$ then edges $\varphi_f(e)$ of $\varphi(f)$ and $\varphi_{f'}(e')$ of $\varphi(f')$ are identified in $\Q_2$.
\end{itemize}
\end{definition}

If for every triangle of $\Q_1$ and $\Q_2$ we fix a bijection between its edges and vertices we can view $\varphi_f$ as a bijection between vertices $f$ and vertices of $f'$. It is standard to fix this bijection so that a vertex $v$ of $f$ is associated to the edge of $f$ opposite to $v$. 


Different combinatorial isomorphisms from $\Q_1$ to $\Q_2$ may have the same bijection \mbox{$\varphi: F_1\rightarrow F_2$}. Nonetheless, for simplicity we often abuse the notation and denote a combinatorial isomorphism  by $\varphi: \mathcal{Q}_1 \rightarrow \mathcal{Q}_2$, understanding that it carries information about both a bijection $\varphi: F_1 \rightarrow F_2$ and bijections $\lbrace \varphi_f \ | \ f \in F_1 \rbrace$. 

If $\Q$ is an ideal triangulation then a combinatorial isomorphism $\varphi: \Q \rightarrow \Q$ is called a \emph{combinatorial automorphism} of $\Q$.
We denote the group of orientation-preserving combinatorial automorphism of $\Q$ by $\Aut^+(\Q)$. It follows directly from Definition~\ref{def:comb:aut} that $\Aut^+(\Q)$ is finite.

In Section \ref{sec:faces} we will sometimes mention that two 3-dimensional ideal triangulations are, or are not, combinatorially isomorphic. To understand what it means it suffices to replace in Definition \ref{def:comb:aut} triangles by tetrahedra and edges by triangular faces.
\subsection{Triangulation of a surface carried by a veering triangulation}\label{subsec:surface:tri}

Let $\V = (\T, \alpha, \B)$ be a finite veering triangulation of $M$ with the set $T$ of tetrahedra, the set $F$ of 2-dimensional faces, and the set $E$ of edges. 
Let $w = (w_f)_{f\in F}$ be a weight system on $(\T, \alpha)$. We denote the triangulation of $S_w$ inherited from $\T$ by $\QVw$. A properly embedded surface $\Semb$ obtained by slightly pulling apart overlapping regions of $S_w$ also can be seen as triangulated by $\QVw$. Recall that each triangle of $\V$ is equipped with a trivalent train track; see Figure \ref{fig:stable:track}. Therefore the surfaces $S_w$, $\Semb$ come equipped with a train track dual to their triangulation $\QVw$. We will denote this train track by $\tau_{\V, w}$ and call it the \emph{stable train track} of $S_w$ or $\Semb$.

Let $F_w = \lbrace f \in F \ | \ w_f > 0 \rbrace$ and $E_w = \lbrace e \in E \ | \ w_e > 0 \rbrace$. Given $f \in F_w$ there are~$w_f$ copies of~$f$ in the triangulation $\QVw$, and given $e \in E_w$ there are $w_e$ copies of $e$ in~$\QVw$. Coorientation on the faces of $\mathcal{V}$ and a fixed embedding of $\Semb$ in~$M$ determines a linear order on the copies of a given simplex of $\V$ in  $\QVw$: from the lowermost to the uppermost; see Figure \ref{fig:order}. Denote by $F_{\V,w}$, $E_{\V,w}$  the set of triangles and edges of $\QVw$, respectively. We define maps
\[
L : F_w \cup E_w \rightarrow F_{\V,w} \cup E_{\V,w} \hspace{1cm} 	U : F_w \cup E_w \rightarrow F_{\V,w} \cup E_{\V,w} \]
such that given a simplex $x \in F_w\cup E_w$ the simplex $L(x)$ denotes the \emph{lowermost copy} of $x$ in $\QVw$ and 	$U(x)$ denotes the \emph{uppermost copy} of $x$ in $\QVw$. 
Observe that both $L$ and $U$ are injective. To simplify notation we will sometimes denote $L(x)$ by $\check{x}$. 
Given $f\in F_w$ we will denote by $\sigma_f^L$, $\sigma_f^U$ the bijections between the edges of $f$ and the corresponding edges of $L(f)$, $U(f)$, respectively.
Furthermore we define
\[\ab: \left(F_{\V,w} \cup E_{\V,w}\right) \bez U(F_w \cup E_w) \rightarrow \left(F_{\V,w} \cup E_{\V,w}\right) \bez L(F_w \cup E_w) \]
such that if $y \notin U(F_w \cup E_w)$ then $\ab(y)$ is the triangle of $\QVw$ which is a copy of the same simplex of $\V$ as $y$ and lies \emph{immediately above} $y$ in $M$; see Figure \ref{fig:order}. Simplex $\ab(y)$  exists by the assumption that $y \notin U(F_w \cup E_w)$ (not uppermost), and is never in $L(F_w \cup E_w)$, because lowermost copies do not have copies of the same simplex below them. 	If $w_x = k \geq 0$ then $\ab^0(\check{x})$, $\ab(\check{x}), \ldots, \ab^{k-1}(\check{x})$ are defined, where by $\ab^0(\check{x})$ we mean~$\check{x}$.

\begin{figure}[h]
\includegraphics[scale=0.75]{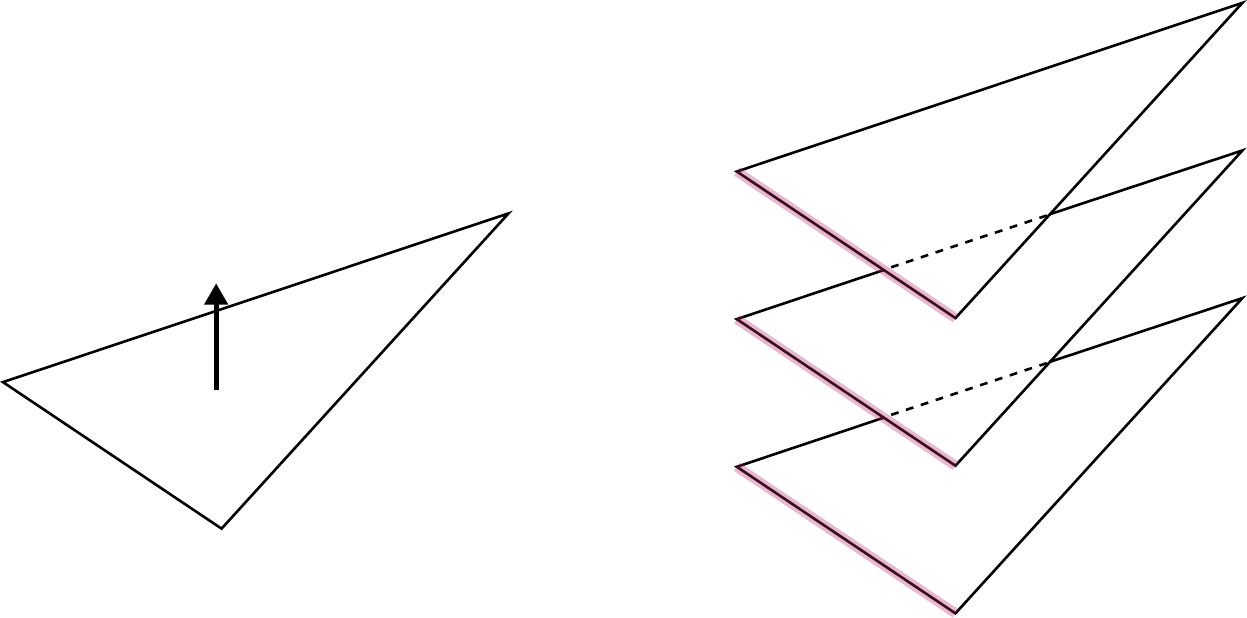}
\footnotesize
\normalsize
\put(-310, 20){$w_f = 3$}
\put(-190, 45){$f$}
\put(-10, 50){$\check{f} := L(f)$}
\put(-10, 82){$\ab(\check{f})$}
\put(-10, 114){$\ab^2(\check{f}) = U(f)$}
\caption{Face $f$ with weight 3. Coorientation on $f$ determines an order on the copies of $f$ in $\mathcal{Q}_{\V,w}$. Shaded edges are $\ab^k(\check{e})$, $\ab^{k+1}(\check{e})$, $\ab^{k+2}(\check{e})$, respectively (from the bottom), for an edge $e$ of $f$ and some $0\leq k\leq w_e-1$.}
\label{fig:order}
\end{figure}
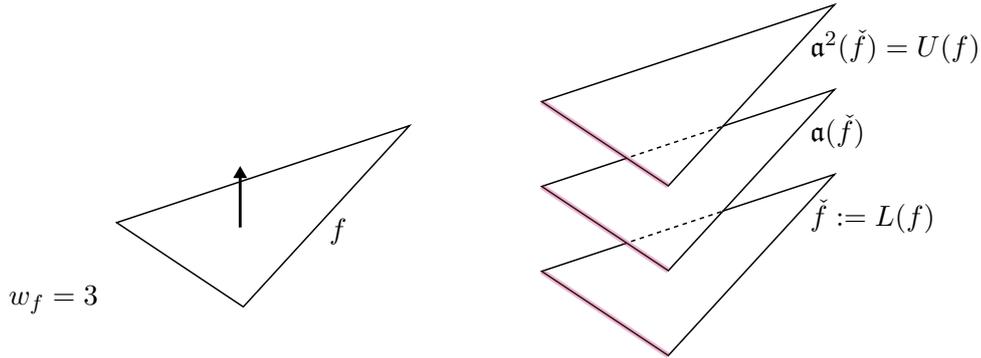


\subsection{Sutured manifolds}\label{subsec:sutured}
If $S$ is an oriented surface properly embedded in an oriented 3-manifold with empty or toroidal boundary then the cut manifold $M|S$ is an example of a \emph{sutured manifold}, defined below.

\begin{definition}\cite[Definition 3.1]{Gabai-sutured}
A \emph{sutured manifold} $(N, \gamma)$ is a compact oriented 3-manifold $N$ together with a set $\gamma \subset \partial N$ of pairwise disjoint annuli $A(\gamma)$, called \emph{sutured annuli}, and tori~$T(\gamma)$, called \emph{sutured tori},  such that
\begin{itemize}
	\item the interior of each component of $A(\gamma)$ contains a homologically nontrivial oriented simple closed curve called a \emph{suture}, 
	\item every connected component of $R(\gamma) = \partial N \bez \mathrm{int}(\gamma)$ is oriented so that every connected component  of $\partial R(\gamma)$ when equipped with the boundary orientation represents the same homology class in $H_1(\gamma)$ as some suture.
\end{itemize}
A fixed orientation of $(N, \gamma)$  endows $R(\gamma)$ with coorientation. This determines a decomposition of of $R(\gamma)$ into $R^+(\gamma)$, where the coorientation points out of $N$, and $R^-(\gamma)$ where the coorientation points into $N$. We call $R^+(\gamma)$ the \emph{top boundary} of the sutured manifold $N$, and $R^-(\gamma)$ its \emph{bottom boundary}. We  also denote them by $\partial^+N$, $\partial^- N$, respectively. A boundary component of a sutured annulus $A$ of $(N, \gamma)$ that is contained in $\partial^+ N$ (respectively, $\partial^- N$) is called its \emph{top} (respectively, \emph{bottom}) \emph{boundary} and denoted by $\partial^+ A$ (respectively, $\partial^- A$).
\end{definition}

The pair $(M|S, \partial M|\partial S)$ is an example of a  sutured manifold. Its sutured tori correspond to the boundary tori of $M$ that are disjoint from $S$. A boundary torus of $M$ containing $k$ boundary components of $S$ gives rise to $k$ sutured annuli in ~$(M|S, \partial M|\partial S)$. For brevity, we often say that $M|S$ is a sutured manifold, without explicitly indicating its sutured tori and annuli.

\subsection{Cutting veering triangulations along carried surfaces}\label{subsec:cut:tri:versus:cut:manifold}
Let $\mathcal{V}$ be a finite veering triangulation of a 3-manifold $M$ with the set $T$ of tetrahedra, the set $F$ of triangular faces and the set $E$ of edges. Let $S_w$ be a surface carried by $\mathcal{V}$ with weights $(w_f)_{f \in F}$.  
As in Subsection \ref{subsec:surface:tri}, we denote by $F_w$ the set of $f\in F$ for which $w_f>0$.  

Recall from Section \ref{subsec:veering} that the caligraphic letter $\V$ implicitly denotes three pieces of combinatorial data: an ideal triangulation $\T$, a taut structure $\alpha$, and a veering structure $\B$. We denote the result of decomposing $\T$ along $F_w$ by $\T|F_w$. All faces of $\T|F_w$ inherit coorientations from $\V = (\T, \alpha, \B)$. We denote this choice of coorientations on the faces of $\T|F_w$ by  $\alpha|F_w$. By $F_w^+$  we denote the boundary triangles of $\mathcal{T} | F_w$ which are cooriented out of $\mathcal{T} | F_w$ and by $F_w^-$  the boundary triangles of $\mathcal{T} | F_w$ which are cooriented into $\mathcal{T} | F_w$. Finally, after cutting the stable branched surface $\B$ along $\B\cap F_w$ we obtain a branched surface $\B|F_w$.
For simplicity, we denote the triple $(\T|F_w, \alpha|F_w, \B|F_w)$ by $\V|F_w$. In Subsections \ref{subsec:cut:tri:versus:cut:manifold} -- \ref{subsec:mutant:homeomorphism} we do not make use of this (partial) veering structure, and consider only the pair $(\T|F_w, \alpha|F_w)$. 

The set $F_w$ can be seen as a branched surface embedded in $M$ which fully carries~$S_w^\epsilon$. We denote the (sutured) manifold underlying $\T|F_w$ by $M|F_w$. Note that $\T|F_w$ is not an ideal triangulation of $M|F_w$ in the sense 
introduced at the beginning of Section 2.  The manifold $M|F_w$ is expressed as a union of ideal tetrahedra of $\T|F_w$, but only some of their faces are identified in pairs by homeomorphisms sending vertices to vertices. The remaining faces make up the triangulations of the top and bottom boundaries of $M|F_w$. To avoid any confusion we call $\T|F_w$ a \emph{cut triangulation}. In this section we will establish a relationship between $M|F_w$ and $M|\Semb$.

Recall from Subsection \ref{subsec:surface:tri} that $\Semb$ is triangulated by $\QVw$. We denote the corresponding triangulations of $\Sembp, \Sembm$ in $M|\Semb$ by $\QVw^+$, $\QVw^-$, respectively. Given a simplex $x$ of $\QVw$ we denote by $x^+$ and $x^-$ the corresponding simplices of $\QVw^+$ and $\QVw^-$, respectively. Notation that we use below was introduced in  Subsection \ref{subsec:surface:tri}.

Let $e \in E_w = \lbrace e \in E \ | \ w_e > 0 \rbrace$. Recall that $\check{e}$ is a shorthand for $L(e)$, the lowermost copy of $e$ in $\QVw$. Suppose that $w_e = k \geq 2$. Then for \mbox{$i=1,2, \ldots, k-1$} there is a disk $D_i^e$ properly embedded in $M|\Semb$ whose boundary decomposes into four arcs: one arc corresponding to the edge $\ab^{i-1}(\check{e})^-$ of $\QVw^-$, one arc corresponding to the edge  $\ab^i(\check{e})^+$ of $\QVw^+$, and two arcs each of which joins $\ab^{i-1}(\check{e})^-$ to $\ab^i(\check{e})^+$ and intersects a suture of $(M|\Semb, \partial M|\partial \Semb)$ exactly once; see Figure \ref{fig:product:disks}. We call $\ab^{i-1}(\check{e})^-$ the \emph{bottom base} of the disk $D^e_i$ and $\ab^i(\check{e})^+$ its \emph{top base}. We denote them by $\partial^- D^e_i$, $\partial^+ D^e_i$, respectively. The set $\partial_vD^e_i = \partial D^e_i - \mathrm{int}(\partial^+ D^e_i) - \mathrm{int}(\partial^- D^e_i)$ is called the \emph{vertical boundary} of $D^e_i$.

Disk $D^e_i$ intersects the sutures of $(M|\Semb, \partial M|\partial \Semb)$ exactly twice. In the theory of sutured manifolds properly embedded disks with this property are called \emph{product disks}.   

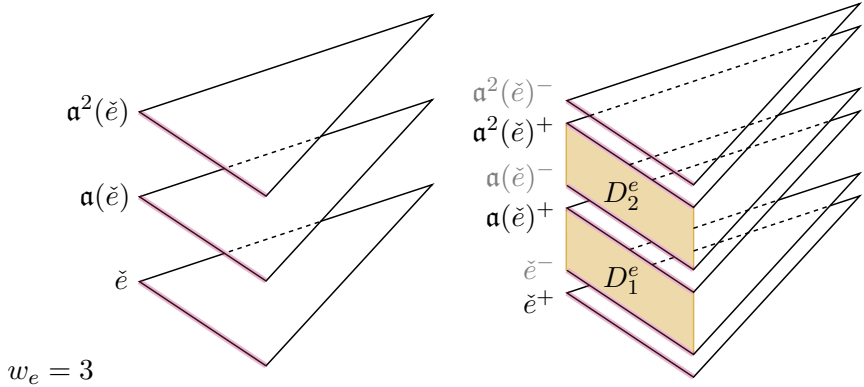
\begin{figure}[h]
\includegraphics[scale=0.75]{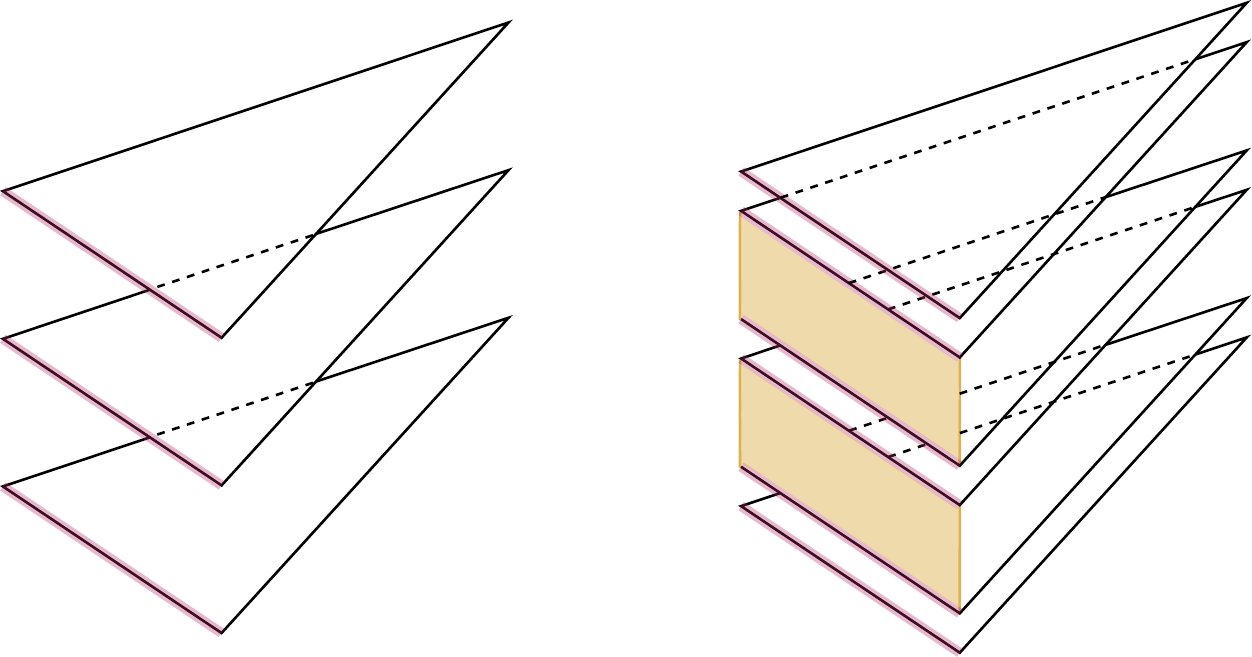}
\put(-320,0){$w_e = 3$}
\put(-280,34){$\check{e}$}
\put(-293, 66){$\ab(\check{e})$}
\put(-298, 98){$\ab^2(\check{e})$}
\put(-127,25){$\check{e}^+$}
\put(-141, 58){$\ab(\check{e})^+$}
\put(-146, 90){$\ab^2(\check{e})^+$}
\color{gray}
\put(-127,38){$\check{e}^-$}
\put(-141, 73){$\ab(\check{e})^-$}
\put(-146, 105){$\ab^2(\check{e})^-$}
\color{black}
\put(-97,35){$D_1^e$}
\put(-97,67){$D_2^e$}
\caption{Two edge product disks $D_1^e$, $D_2^e$ associated to an edge $e \in E$ of weight~3. For each copy of $e$ in $\QVw$ we draw only one triangle attached to it so that the disks $D_1^e$, $D_2^e$ are clearly visible. To simplify notation we denote $L(e)$ by~$\check{e}$.} 
\label{fig:product:disks}
\end{figure}

Let $D_w$ denote the set of product disks in $M|\Semb$ associated to the edges of $\V$ with $w_e >1$. We say that an element of $D_w$ is an \emph{edge product disk}. Note that $M|\Semb$ can admit more product disks, which are not elements of $D_w$. Let $M_w = (M|\Semb)|D_w$. Since~$M_w$ arises as a result of decomposing a sutured manifold along finitely many product disks, it is also a sutured manifold; see \cite[Definition~3.8]{Gabai-sutured}. Figure~\ref{fig:cutting_along_carried} illustrates the relationship between $M|F_w$ and $M_w$ that we formalize in Lemma \ref{lem:cut:triangulation:sutured:manifold}, after introducing necessary notation below.

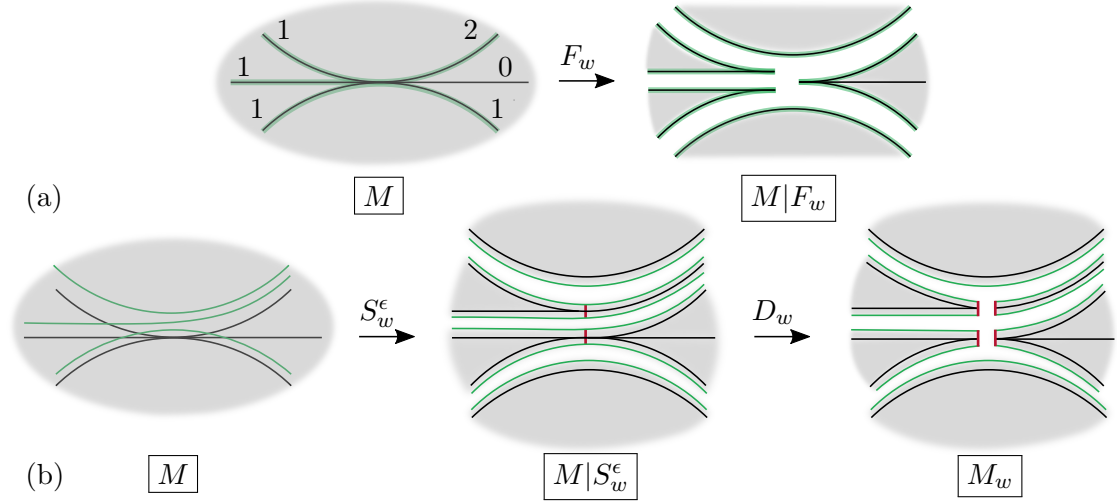
\begin{figure}[h]
\includegraphics[scale=0.78]{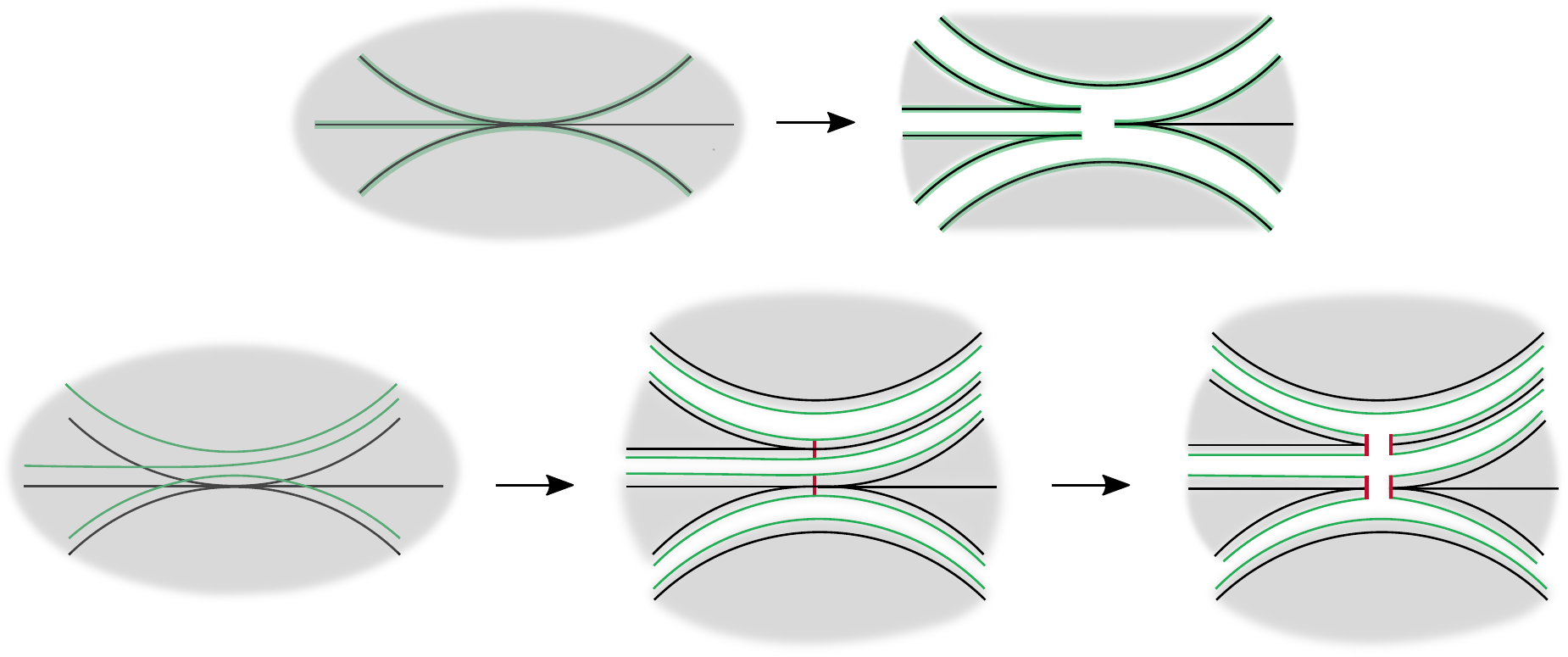}
\put(-410, 95){(a)}
\put(-287, 95){\fbox{$M$}}
\put(-142, 95){\fbox{$M|F_w$}}
\put(-210, 148){$F_w$}
\put(-316,158){1}
\put(-331,144){1}
\put(-326,128){1}
\put(-246,158){2}
\put(-233,144){0}
\put(-236,128){1}
\put(-410, -10){(b)}
\put(-364, -10){\fbox{$M$}}
\put(-216, -10){\fbox{$M|\Semb$}}
\put(-61, -10){\fbox{$M_w$}}
\put(-285, 52){$\Semb$}
\put(-138, 52){$D_w$}
\caption{(a) Cutting along $F_w$. A weight on a face is indicated by the number immediately above the face. (b) First arrow: cutting along an embedded surface $\Semb$. Second arrow: cutting along edge product disks.}
\label{fig:cutting_along_carried}
\end{figure}

Triangulations $\QVw^+, \QVw^- \subset M|\Semb$ determine a pair of triangulations $\overline{\QVw^+}, \overline{\QVw^-}$ in the top and the bottom boundary of $M_w$, respectively. For any triangle $g^\pm$ of $\QVw^\pm$ there is an associated triangle $\overline{g^\pm}$ of $\overline{\QVw^\pm}$. We will always assume that the indexing of edges/vertices of $\overline{g^\pm}$ is the same as in $g^\pm$. The only difference between $\QVw^\pm$ and $\overline{\QVw^\pm}$ is that there might be triangles $g_1^\pm,g_2^\pm$ of $\QVw^\pm$ which are identified along an edge $e_1$ of $g_1^\pm$ and an edge $e_2$ of $g_2^\pm$ such that the corresponding edges $\overline{e_1}$ of $\overline{g_1^\pm}$ and $\overline{e_2}$ of $\overline{g_2^\pm}$ are not identified in $\overline{\QVw^\pm}$. This happens if and only if the common edge of $g_1^\pm, g_2^\pm$ is the top or bottom base of an edge product disk from $D_w$.

For any $f \in F$ with $w_f > 1$ the sutured manifold $M_w$ admits $(w_f - 1)$ connected components of the form $f \times \lbrack 0, 1 \rbrack$. We call them \emph{triangular prisms}, and denote the set of such triangular prisms in $M_w$ by~$P_w$. Each $P\in P_w$ is a sutured 3-ball, so we can speak about its top and bottom boudaries $\partial^+P$, $\partial^-P$, respectively. Observe that  if $\partial^-P = \overline{g^-} \in \overline{F_{\V,w}^-}$ then $\partial^+P = \overline{\ab(g)^+} \in \overline{F_{\V,w}^+}$. We call $\partial_v P = \partial P  - \mathrm{int}(\partial^+ P) - \mathrm{int}(\partial^- P)$ the \emph{vertical boundary} of $P$.   

Each edge product disk $D \in D_w$ gives rise to two disks $D',D''$ contained in the sutured annuli of $M_w$; see Figure \ref{fig:doubled:disk}. We denote the set of such disks contained in the sutured annuli of $M_w$ by $D(M_w)$. 
Let 
\begin{equation}\label{eqn:collapse}\mathrm{coll}: M_w \rightarrow \mathrm{coll}(M_w)\end{equation}
be the map which vertically collapses every $D \in D(M_w)$ to $\partial^-D$ and every $P \in P_w$ to~$\partial^-P$.

\begin{figure}[h]
\includegraphics[scale=0.5]{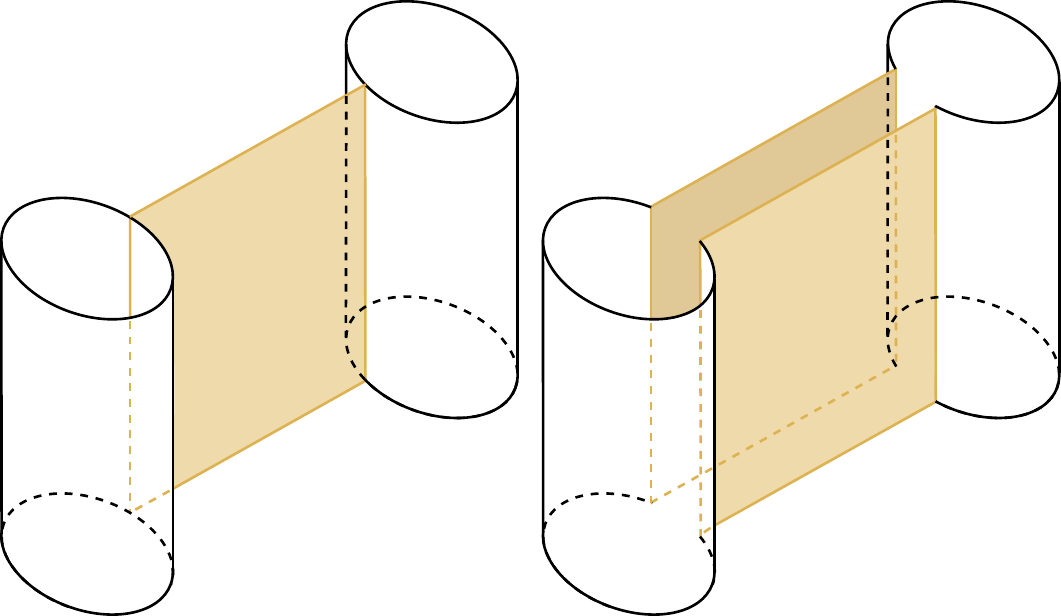}
\caption{Left: Product $D$ disk in $M|\Semb$. Right: Two disks contained in a sutured annulus of $M_w$ arising from cutting $M|\Semb$ along $D$. }
\label{fig:doubled:disk}
\end{figure}


\begin{lemma}\label{lem:cut:triangulation:sutured:manifold}
The image of $M_w - P_w$ under $\mathrm{coll}$ is homeomorphic to $M|F_w$. Furthermore, for any $f \in F_w$ 
\begin{itemize}
	\item if $\overline{g^+}$ is a triangle of $\partial^+(M_w -P_w)$  with $\coll(\overline{g^+}) = f^+$ then $\overline{g^+} = \overline{L(f)^+}$,
	\item if $\overline{g^-}$ is a triangle of  $\partial^-(M_w -P_w)$ with $\coll(\overline{g^-}) = f^-$ then $\overline{g^-} = \overline{U(f)^-}$,
\end{itemize}
\end{lemma}
\begin{proof} The manifold $M_w$ can be seen as a sutured manifold $M|(\Semb \cup D_w)$. 
Observe that
$\Semb \cup D_w$ can be obtained from $F_w$ in the following three steps.
\begin{enumerate}
	\item Replace every edge $e$ in the branch locus of $F_w$ by $e \times \lbrack 0, 1 \rbrack$, keeping the triangles  that were on the two sides of $e$ attached to $e \times \lbrace 0 \rbrace \subset e\times \lbrack 0,1\rbrack$.
	\item For every $f \in F_w$ with $w_f > 1$ add additional $(w_f - 1)$ copies of $f$ and for every edge $e$ of $f$ attach them along $e \times \lbrace 0 \rbrace$ immediately above $f$.
	\item For every $e \times \lbrack 0, 1 \rbrack$ that occur as a result of step (1), spread the triangles that are on the two sides of $e \times \lbrace 0 \rbrace$ evenly along $e \times \lbrack 0, 1 \rbrack$. 
\end{enumerate}

The last step is possible because edges in the branch locus of $F_w$ correspond to the edges of $\V$ with weight greater than one, and thus there are at least two triangles attached to either side of $e \times \lbrack 0, 1 \rbrack$.

It follows	that collapsing every $D \in D(M_w)$ vertically to $\partial^-D$ collapses $\Semb \cup D_w$ into a branched surface whose branch locus can be identified with that of $F_w$, but which has multiple parallel copies of $f \in F_w$ whenever $w_f >1$. Each region between two parallel copies of $f$ corresponds a triangular prism  $P$ with its vertical boundary collapsed and such that $\partial^-P = \overline{\ab^k(\check{f})^-}$ and $\partial^+P = \overline{\ab^{k+1}(\check{f})^+}$ for  some $0\leq k \leq w_f-2$ (recall that $\check{f} = L(f)$). Therefore collapsing each such $P$ into $\partial^-P$ results in collapsing them all to the lowermost copy of $f$. Performing this for all $f\in F_w$ yields $F_w$. It follows that $\coll(M_w - P_w)$ can be identified with $M|F_w$. 

The `furthermore' part follows from the observation that  $\overline{L(f)^+}$ is the only copy of $f$ in $\overline{\QVw^+}$ which does not have a triangular prism below it, and $\overline{U(f)^-}$ is the only copy of $f$ in $\overline{\QVw^-}$ which does not have a triangular prism above it.
\end{proof}

\subsection{The mutant triangulation}\label{subsec:mutating}

In this subsection we explain how to construct the \emph{mutant triangulation}~$\T^\varphi$ out of the cut triangulation $(\T|F_w, \alpha|F_w)$, defined in Section~\ref{subsec:cut:tri:versus:cut:manifold}, and a combinatorial automorphism $\varphi \in \Aut^+(\QVw)$. This mutant triangulation is not guaranteed to be veering even when we do have a partial veering structure $\B|F_w$ on $(\T|F_w, \alpha|F_w)$. In Section \ref{subsec:mutant:homeomorphism} we give sufficient and necessary conditions on~$\varphi$ so that $\T^\varphi$ is an ideal triangulation of $M^\varphi$. In Section \ref{subsec:veeringness:of:mutant} we put additional restrictions on $\varphi$ which allow us to define taut and veering structures on~$\T^\varphi$, thus resulting in a veering triangulation $\V^\varphi = (\T^\varphi, \alpha^\varphi, \B^\varphi)$. 

Recall from Subsection \ref{subsec:surface:tri} that $F_{\V,w}$ denotes the set of triangles of $\QVw$. A combinatorial automorphism  $\varphi \in \Aut^+(\QVw)$ gives a bijection $\varphi: F_{\V,w} \rightarrow F_{\V,w}$ and a set of bijections $\lbrace\varphi_g\rbrace_{g \in F_{\V,w}}$ between the edges of \mbox{$g \in F_{\V,w}$} and the edges of $\varphi(g) \in F_{\V,w}$. Using the natural correspondence between the triangulations $\QVw^+$ and $\QVw^-$ in the top an bottom boundries of $M|\Semb$, respectively, we can view $\varphi \in \Aut^+(\QVw)$ as a combinatorial isomorphism $\varphi: \QVw^+ \rightarrow \QVw^-$. Thus we can use $\varphi$ to construct a mutant manifold~$M^\varphi$ out of $M|\Semb$.

However, as explained in Subsection \ref{subsec:cut:tri:versus:cut:manifold}, when we work with the triangulation $(\T, \alpha)$ of~$M$, we generally do not cut along $\Semb$, but along $F_w$. Thus to construct the mutant triangulation $\T^\varphi$ we need to specify a \emph{regluing map} $r(\varphi) = \left(r^\varphi: F_w^+ \rightarrow F_w^-, \left(r^\varphi_{f^+}\right)_{f^+ \in F_w^+}\right)$ determined by~$\varphi$, consisting of a bijection  \mbox{$r^\varphi: F_w^+ \rightarrow F_w^-$} and a family of bijections $\left(r^\varphi_{f+}\right)_{f^+ \in F_w^+}$ between edges of $f^+ \in F_w^+$ and edges of $r^\varphi(f^+) \in F_w^-$. The map $r(\varphi)$ has to be such that  $\T^\varphi$ obtained from $(\T|F_w, \alpha|F_w)$ by identifying $F_w^+$ with $F_w^-$ via $r(\varphi)$ is, at least under certain conditions,  a triangulation of~$M^\varphi$. In this section we will define~$r(\varphi)$. A sufficient and necessary condition on $\varphi$ for the mutant triangulation to be a triangulation of $M^\varphi$ appears in Theorem \ref{thm:mutant:correct:manifold}.

Below we use notation introduced in Subsection~\ref{subsec:surface:tri}. In particular, recall that given $f \in F_w$ by $U(f)$, $L(f)$ we denote the uppermost and the lowermost copies of $f$ in $\QVw$, respectively. 
Furthermore, let $\iota: M_w \rightarrow M|\Semb$ be the surjective immersion, induced by cutting $M|\Semb$ along~$D_w$, which sends $\overline{g^+}$ to $g^+$ and $\overline{g^-}$ to $g^-$ for every $g\in F_{\V,w}$. 	Given $P \in P_w$ we will say that $\iota(P)$ is a triangular prism in $M|\Semb$. To simplify notation, we set $P^\iota= \iota(P)$ and $\partial^\pm P^\iota = \iota(\partial^\pm P)$.

Before we formally define $r(\varphi)$ let us briefly explain the idea behind its definition.
A triangle $g^+$ of $\QVw^+$ does not have a triangular prism below it if and only if it is in $L(F_w)^+$; see Figure \ref{fig:order}. Thus there is a natural identification between $L(F_w)^+$ and $F_w^+$. In particular, $r^\varphi(f^+)$ will depend on $\varphi(L(f)^+)$. Similarly, a triangle $g^-$ of $\QVw^-$ does not have a triangular prism above it if and only if it is in $U(F_w)^-$. Therefore we can identify $U(F_w)^-$ with $F_w^-$. If $\varphi(L(f)^+)$ is the bottom base of a triangular prism~$P^\iota$ then it is not immediately clear to which  triangle of $F_w^-$ the map $r^\varphi$ should send~$f^+$. In this case we flow upwards through the prism $P^\iota$ and look at $\varphi(\partial^+P^\iota)$. If it is in $U(F_w)^-$ then the image of $f^+$ under $r^\varphi$ will be the triangle $f'^-$ with $U(f')^- = \varphi(\partial^+ P^\iota)$. Otherwise, we continue flowing upwards through triangular prisms. Below we describe this procedure more formally and prove that it always terminates.

Given $f \in F_w$ we define a sequence $g^\varphi(f) = (g_i)_{i \geq 1}$ of triangles of $\Q_{\V, w}$  as follows. The first element $g_1$ is equal to $\varphi(L(f))$. For $i \geq 1$ if $g_i \in U(F_w)$ we are done. Otherwise, there is another triangle $\ab(g_i)$ of $\mathcal{Q}_{\V, w}$ which is a copy of the same triangle of~$(\mathcal{T}, \alpha)$ as $g_i$ and lies immediately above~$g_i$. Then we set $g_{i+1} = \varphi(\ab(g_i))$. 

\begin{lemma}\label{lem:finite:seq}
For every $f \in F_w$ the sequence $g^\varphi(f)$ is finite. Furthermore, if $f, f' \in F_w$ are distinct then the last elements of $g^\varphi(f)$ and $g^\varphi(f')$ are distinct. 
\end{lemma}
\begin{proof}
Since the triangulation $\Q_{\V, w}$ consists of finitely many triangles and $g_i$ completely determines $g_{i+1}$, if the sequence $g^\varphi(f)$ is infinite then it is eventually periodic. That is, there are integers $m\geq 0$, $N \geq 1$ such that $g_{m + j} = g_{m+k\cdot N + j}$ for any $j \in \lbrace 1, 2, \ldots, N\rbrace$ and $k \geq 0$. Pick minimal such $m$ and $N$. We can write~$g^\varphi(f)$ as 
\[g^\varphi(f) = g_1, g_2, \ldots, g_m, (h_1, \ldots, h_N), (h_1, h_2, \ldots, h_N), \ldots \]
First observe that we can assume that $m<N$. Otherwise, using the definition of $g^\varphi(f)$ and the fact that $\varphi$ is a bijection on the set of triangles of $\QVw$, we get that $g_{m-k} = h_{N-k}$ for any $k \in \lbrace 0, 1, \ldots, N-1 \rbrace$. This means that the period $(h_1, \ldots, h_N)$ starts immediately after $g_{m-N}$, if $m>N$, or there is no pre-periodic sequence at all if $m=N$. This is a contradiction with the minimality of $m$. 
On the other hand, if $m<N$ we obtain the equality $L(f) = \ab(h_{N-m})$. This is a contradiction, because the lowermost copy of  $f$ in $\QVw$ does not have any copies of $f$ below it, so in particular it cannot lie immediately above $h_{N-m}$. Thus $g^\varphi(f)$ is finite. 


Now suppose that for some $1\leq k \leq l$ we have
\[g^\varphi(f) = (g_1, g_2, \ldots, g_k), \hspace{1cm} g^\varphi(f) = (g_1', g_2', \ldots, g_l').\]
If $g_k = g_l'$ then $g_1 = g_{l-k+1}'$. If $l<k$ we get  $L(f) = \ab(g_{l-k}')$ which is a contradiction, because $L(f)$ is the lowermost copy of $f$ in $\QVw$ while $\ab(g_{l-k}')$ lies above $g_{l-k}'$. If $k=l$ we get $L(f) = L(f')$ and thus $f=f'$ by the injectivity of $L$.
\end{proof}

Let $f \in F_w$.   Denote by $k \geq 1 $ the length of the sequence $g^\varphi(f)$. By definition, $g_i \notin U(F_w)$ for all $i \in \lbrace 1, 2, \ldots, k-1 \rbrace$ and $g_k \in  U(F_w)$. Since $U$ is injective, there is a unique $f' \in F_w$ such that $g_k = U(f')$. Let $f^+$, $f'^-$ be the triangles corresponding to $f, f'$ in $F_w^+, F_w^-$, respectively, and set
\[r^\varphi(f^+) = f'^- \in F_w^-. \]

Lemma \ref{lem:finite:seq} and injectivity of $U$ imply that $r^\varphi$ is a bijection.  Thus it determines a pairing between faces of the top boundary $F_w^+$ of  the cut triangulation $(\T|F_w, \alpha|F_w)$ and faces of the bottom boundary $F_w^-$ of $(\T|F_w, \alpha|F_w)$. To define a mutant triangulation built out of $(\T|F_w, \alpha|F_w)$ it therefore remains to specify, for every $f\in F_w$, a bijection $r^\varphi_{f^+}$ between the edges of $f^+$ and that of $f'^-$. As in Subsection \ref{subsec:surface:tri}, to simplify notation we will denote $L(f)$ by $\check{f}$. Recall from Definition~\ref{def:comb:aut} that~$\varphi$ associates to $\check{f}$ 
a bijection $\varphi_{\check{f}}$ between the edges of $\check{f}$ and edges of $g_1 = \varphi(\check{f})$. Analogously, for $i=1, 2, \ldots, k-1$ there is a bijection $\varphi_i$ between the edges of $\ab(g_i)$ and edges of $g_{i+1} = \varphi(\ab(g_i))$.  
Let $\delta_i$ be the bijection between the edges of $g_i$ and edges of $\ab(g_i)$ such that $\delta_i(e) = \ab(e)$ for any edge $e$ of $g_i$. Recall from Subsection \ref{subsec:surface:tri} that $\sigma_f^L$, $\sigma_f^U$ denote the bijections between the edges of $f$ and edges of $L(f)$, $U(f)$, respectively. 
Using this we set
\[r^\varphi_{f^+} = \left(\sigma_{f'}^U\right)^{-1}\circ \varphi_{k-1}\circ \delta_{k-1} \circ\varphi_{k-2} \circ \delta_{k-2} \circ \cdots \circ \varphi_{1} \circ \delta_1 \circ \varphi_{\check{f}} \circ \sigma_f^L.\]

We will also write $r^\varphi_{f^+} = \left(\sigma_{f'}^U\right)^{-1}\circ  \left(\varphi_{i}\circ \delta_i\right)_{i=1}^{k-1} \circ \varphi_{\check{f}} \circ \sigma_f^L$ for brevity. We define the \emph{mutant triangulation} $\T^\varphi$ as the triangulation obtained from $(\T|F_w, \alpha|F_w)$ by identifying a triangle $f^+ \in F_w^+$ with the triangle $r^\varphi(f^+) \in F_w^-$ in such a way that an edge $e$ of $f^+$ is identified with the edge $r^\varphi_{f^+}(e)$ of $r^\varphi(f^+)$. 

Observe that $\T^\varphi$ is an ideal triangulation of $M^{r(\varphi)}$ which is not necessarily homeomorphic to $M^\varphi$. We explore this problem in the next subsection. For now, we state the relationship between $r(\varphi) = \left(r^\varphi: F_w^+ \rightarrow F_w^-, \left(r^\varphi_{f^+}\right)_{f ^+\in F_w^+}\right)$  and $\varphi = \left( \varphi: \QVw^+ \rightarrow \QVw^-, \left(\varphi_g\right)_{g^+ \in F_{\V,w}^+}\right)$.
\begin{lemma}\label{lem:correct:gluing}
Let $f\in F_w$. Suppose that \begin{gather*}r^\varphi(f^+) = f'^-\\
	r^\varphi_{f^+} = \left(\sigma_{f'}^U\right)^{-1}\circ \left(\varphi_{i}\circ \delta_i\right)_{i=1}^{k-1} \circ \varphi_{\check{f}}\circ \sigma_f^L.\end{gather*} 
Then
\begin{enumerate} \item $k=1$ if and only if $\varphi(L(f)^+) = U(f')^-$ and their vertices are identified by  $ \varphi_{\check{f}}$.
	\item $k \geq2$ if and only if there is a sequence $(P_1^\iota, \ldots, P_{k-1}^\iota)$ of triangular prisms in $M|\Semb$ with the following properties:
	\begin{itemize}
		\item $\varphi(L(f)^+) = \partial^- P_1^\iota$ and their vertices are identified by  $\varphi_{\check{f}}$,
		\item  vertex $v$ of $\partial^- P_i^\iota$ is below the vertex $\delta_i(v)$ of  $\partial^+P_i^\iota$ for $i = 1,2, \ldots, k-1$, 
		\item $\varphi(\partial^+P_i^\iota) = \partial^- P_{i+1}^\iota$ and their vertices are identified by  $\varphi_i$ for $i=1,2, \ldots, k-2$,
		\item $\varphi(\partial^+P_{k-1}^\iota) = U(f')^-$ and their vertices are identified by  $\varphi_{k-1}$.
	\end{itemize}
	Denote by $P(f)$ the quotient space of $L(f)^+ \cup \left(\bigcup\limits_{i=1}^{k-1}P_i^\iota\right)\cup U(f')^-$ by the identifications listed above. 
	Vertically collapsing $P(f)$  results in an identification of $L(f)^+$ with $U(f')^-$ with vertex correspondence given by $\left(\varphi_{i}\circ \delta_i\right)_{i=1}^{k-1} \circ \varphi_{\check{f}}$.
\end{enumerate}
\end{lemma}
\begin{proof}
We view $\varphi$ as a combinatorial automorphism $\varphi: \QVw^+ \rightarrow \QVw^-$. 
The assumption on $r^\varphi_{f^+}$ implies that $g^\varphi(f)$ has length~$k$. Let $g^\varphi(f) = (g_1, \ldots, g_k)$. 
By the definition of $g^\varphi(f)$ we get that $\varphi(L(f)^+)  \in U(F_w)^-$ if and only if $k=1$. In this case indeed $\varphi(L(f)^+) = g_1^- = U(f')^-$ and the vertex correspondence between these triangles is given by $ \varphi_{\check{f}}$.

For the case $k\geq 2$ the existence of $P_1^\iota, P_2^\iota, \ldots, P_{k-1}^\iota$ follows from the definition of $g^\varphi(f)$. Namely, since $g_i^- \notin U(F_w)^-$ and $g_{i+1}^- = \varphi(\ab(g_i)^+)$  for $1\leq i \leq k-1$, there are prisms $P_1^\iota, \ldots, P_{k-1}^\iota$ in $M|\Semb$ satisfying 	$\partial^-P_i^\iota = g_i^-$ and $\partial^+P_i^\iota = \ab(g_i)^+$, and such that identifications on $\partial^\pm P_i^\iota$ are as required.
\end{proof}

\subsection{Manifold underlying the mutant triangulation}\label{subsec:mutant:homeomorphism}
The mutant triangulation~$\T^\varphi$ defined in Section \ref{subsec:mutating} is an ideal triangulation of $M^{r(\varphi)}$. In this section we study the relationship between $M^{r(\varphi)}$ and $M^\varphi$, and derive sufficient and necessary condition for them to be homeomorphic.


Recall from Lemma \ref{lem:cut:triangulation:sutured:manifold} that $M|F_w$ is more closely related to $M_w = (M|\Semb)|D_w$ than it is to $M|\Semb$. As in Subsection \ref{subsec:cut:tri:versus:cut:manifold}, we will denote the triangulations in the top and bottom boundary of $M_w$ by $\overline{\QVw^+}, \overline{\QVw^-}$, respectively. 
A combinatorial isomorphism  $\varphi: \QVw^+ \rightarrow \QVw^-$ determines a map \mbox{$\overline{\varphi}: \overline{\QVw^+} \rightarrow \overline{\QVw^-}$} via $\overline{\varphi}(\overline{f}) = \overline{\varphi(f)}$ and $\overline{\varphi}_{\overline{f}} = \varphi_f$. Note that $\overline{\varphi}$ is not a combinatorial automorphism in the sense of Definition \ref{def:comb:aut}, as it can map a pair non-adjacent triangles to a pair of adjacent triangles, and vice versa. Nonetheless, we can use $\overline{\varphi}$ to construct a mutant manifold $M_w^{\overline{\varphi}}$ out of~$M_w$.
Figure~\ref{fig:manifolds:diagram} summarizes relationships between $M^\varphi$, $M^{r(\varphi)}$ and~$M_w^{\overline{\varphi}}$.

\begin{figure}[h]
\[\begin{tikzcd}[ar symbol/.style = {draw=none,"\textstyle#1" description,sloped},
	isomorphic/.style = {ar symbol={\cong}},]
	\text{combinatorially:} &(\T, \alpha) \arrow[r, "F_w"] \ar[d, isomorphic]&(\T|F_w, \alpha|F_w)\arrow[r, rightsquigarrow, "r(\varphi)"] \ar[d, isomorphic]& \T^\varphi \ar[d, isomorphic]& \\
	\text{topologically:} &M \arrow[r, "F_w"] \arrow[rd, "\Semb"]& M|F_w \arrow[r, rightsquigarrow, "r(\varphi)"] & M^{r(\varphi)}&\\ 	
	& & M|\Semb \arrow[r, rightsquigarrow, "\varphi"] \arrow[rd, "D_w"]& M^\varphi\\
	& & & M_w \arrow[r, rightsquigarrow, "\overline{\varphi}"]& M_w^{\overline{\varphi}}\\
\end{tikzcd}
\]
\caption{Each straight  arrow corresponds to cutting the \mbox{3-manifold} on the left of the arrow open along the set specified above the arrow. Each squiggly arrow corresponds to gluing the top boundary of the sutured manifold on the left of the arrow to its bottom boundary via the map specified above the arrow.}
\label{fig:manifolds:diagram}
\end{figure}

The relationship between $M^\varphi$ and $M_w^{\overline{\varphi}}$ is the easiest to state.
\begin{lemma}\label{lem:cuts:M_w^phi}
$M_w^{\overline{\varphi}}$ is obtained from $M^\varphi$ by cutting it along finitely many (potentially zero) disks, annuli and M\"obius bands.
\end{lemma}
\begin{proof}
Recall that $M_w = (M|\Semb)|D_w$ is obtained from $M|\Semb$ by cutting it along finitely many product disks.  These cuts persists in the mutant manifold $M_w^{\overline{\varphi}}$. Since $\overline{\varphi}$ maps $\overline{f}$ to $\overline{\varphi(f)}$ with vertex correspondence $\varphi_f$, these cuts are the only difference between $M^\varphi$ and $M_w^{\overline{\varphi}}$.

Observe that under $\varphi$ the top boundary of an edge product disk $D\in D_w$ can be mapped to the bottom boundary of an edge product disk $D' \in D_w$. Thus edge product disks can match up into annuli or M\"obius bands in $M^\varphi$ which are cut in $M_w^{\overline{\varphi}}$.
\end{proof}

We will say that the disks, annuli and M\"obius bands in $M^\varphi$ coming from edge product disks  are \emph{vertical}. Below we define a property of $\varphi$ which ensures the existence of vertical annuli or M\"obius bands.

\begin{definition}\label{defn:aligns:product:disks}
Let $S_w$ be a surface carried by a veering triangulation $\V$. We say that $\varphi \in \Aut^+(\QVw)$ \emph{aligns edge product disks} if there is a sequence of edge product disks $(D_i)_{i \in I} \subset D_w$ in $M|\Semb$ which glue up to an annulus or a M\"obius band in $M^\varphi$. Otherwise we say that $\varphi$ \emph{misaligns edge product disks}.
\end{definition}

In Theorem \ref{thm:mutant:correct:manifold} we will prove that the mutant triangulation $\T^\varphi$ is an ideal triangulation of $M^\varphi$ if and only if $\varphi$ misaligns edge product disks. 
The forward direction will rely on an observation that when $\varphi$ aligns edge product disks $M^\varphi$ and $M^{r(\varphi)}$ either  have different number of connected components or have non-homeomorphic boundary. The boundary of $M^\varphi$ is composed of sutured annuli and tori of $M|\Semb$. Since~$\Semb$ is an oriented surface in an oriented 3-manifold $M$, it induces orientation on the boundaries of sutured annuli of $M|\Semb$. Furthermore, since $\varphi$ is orientation-preserving, it sends a top boundary of a sutured annulus of $M|\Semb$ to a bottom boundary of a sutured annulus of $M|\Semb$ in an orientation-preserving way. 
It follows that all boundary components of~$M^\varphi$ are tori.

Lemma \ref{lem:cut:triangulation:sutured:manifold} implies that by identifying, for every $f\in F_w$, $f^+$ with $\overline{L(f)^+}$ and $f^-$ with $\overline{U(f)^-}$  we can view $M^{r(\varphi)}$ as a quotient space of $\coll(M_w \bez P_w)$. Therefore boundary components of $M^{r(\varphi)}$ are composed of the images of sutured annuli and tori of $M_w \bez P_w$ under the collapsing map \eqref{eqn:collapse}. 	
Given a sutured annulus $A$ of $M|\Semb$ the image $\coll(\iota^{-1}(A))$ is either an annulus or a disjoint union of bigon disks and intervals. 
The latter option happens if and only if $A \cap D_w \neq \emptyset$. In this case $D_w$ separates $A$ into finitely many rectangles that we call \emph{$D_w$-rectangles}.

\begin{definition}
Let $A$ be a sutured annulus of $M|\Semb$.
We say that a subset $R \subset A$ is a \emph{$D_w$-rectangle} if there are edge product disks $D, D' \in D_w$ such that the boundary of~$R$ decomposes into four arcs: one arc $\partial^+R$ contained in $\partial^+A$, one arc $\partial^-R$ contained in $\partial^-A$, one arc contained in $\partial_v D$, and  one arc contained in $\partial_v D'$. We call the last two arcs in the boundary of $R$ the \emph{vertical sides} of $R$.
\end{definition}
We say that a $D_w$-rectangle $R$ is \emph{prismatic} if there is a triangular prism $P \in P_w$ such that  $P\cap \iota^{-1}(R) = \iota^{-1}(R)$. Otherwise we say that $R$ is \emph{non-prismatic}. If $R$ is non-prismatic then $\coll(\iota^{-1}(R))$ is a bigon disk. In this case the boundary of $\coll(\iota^{-1}(R))$ decomposes into the \emph{positive boundary}  $\partial^+\coll(\iota^{-1}(R)) = \coll(\iota^{-1}(\partial^+ R))$, and the \emph{negative boundary} $\partial^-\coll(\iota^{-1}(R)) = \coll(\iota^{-1}(\partial^- R))$; see Figure \ref{fig:cones}. When $R$ is prismatic $\coll(\iota^{-1}(R))$ is an interval.

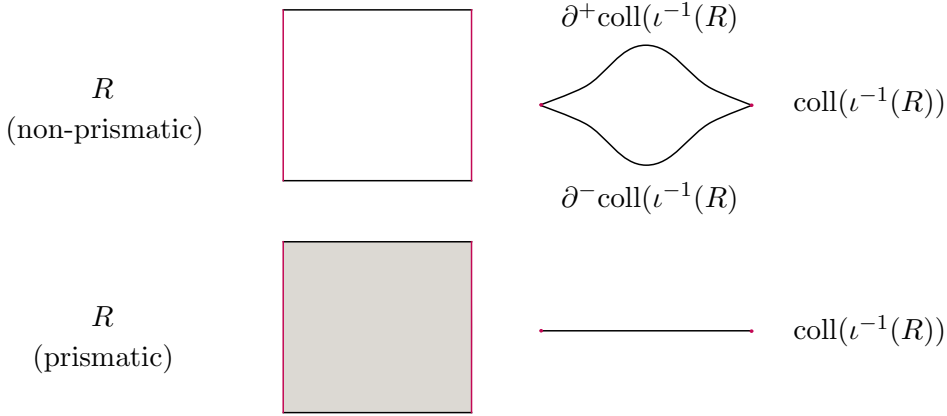
\begin{figure}[h]
\includegraphics[scale=0.75]{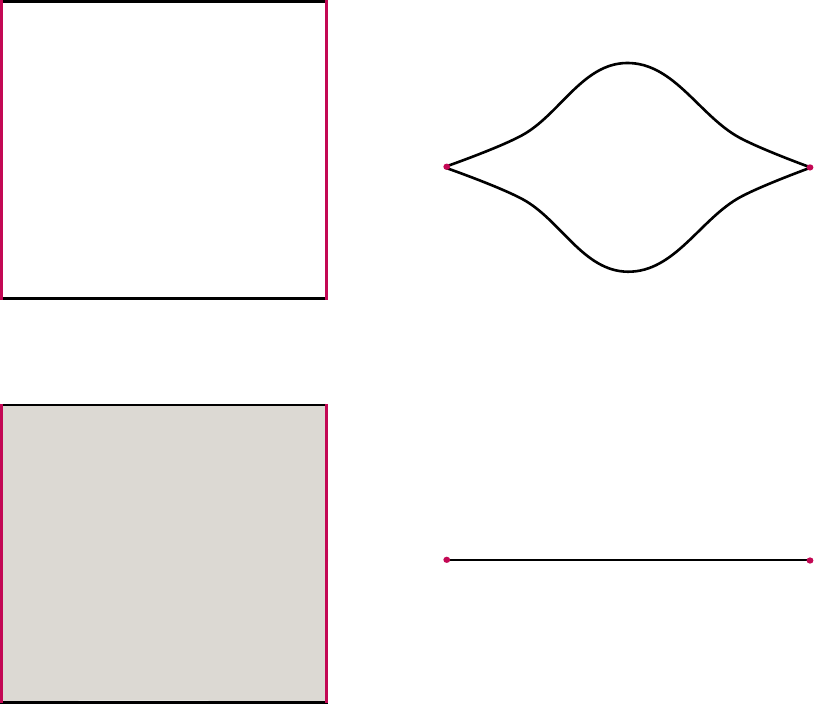}
\put(-248,119){$R$}
\put(-280,104){(non-prismatic)}
\put(15,114){$\coll(\iota^{-1}(R))$}
\put(-72,146){$\partial^+\coll(\iota^{-1}(R)$}
\put(-72,78){$\partial^-\coll(\iota^{-1}(R)$}
\put(-248,33){$R$}
\put(-270,18){(prismatic)}
\put(15,28){$\coll(\iota^{-1}(R))$}
\caption{Top: A non-prismatic $D_w$-rectangle and its image under $\coll$. Bottom: A prismatic $D_w$-rectangle and its image under $\coll$. Red vertical intervals correspond to the intersection of $R$ with $D_w$.}
\label{fig:collapsed:rectangle}
\end{figure}

\begin{theorem}\label{thm:mutant:correct:manifold}
The mutant triangulation $\T^\varphi$ is an ideal triangulation of $M^\varphi$ if and only if $\varphi$ misaligns edge product disks.
\end{theorem}

\begin{proof}

By Lemma \ref{lem:cut:triangulation:sutured:manifold}, we can view view $M^{r(\varphi)}$ as a quotient space of $\coll(M_w \bez P_w)$. 
The map
\[\partial^+(\coll(M_w \bez P_w)) \rightarrow \partial^-(\coll(M_w \bez P_w))\] by which we quotient $\coll(M_w \bez P_w)$ to get $M^{r(\varphi)}$ is obtained from $r(\varphi) = \left(r^\varphi:F_w^+ \rightarrow F_w^-, \left(r_{f^+}^\varphi\right)_{f^+\in F_w^+}\right)$ by modifying the bijections $r^\varphi_{f^+}$ to make up for the identifications $f^+ \sim\overline{L(f)^+}$, $f^-\sim\overline{U(f)^-}$. For simplicity, we abuse the notation and denote this map by $r(\varphi)$.

Let $V_w^\varphi$ denote the set of vertical annuli and M\"obius bands in $M^\varphi$. This set is nonempty if and only if $\varphi$ aligns edge product disks. 

Suppose that $\varphi$ aligns edge product disks. Let $T$ be a boundary torus of $M^\varphi$ with $T \cap V_w^\varphi \neq \emptyset$. Observe that $T\cap V_w^\varphi$ consists of finitely many parallel simple closed curves in $T$. We denote the connected components of $T\cap V_w^\varphi$ by $d_1, \ldots, d_r$, $r\geq 1$, and we assume that they are circularly ordered so that $d_i$ and $d_{i+1}$ cobound an annulus $X_i \subset T$ whose interior is disjoint from $V_w^\varphi$ (the subscript $r$ is taken modulo~$r$). 
Let $A_1, \ldots, A_N$ be sutured annuli of $M|\Semb$ such  that $\varphi(\partial^+ A_j) = \partial^-A_{j+1}$ for every $j=1,2, \ldots, N$ (the subscript $j$ is taken modulo $N$) and $T$ is the quotient of $A_1 \sqcup A_2 \sqcup \ldots \sqcup A_N$ by~$\varphi$. 
Each $A_j \cap X_i$ consists of finitely many $D_w$-rectangles. Let $\mathcal{R}_i^j$ be the collection of $D_w$-rectangles making up $A_j \cap X_i$ and let~$\mathcal{R}_i$ be the union of all $\mathcal{R}_i^j$. 

\noindent\textbf{Case 1:} There is $i \in \lbrace1, 2, \ldots, r\rbrace $ such that $\mathcal{R}_i$ contains a non-prismatic $D_w$-rectangle.

If every non-prismatic $D_w$-rectangle $R \in \mathcal{R}_i$ has both vertical sides contained in $d_i \cup d_{i+1}$ then each $\mathcal{R}_i^j$ either contains only prismatic $D_w$-rectangles or  contains exactly one non-prismatic $D_w$-rectangle. Furthermore, by the assumption of this case, there is at least one $\mathcal{R}_i^j$ of the latter type. Thus, by the definition of $r(\varphi)$, there are non-prismatic $D_w$-rectangles $R_1, \ldots, R_n \in \mathcal{R}_i$, $1\leq n\leq N$, such that $r(\varphi)(\partial^+\coll(\iota^{-1}(R_j))) = \partial^-\coll(\iota^{-1}(R_{j+1}))$, where the subscript $j$ is taken modulo~$n$. 
Since the image of a non-prismatic $D_w$-rectangle under $\coll$ is  a bigon disk (see Figure~\ref{fig:collapsed:rectangle}), this gives us a sequence of bigon disks glued to each other top to bottom. The bigons $\coll(\iota^{-1}(R_j))$ inherit orientation on their boundary from $R_j$. Thus the assumption that $\varphi$ is orientation-preserving together with Lemma \ref{lem:correct:gluing} imply that the quotient space of $\coll(\iota^{-1}(\mathcal{R}_i))$ by $r(\varphi)$ is a sphere. This means that $M^{r(\varphi)}$ admits a spherical boundary component. Since $M^\varphi$ has only toroidal boundary components, these manifolds cannot be homeomorphic.

Now suppose that $R\in \mathcal{R}_i$ is a non-prismatic $D_w$-rectangle which has a vertical side $d$ that is disjoint from $d_i \cup d_{i+1}$. An example of such a situation is presented in Figure~\ref{fig:spherical}. 
The assumption that $\mathrm{int}(X_i)\cap V_w^\varphi = \emptyset$ implies that there is a non-prismatic $D_w$-rectangle  $R' \in \mathcal{R}_i$ such that $r(\varphi)$ identifies $\coll(\iota^{-1}(d))$ with an interior point of $\partial^-\coll(\iota^{-1}(R'))$. Therefore the quotient space of $\coll(\iota^{-1}(\mathcal{R}_i))$ by $r(\varphi)$ is again a sphere and thus $M^{r(\varphi)}$ is not homeomorpic to $M^\varphi$.  

\noindent\textbf{Case 2:} For every $1\leq i \leq r$ the set $\mathcal{R}_i$ contains only prismatic $D_w$-rectangles.


\noindent \textbf{Subcase 2A:} There is a boundary torus $T'$ of $M^\varphi$ such that $T'\cap V_w^\varphi \neq \emptyset$ and $T'$ is not composed entirely of prismatic $D_w$-rectangles. In this situation we can apply Case~1 to a connected component of  $T' \bez V_w^\varphi$  to deduce that $M^{r(\varphi)}$ admits a spherical boundary component and is therefore not homeomorphic to $M^\varphi$.

\noindent \textbf{Subcase 2B:} For every boundary torus $T'$ of $M^\varphi$ either $T'\cap V_w^\varphi = \emptyset$ or $T'$ is composed entirely of prismatic $D_w$-rectangles. Since $V_w^\varphi \neq \emptyset$ this assumption implies that there is a connected component of $M^\varphi$ consisting entirely of the images of triangular prisms under $\iota$. In particular, $M^\varphi$ has strictly more connected components than $M^{r(\varphi)}$, so these manifolds are not homeomorphic. (Note, however, that $M^{r(\varphi)}$ may be homeomorphic to the union of other connected components of $M^\varphi$.)

\begin{figure}[h]
	\includegraphics[scale=1]{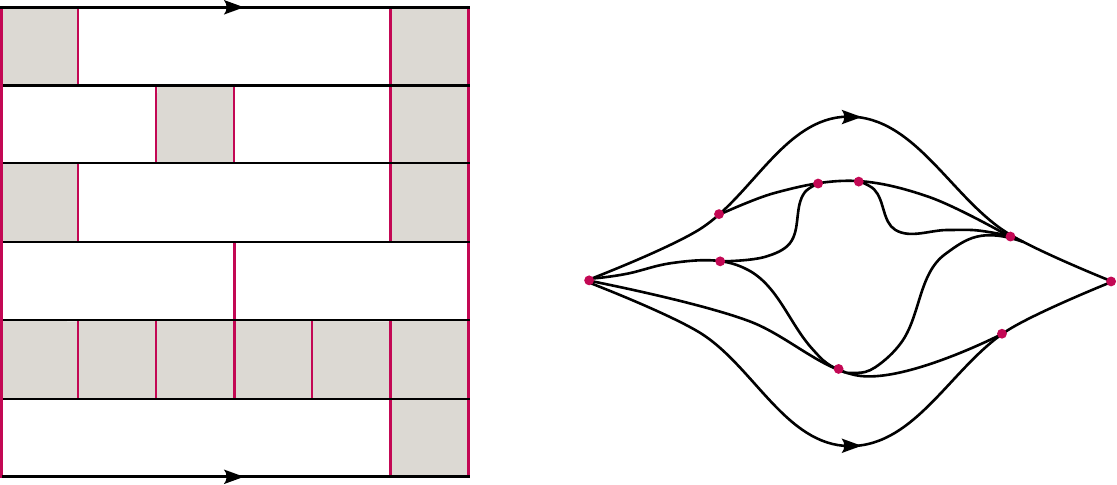}
	\put(-325,-10){$d_i$}
	\put(-195,-10){$d_{i+1}$}
	\put(-370,10){$A_1 \cap X_i$}
	\put(-370,32){$A_2 \cap X_i$}
	\put(-370,54){$A_3 \cap X_i$}
	\put(-370,76){$A_4 \cap X_i$}
	\put(-370,98){$A_5 \cap X_i$}
	\put(-370,123){$A_6 \cap X_i$}
	\put(-268,9){1}
	\put(-292,54){2}
	\put(-223,54){3}
	\put(-258,78){4}
	\put(-303,100){5}
	\put(-235,100){6}
	\put(-258,123){7}
	\put(-81,18){1}
	\put(-120,54){2}
	\put(-40,54){3}
	\put(-80,55){4}
	\put(-108,70){5}
	\put(-62,75){6}
	\put(-80,92){7}
	
	\caption{Left: An annular subset $X_i$ of a boundary torus $T$ of $M^\varphi$ cobounded by a pair of vertical annuli or M\"obius bands. Red intervals correspond to the intersection of $X_i$ with edge product disks. Prismatic $D_w$-rectangles are shaded gray. Non-prismatic $D_w$-rectangles are numbered. Right:  Spherical boundary component of $M^{r(\varphi)}$ corresponding to~$X_i$. A bigon arising as a result of collapsing a non-prismatic $D_w$-rectangle labelled with $i$ on the left is labelled with $i$.}
	\label{fig:spherical}
\end{figure}
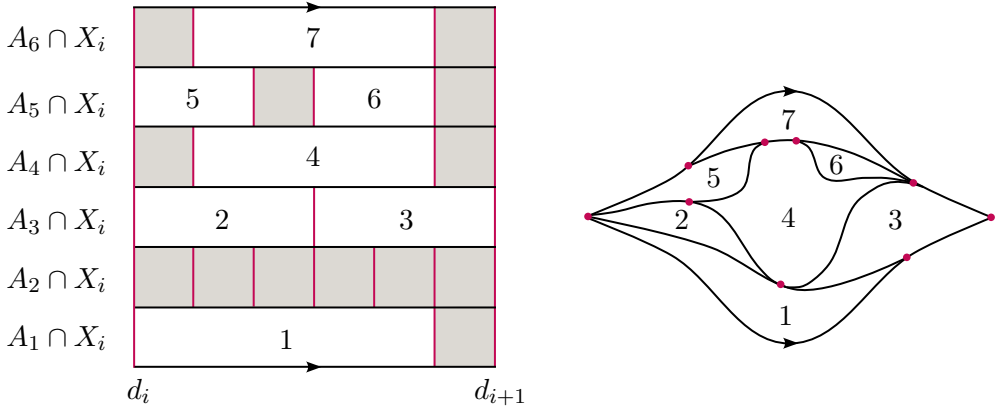

Now suppose that $\varphi$ misaligns edge product disks. Then there are no vertical annuli or M\"obius bands in $M^\varphi$ and hence, by Lemma~\ref{lem:cuts:M_w^phi}, $M_w^{\overline{\varphi}}$ can be obtained from $M^\varphi$ by cutting it along finitely many vertical disks. Equivalently, $M^\varphi$ is the quotient space of $\coll(M_w)$ by $\overline{\varphi}$.

Lemma \ref{lem:correct:gluing} explains how the definition of the regluing map simulates the process of collapsing triangular prisms into their bottom triangles. Thus $r(\varphi)$ on $F_w^+$  respects not only the identification between $L(f)^+$ and $\varphi(L(f)^+)$, for  all $f \in F_w$, but also the identification between $g^+$ and $\varphi(g^+)$ for all $g \in F_{\V,w}$ such that either $g$ or $\ab^{-1}(g)$ appears in the sequence $g^\varphi(f)$ for some $f \in F_w$. If there is $g  \in F_{\V,w}\bez L(F_w)$ such that for every $f\in F_w$ neither $g$ nor $\ab^{-1}(g)$ appears in $g^\varphi(f)$ then there are triangular prisms of $M_w = (M|\Semb)|D_w$ through which we have not passed  when defining $r(\varphi)$. These triangular prisms would arrange into solid tori components of $M_w^{\overline{\varphi}}$ consistiny entirely of triangular prisms. However,  the assumption that $\varphi$ misaligns edge product disks implies that $M_w^{\overline{\varphi}}$ does not admit such solid tori components. Therefore when $\varphi$ misaligns edge product disks Lemma \ref{lem:correct:gluing} implies that the quotient space of $\coll(M_w)$ by $\overline{\varphi}$ is homeomorphic to the quotient space of $\coll(M_w -P_w)$ by $r(\varphi)$. The latter is $M^{r(\varphi)}$, while the former --- as explained in the previous paragraph --- is $M^\varphi$.
Thus $M^{r(\varphi)}$ is homeomorphic to~$M^\varphi$.
\end{proof}

\begin{remark}
In the proof of Theorem \ref{thm:mutant:correct:manifold} we constructed a sphere out of bigon disks. This may look like a contradiction to the fact, mentioned in Subsection \ref{subsec:taut},  that only surfaces with zero Euler characteristic can admit a bigon train track. However, the obtained decomposition of $S^2$ into bigons is not a bigon track in the usual sense.  If $\tau$ is a train track on $S$ then for every switch $v$ of $\tau$ which is not contained in $\partial S$ there must be two complementary regions of $\tau$ which meet $v$ along a smooth point in their boundary. In the construction we get two points in the sphere which meet only cusps of bigons.
\end{remark}
\subsection{Veeringness of the mutant triangulation} \label{subsec:veeringness:of:mutant}
In Subsection \ref{subsec:mutant:homeomorphism}  we found a sufficient and necessary condition on $\varphi$ for the mutant triangulation $\T^\varphi$ to be a triangulation of~$M^\varphi$. In this subsection we are interested in endowing $\T^\varphi$ with a veering structure.  

By tautness, edges of the dual spine $\D$ of $(\T, \alpha)$ admit orientations such that every vertex $v$ of $\D$ has exactly two incoming edges and two outgoing edges; this is condition~(1) from Definition \ref{def:taut}. When we construct $\T^\varphi$ out of $(\T|F_w, \alpha|F_w)$ we always identify a face $f^+ \in F_w^+$ with a face $r^\varphi(f^+) \in F_w^-$. Therefore there is a natural orientation on the edges of the dual spine $\D^\varphi$ of $\T^\varphi$ that is induced from the orientation on the edges of the dual spine of $\T$. With this orientation $\D^\varphi$  satisfies condition (1) from Definition \ref{def:taut}. To obtain a taut structure on $\T^\varphi$ it suffices to find a sufficient condition on $r(\varphi)$ so that every 2-cell of $\D^\varphi$ has exactly one top vertex and exactly one bottom vertex. To derive such a condition it is helpful to analyze the structure of $\D|F_w$ and its relationship to $(\T|F_w, \alpha|F_w)$. First, observe that edges of $(\T|F_w, \alpha|F_w)$ can be classified into four types.
We say that an edge $e$ of  $(\T|F_w, \alpha|F_w)$, or $\V|F_w$, is
\begin{itemize}
\item \emph{internal} if $e$ is neither an edge of a triangle from $F_w^+$ nor an edge of a triangle from $F_w^-$, 
\item \emph{positive} if $e$ is an edge of a triangle from $F_w^+$ and is not edge of a triangle from~$F_w^-$,
\item \emph{negative} if $e$ is an edge of a triangle from $F_w^-$ and is not an edge of a triangle from $F_w^+$,
\item \emph{mixed} if $e$ is an edge of a triangle from $F_w^+$ and also an edge of a triangle from~$F_w^-$.
\end{itemize}

Assume that $\T$ and~$\D$ are embedded in~$M$ so that they are  dual to one another. 
For every 2-cell~$p$ of $\D|F_w$ there is a 2-cell $s$ of $\D$ such that $p$ is a connected component of $s|F_w$. If $s\cap F_w = \emptyset$ then we say that $p$ is an \emph{internal cell} of $\D|F_w$. Now suppose that $s\cap F_w \neq \emptyset$. If  $p$ contains the top vertex of $s$ we say that $p$ is a \emph{negative cell} of $\D|F_w$; see Figure~\ref{fig:partial_cells}(a). If $p$ contains the bottom vertex of $s$ we say that $p$ is a \emph{positive cell} of $\D|F_w$; see Figure \ref{fig:partial_cells}(b). If $p$ contains neither the top nor the bottom vertex of $s$ we say that $p$ is a \emph{mixed cell} of $\D|F_w$; see Figure \ref{fig:partial_cells}(c). Naturally, internal/positive/negative/mixed cell of $\D|F_w$ is dual to an internal/positive/negative/mixed edge of $(\T|F_w, \alpha|F_w)$.

\begin{figure}[h]
\includegraphics[scale=0.8]{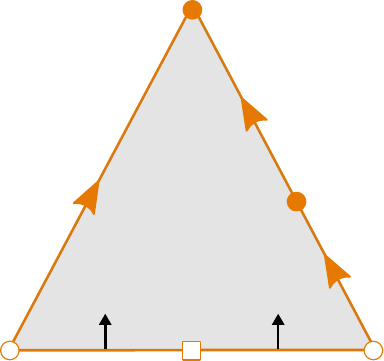} \hspace*{1cm}
\includegraphics[scale=0.8]{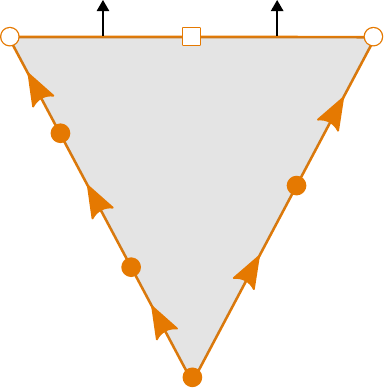} \hspace*{1cm}
\includegraphics[scale=0.8]{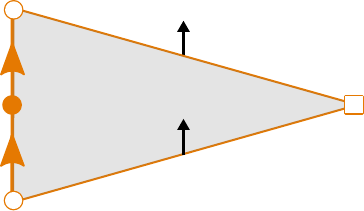}
\put(-295, -30){(a)}
\put(-292, 89){$t_s^p$}
\put(-292, -12){$e_s^p$}
\put(-315, -10){$\gamma_1^-$}
\put(-273, -10){$\gamma_2^-$}
\put(-325, 40){$\delta_1$}
\put(-262, 40){$\delta_2$}
\put(-170, -30){(b)}
\put(-167, -12){$b_s^p$}
\put(-168, 89){$e_s^p$}
\put(-200, 89){$\gamma_1^+$}
\put(-141, 89){$\gamma_2^+$}
\put(-200, 40){$\delta_1$}
\put(-135, 40){$\delta_2$}
\put(-48, -30){(c)}
\put(-52, 0){$\partial^- p$}
\put(-52, 45){$\partial^+ p$}
\put(5, 21){$e_s^p$}
\put(-91, 21){$\delta$}
\caption{(a) Negative cell of $\D|F_w$. b) Positive cell of $\D|F_w$. (c) Mixed cell of $\D|F_w$. The number of vertices in the interiors of $\delta_1, \delta_2, \delta$ may vary.}
\label{fig:partial_cells}
\end{figure}
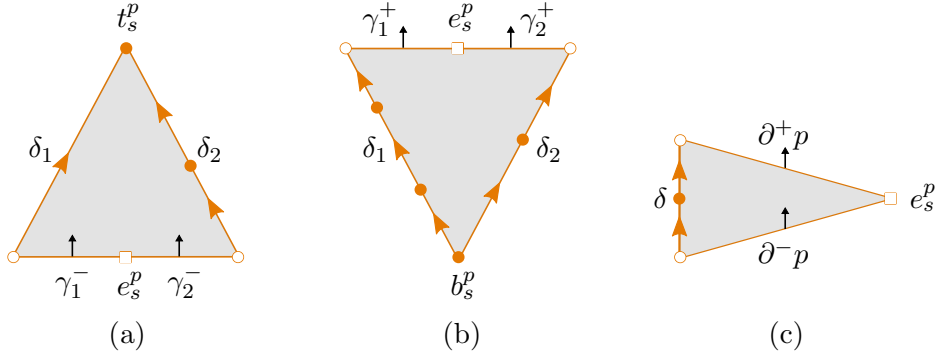

For every cell $s$ of $\D$ such that $s \cap F_w \neq \emptyset$ we denote by $e_s$ the intersection of~$s$ with its dual edge in $(\T,\alpha)$. Then every connected component $p$ of $s|F_w$ has a point in its boundary corresponding to $e_s$ that we will denote by $e_s^p$. If $p$ is negative then there is a point $t_s^p$ in the boundary of $p$ corresponding to the top vertex $t_s$ of $s$. If $p$ is positive then there is a point $b_s^p$ in the boundary of $p$ corresponding to the bottom vertex $b_s$ of~$s$.

If $p$ is a negative cell of $\D|F_w$ then its boundary decomposes into: 
\begin{itemize}
\item two arcs $\gamma_1^-, \gamma_2^-$ meeting at $e_s^p$, both cooriented into $p$,
\item two arcs $\delta_1, \delta_2$ meeting at  $t_s^p$, both oriented so that they point into $t_s^p$.
\end{itemize}
See Figure \ref{fig:partial_cells} (a). We say that $\gamma_1^-, \gamma_2^-$ are \emph{maximal negative arcs} in the boundary of~$p$.

If $p$ is a positive cell of $\D|F_w$ then its boundary decomposes into:
\begin{itemize}
\item two arcs $\gamma_1^+, \gamma_2^+$ meeting at $e_s^p$, both cooriented out of $p$,
\item two arcs $\delta_1, \delta_2$ meeting at $b_s^p$, both oriented so that they point out of $b_s^p$.
\end{itemize}
See Figure \ref{fig:partial_cells} (b). We say that  $\gamma_1^+, \gamma_2^+$ are \emph{maximal positive arcs} in the boundary of $p$.

If $p$ is a mixed cell of $\D|F_w$ then its boundary decomposes into
\begin{itemize}
\item two arcs $\partial^+p, \partial^-p$ meeting at $e_s^p$, such that $\partial^+p$ is cooriented out of $p$ and $\partial^-p$ is cooriented into $p$,
\item one arc $\delta$ oriented from $\partial^-p$ to $\partial^+p$. 
\end{itemize}
See Figure \ref{fig:partial_cells} (c). We say that  $\partial^+p$ is the \emph{maximal positive arc} in the boundary of $p$, and that $\partial^-p$ is the \emph{maximal negative arc} in the boundary of $p$. 

If $\gamma^\pm$ is a maximal positive/negative arc in the boundary of a cell $p$ of $\D|F_w$ we say that $e_s^p$ is the \emph{internal endpoint} of $\gamma^\pm$. 

Using the duality between internal/positive/negative/mixed edges of $(\T|F_w,  \alpha|F_w)$ and internal/positive/negative/mixed cells of $\D|F_w$, respectively, we can now derive a combinatorial sufficient condition on $r(\varphi)$ so that $\T^\varphi$ admits a taut structure.

\begin{lemma}\label{lem:sufficient:tautness}
If  for every mixed edge $e$ of $(T|F_w, \alpha|F_w)$ there is a positive edge $e^+$ of $(T|F_w, \alpha|F_w)$ and a negative edge $e^-$ of $(T|F_w, \alpha|F_w)$ such that $e, e^+, e^-$ are identified  in $\T^\varphi$ then the triangulation $\T^\varphi$ admits a taut structure.
\end{lemma}
\begin{proof}
We assume that the 1-skeleton of $D^\varphi$ is equipped with the orientation induced by $\alpha|F_w$. It suffices to show that under the assumption of the lemma, every 2-cell of~$D^\varphi$ has exactly one top vertex and exactly one bottom vertex; see Definition \ref{def:taut}.

Denote by $\Gamma^+$ (respectively, $\Gamma^-$) the set of maximal positive (respectively, negative) arcs in the boundaries of cells of $\D|F_w$. Recall that $\T^\varphi$ is the quotient space of $(\T|F_w, \alpha|F_w)$ under the regluing map $r(\varphi)$. 
Dually we get a regluing map $\Gamma(\varphi) = \left\lbrace \Gamma^\varphi: \Gamma^+ \rightarrow \Gamma^-, \left(\Gamma^\varphi_{\gamma^+}\right)_{\gamma^+ \in \Gamma^+} \right\rbrace$, where $\Gamma^\varphi$ is a bijection between $\Gamma^+$ and $\Gamma^-$ and  $\Gamma^\varphi_{\gamma^+}$ is a bijection between the endpoints of $\gamma^+$ and the endpoints of $\Gamma^\varphi(\gamma^+)$, such that the quotient of $\D|F_w$ by $\Gamma(\varphi)$ gives $\D^\varphi$. Since the 1-skeleton of $\D^\varphi$ arises from recombining the 1-skeleton of $\D$ we get that $\Gamma^\varphi_{\gamma^+}$ must send the internal endpoint of $\gamma^+$ to the internal endpoint of~$\Gamma^\varphi(\gamma^+)$. 

Let $p$ be a positive cell of $\D|F_w$. Denote by $\gamma_1^+, \gamma_2^+$ the two distinct maximal positive arcs in the boundary of $p$. For $i=1,2$ the bijection $\Gamma^\varphi$ can send $\gamma_i^+$ only to a maximal negative arc in the boundary of a mixed cell or to a maximal negative arc in the boundary of a negative cell. Let $q_1, \ldots, q_m$ be the maximal collection of mixed cells of $\D|F_w$ such that $\Gamma^\varphi(\gamma_1^+) = \partial^- q_1$ and $\Gamma^\varphi(\partial^+ q_i) = \partial^- q_{i+1}$ for $i=1,2, \ldots, m-1$. Let $q_1', \ldots, q_n'$ be the maximal collection of mixed cells of $\D|F_w$ such that $\Gamma^\varphi(\gamma_2^+) = \partial^- q_1'$ and $\Gamma^\varphi(\partial^+ q_j') = \partial^- q_{j+1}'$ for $j=1,2, \ldots, n-1$. First assume that these collections of mixed cells are nonempty, that is $m,n \geq 1$. By maximality and the fact that positive cells do not have arcs in $\Gamma^-$, there are negative cells $p', p''$ of $\D|F_w$ such that $\Gamma^\varphi(\partial^+ q_m)$  is a maximal negative arc in the boundary of $p'$ and $\Gamma^\varphi(\partial^+ q_n')$  is a maximal negative arc in the boundary of $p''$. The cells $q_1, \ldots, q_m, q_1', \ldots, q_n'$ must all be distinct and hence $\Gamma^\varphi(\partial^+ q_m)$, $\Gamma^\varphi(\partial^+ q_n')$ are distinct. Since $\Gamma(\varphi)$ sends internal endpoints of arcs to internal endpoints of arcs, we must have $p' = p''$. Thus we obtain a cell of $\D^\varphi$ composed of $p, q_1, \ldots, q_m, q_1', \ldots, q_n', p'=p''$. Such a cell has exactly one top vertex (coming from $p'=p''$) and exactly one bottom vertex (coming from $p$). 
If $m=0$ it suffices to replace $\Gamma^\varphi(\partial^+ q_m)$ with $\Gamma^\varphi(\gamma_1^+)$, and if $n=0$ it suffices to replace $\Gamma^\varphi(\partial^+ q_n')$ with $\Gamma^\varphi(\gamma_2^+)$, to still get a cell of $\D^\varphi$ with precisely one top vertex and precisely one bottom vertex.

It follows that every 2-cell of $\D^\varphi$ is either
\begin{itemize}
	\item composed of one internal cell of $\D|F_w$, or
	\item composed of one positive cell of $\D|F_w$, one negative cell of $\D|F_w$ and finitely many (potentially zero) mixed cells of $\D|F_w$, or
	\item composed of finitely many mixed cells of $\D|F_w$.
\end{itemize}
The last type of cells of $\D^\varphi$ have cyclically oriented edges in their boundary; see Figure~\ref{fig:loosing_tautness}. These are the only cells of $\D^\varphi$ that do not satisfy Definition \ref{def:taut}. Using the duality between positive/negative/mixed cells of $\D|F_w$ and positive/negative/mixed edges of $(\T|F_w, \alpha|F_w)$ it is easy to see that if for every mixed edge $e$ of $(T|F_w, \alpha|F_w)$ there is a positive edge $e^+$ of $(T|F_w, \alpha|F_w)$ and a negative edge $e^-$ of $(T|F_w, \alpha|F_w)$ such that $e, e^+, e^-$ are identified  in $\T^\varphi$  then $\D^\varphi$ does not admit such cells. Thus under this assumption $\alpha|F_w$ induces a taut structure on $\T^\varphi$.
\end{proof}

\begin{figure}[h]
\includegraphics[scale=0.8]{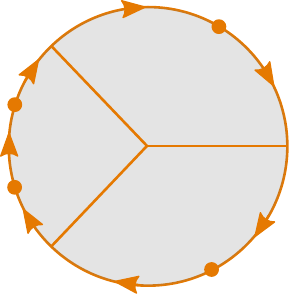}
\caption{Gluing mixed cells of $\D|F_w$ cyclically yields a cell of $\D^\varphi$ whose edges are cyclically oriented.}
\label{fig:loosing_tautness}
\end{figure}

We will devote the rest of this subsection to restate the condition of Lemma \ref{lem:sufficient:tautness} in terms of $\varphi$ and then prove that it is not only sufficient but also necessary for the existence of a taut structure on $\T^\varphi$. Recall that by $D(M_w)$ we denote the set of disks contained in the sutured annuli of $M_w$ arising from cutting $M|\Semb$ along $D_w$. Let $D(M_w  \bez P_w)$ be the subset of $D(M_w)$ consisting only of those disks which are not contained in the sutured annuli of triangular prisms of $M_w$. 
Lemma~\ref{lem:cut:triangulation:sutured:manifold} 
implies the following relationship between mixed edges of $(\T|F_w, \alpha|F_w)$ and disks $D(M_w \bez P_w)$ contained in the sutured annuli of $M_w-P_w$.
\begin{corollary}\label{cor:mixed-edges}
For every one sided-edge $e$ of $(\T|F_w, \alpha|F_w)$ there is precisely one \mbox{$D' \in D(M_w \bez P_w)$} such that $e = \coll(D')$.\qed 
\end{corollary}

However, we are mainly interested in the relationship between mixed edges of $(\T|F_w, \alpha|F_w)$ and edge product disks in $M|\Semb$.

\begin{definition}
Let $D', D'' \in D(M_w)$ be such that $\iota(D') = \iota(D'') = D \in D_w$. We say that the edge product disk $D$ 
\begin{itemize} \item \emph{has prisms on both sides} if there are triangular prisms $P', P'' \in P_w$ such that~$D'$ is contained in the sutured annulus of $P'$ and $D''$ is contained in the sutured annulus of~$P''$,
	\item \emph{has a prism on one side}  if exactly one disk out of $D', D''$ is  contained in the sutured annulus of some triangular prism  of $M_w$,
	\item \emph{does not have a prism on either side} if $D,D' \in D(M_w \bez P_w)$.
\end{itemize}
\end{definition}

Using this definition, we can restate Corollary \ref{cor:mixed-edges} and say that an edge product disk $D \in D_w$ corresponds to two, one or zero mixed edges of $(\T|F_w, \alpha|F_w)$ if and only if $D$ does not have a prism on either side, has a prism on one side or has prisms on both sides, respectively.

Let $V$ be a vertical annulus or a vertical M\"obius band in $M^\varphi$. Let $D_1, \ldots, D_N$ be edge product disks of $M|\Semb$ such that  $V$ is the quotient space of $D_1 \sqcup \ldots \sqcup D_N$ by~$\varphi$. We say that~$V$ \emph{lies in a prismatic region of $M^\varphi$} if $D_i$ has prisms on both sides for every $1\leq i \leq N$. 
Using this terminology we can now restate the assumption of Lemma \ref{lem:sufficient:tautness} in terms of topological properties of $M^\varphi$.

\begin{lemma}\label{lem:prismatic:vertical}
For every mixed edge $e$ of $(\T|F_w,\alpha|F_w)$ there is a positive edge $e^+$ of $(\T|F_w,\alpha|F_w)$ and a negative edge $e^-$ of $(\T|F_w,\alpha|F_w)$ such that $e, e^+, e^-$ are identified in $\T^\varphi$ if and only if every vertical annulus or M\"obius band in $M^\varphi$ lies in a prismatic region of $M^\varphi$.
\end{lemma}
\begin{proof}
First observe that if $e$ is an edge of triangle from $F_w^+$ then the the edge $r(\varphi)(e)$ might not be well-defined, because a positive edge can be mapped to two different mixed edges. However, if $e$ is a mixed edge then $r(\varphi)(e)$ is well defined, because $e$ is an edge of only one $f^+\in F_w^+$.

In what follows the subscript $i$ is taken modulo~$n$. Suppose that there is a collection $e_1, \ldots, e_n$ of mixed edges of $(\T|F_w, \alpha|F_w)$ such that $r(\varphi)(e_i) = e_{i+1}$ for every \mbox{$1\leq i \leq n$}. By Corollary~\ref{cor:mixed-edges}, there is a collection $D_1', \ldots, D_n' \in D(M_w \bez P_w)$ of disks contained in the sutured annuli of $M_w- P_w$ such that $\coll(D_i') = e_i$. Let \mbox{$D_i = \iota(D_i') \in D_w$}. By the definition of $r(\varphi)$ for every $1\leq i \leq n$ either $\varphi(\partial^+ D_i) = \partial^- D_{i+1}$ or there is $k_i \geq 1$ and a collection of edge product disks $D_i^{(1)}, \ldots, D_i^{(k_i)} \in D_w$  (which all have prisms on at least one side) such that $\varphi(\partial^+D_i) = \partial^- D_i^{(1)}$, $\varphi(\partial^+ D_i^{(j)}) = \partial^- D_i^{(j+1)}$ for every $1\leq j \leq k_i - 1$ and $\varphi(\partial^+ D_i^{(k_i)}) = \partial^- D_{i+1}$. Therefore the quotient of $\bigsqcup\limits_{i=1}^n D_i \sqcup \bigsqcup\limits_{j=1}^{k_i} D_i^{(j)}$ by $\varphi$ is a vertical annulus or M\"obius band~$V$ in $M^\varphi$. Since $D_i' \in D(M_w \bez P_w)$, $D_i$ does not have prisms on both sides. It follows that $V$ does not lie in a prismatic region of~$M^\varphi$.

Now, let $D_1, \ldots, D_n \in D_w$ be a collection of edge product disks in $M|\Semb$ such that $\varphi(\partial^+ D_i) = \partial^- D_{i+1}$ for every $1\leq i \leq n$. Denote by $V$ the quotient space of $D_1 \sqcup \ldots \sqcup D_n $ by $\varphi$. If $V$ does not lie in a prismatic region of $M^\varphi$ there is $1\leq k \leq n$ and a sequence $1\leq i_1 <i_2 <\ldots<i_k\leq n$ such that $D_{i_j}$ does not have prisms on both sides for every $1\leq j \leq k$, and for every $l \notin \lbrace i_1, i_2, \ldots, i_k \rbrace$ the edge product disk~$D_l$ has prisms on both sides. For $1\leq j \leq k$ if $D_{i_j}$ does not have a prism on either side there are disks $D_{i_j}', D_{i_j}'' \in D(M_w \bez P_w)$ such that $\iota(D_{i_j}') = \iota(D_{i_j}'') = D_{i_j}$. If $D_{i_j}$ has a prism on one side, there is $D_{i_j}'\in D(M_w \bez P_w)$ such that $\iota(D_{i_j}') = D_{i_j}$. By Corollary~\ref{cor:mixed-edges} there are mixed edges $e_{i_j} = \coll(D_{i_j}')$, $e_{i_j}' = \coll(D_{i_j}'')$ of $(\T|F_w, \alpha|F_w)$. By the definition of $r(\varphi)$, after possibly switching $e_{ij}$ with $e_{i_j}'$ for some $1\leq j \leq k$, there is $m\geq 1$ and a sequence $1\leq l_1 < l_2 <\ldots <l_m \leq k$ such that $r(\varphi)(e_{i_{l_j}}) = e_{i_{l_{j+1}}}$ for every $1\leq j \leq m$. 
This gives an edge of $\T^\varphi$ which is composed entirely of mixed edges of $(\T|F_w,\alpha|F_w)$.
\end{proof}

\begin{proposition}\label{prop:taut:sufficient:necessary}
Ideal triangulation $\T^\varphi$ admits a taut  structure if and only if every vertical annulus or M\"obius band in $M^\varphi$ lies in a prismatic region of $M^\varphi$.
\end{proposition}

\begin{proof}

The fact that when every vertical annulus or M\"obius band in $M^\varphi$ lies in a prismatic region of $M^\varphi$ then $\T^\varphi$ admits a taut structure follows from Lemmas \ref{lem:sufficient:tautness} and~\ref{lem:prismatic:vertical}.

We prove the other direction by contraposition. Suppose that there is a vertical annulus or a M\"obius band $V$ in $M^\varphi$ which does not lie in a prismatic region of $M^\varphi$. Let $T$ be a boundary torus of $M$ such that $T\cap V \neq \emptyset$. Since $V$ does not consist entirely of edge product disks which have prisms on both sides, $T$ does not consist entirely of prismatic $D_w$-rectangles. It follows from the proof of Theorem \ref{thm:mutant:correct:manifold} (Case~1) that $M^{r(\varphi)}$, the manifold underlying $\T^\varphi$, admits a spherical boundary component. Therefore, by Lemma \ref{lem:taut:torus:boundary}, $M^{r(\varphi)}$ does not have a taut triangulation. In particular, $\T^\varphi$ does not admit a taut structure.
\end{proof}

It follows that $\T^\varphi$ admits a taut structure if and only if orientations on the edges of the dual spine $\D^\varphi$ inherited from $(\T|F_w, \alpha|F_w)$ determine a taut structure on $\T^\varphi$. In this case we denote the taut structure on $\T^\varphi$  by $\alpha^\varphi$. We also say that $(\T, \alpha)$ and $(T^\varphi, \alpha^\varphi)$ are \emph{taut mutants}.

The assumption that $\varphi$ misaligns edge product disks is stronger than the assumption that every vertical annulus or M\"obius band in $M^\varphi$ lies in a prismatic region of $M^\varphi$. However, since it is a necessary condition for  $\T^\varphi$ to be an ideal triangulation of $M^\varphi$ (Theorem~\ref{thm:mutant:correct:manifold}), it is reasonable to assume this stronger condition for the rest of the paper. Below we also prove that when $\varphi$ aligns edge product disks then $M^\varphi$ does not admit a veering triangulation. This further justifies restricting our considerations only to automorphisms which misalign edge product disks.

\begin{proposition}\label{prop:not:hyperbolic} 
If $\varphi \in \Aut^+(\QVw)$ aligns edge product disks then $M^\varphi$ does not admit a veering triangulation.
\end{proposition}
\begin{proof}
If $\varphi$ aligns edge product disks then $M^\varphi$ admits either an annulus or a M\"obius band joining two boundary components of $M^\varphi$ or one boundary component of $M^\varphi$ to itself. Consequently, $M^\varphi$ either contains an essential annulus or is a Seifert fibered space. In either case, $M^\varphi$ is not hyperbolic. Since veering triangulations admit strict angle structures \cite[Theorem 1.5 (stated here as Theorem \ref{thm:HRST})]{veer_strict-angles}, they can live only on hyperbolic manifolds. Thus $M^\varphi$ does not admit a veering triangulation. 
\end{proof}

\begin{remark}\label{remark:Pachner}
If $\varphi$ aligns edge product disks, we could consider a certain modification of a taut mutation that relies on applying 0-2 Pachner moves to get rid of edges of weight greater than one; see Figure~\ref{fig:Pachner} for an illustration of a 0-2 Pachner move. Suppose that a weight system~$w$ on a taut triangulation $(\T,\alpha)$ of $M$ is such that some edges of $(\T, \alpha)$ have weight greater than one.  That is, the set $D_w$ of edge product disks in $M|\Semb$ is nonempty. Let $Q$ be the induced triangulation on $S_w$. We can perform finitely many 0-2 Pachner moves above pairs of triangles adjacent to edges with weight greater than one to get a taut triangulation $(\T_\ast, \alpha_\ast)$ of $M$ with the following properties:
\begin{itemize}
	\item $(\T_\ast, \alpha_\ast)$ carries a surface $S_{w_\ast}$ whose induced triagulation $\Q_\ast$ is combinatorially isomorphic to $\Q$,
	\item the weight system $w_\ast$ on $(\T_\ast, \alpha_\ast)$ is such that no edge of $(\T_\ast, \alpha_\ast)$ has weight greater than one. 
\end{itemize}  Then for every $\varphi_\ast \in \Aut^+(\Q_\ast)$ there is $\varphi \in \Aut^+(\Q)$ such that $M^\varphi =M^{\varphi_\ast}$. Since $M|S_{w_\ast}^\epsilon$ does not admit any edge product disks, every $\varphi_\ast \in \Aut^+(\Q_\ast)$ misaligns edge product disks. Therefore, by Proposition \ref{prop:taut:sufficient:necessary} and Theorem \ref{thm:mutant:correct:manifold}, we can use $\varphi_\ast$ to construct a taut triangulation $(\T_\ast^{\varphi_\ast}, \alpha_\ast^{\varphi_\ast})$ of $M^\varphi =M^{\varphi_\ast}$.  

It follows that applying 0-2 Pachner moves before a mutation can solve both the problem of the lack of tautness after the mutation as well as the problem with the wrong homeomorphism type after the mutation. 
We do not consider this modification of the taut mutation, because our focus is on veering triangulations and Proposition~\ref{prop:not:hyperbolic} implies that when $\varphi$ aligns edge product disks $M^\varphi$ cannot admit a veering triangulation. (Note that even though $\varphi^\ast$ does not align \emph{edge} product disks, it aligns different types of product disks in $M|S_{w_\ast}^\epsilon$.)
\end{remark}

\begin{figure}[h]
\includegraphics[scale=1.2]{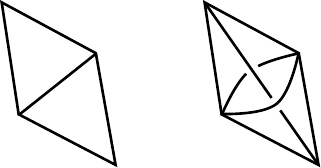}
\put(-65,30){$\longrightarrow$}
\caption{0-2 Pachner move. If $f, f'$ are adjacent along $e$ we can replace $f\cup f'$ by a union of two tetrahedra glued along a pair of faces obtained from $f \cup f'$ by a diagonal exchange. If before the move the triangulation is taut and $f$, $f'$ are on different sides of the common edge then the triangulation after the move is taut as well.}
\label{fig:Pachner}
\end{figure}

From now on we assume that $\varphi$ misaligns edge product disks.  The taut triangulation $(\T^\varphi, \alpha^\varphi)$ is veering if and only if its dual spine $\D^\varphi$ has a smoothening into a branched surface which locally looks like in Figure \ref{fig:veering_branched_surface}; see Definition \ref{def:veering}. One way of ensuring that is to construct the required branched surface structure on $D^\varphi$ using the branched surface structure $\B|F_w$ on $\D|F_w$. 
This, however, is possible only if for every $f^+ \in F_w^+$ the regluing map $r(\varphi)$ maps the large edge of $f^+$ to the large edge of $r^\varphi(f^+)$. For this reason we define the group $\Aut^+(\QVw \ | \ \tau_{\V, w})$ of orientation-preserving combinatorial automorphisms of $\Q_{\V,w}$ which preserve $\tau_{\V, w}$. 
\begin{lemma}\label{lem:train:track}
Let $\varphi \in \Aut^+(\QVw \ | \ \tau_{\V, w})$.
If $e$ is the large edge of $f ^+\in F_w^+$ then $r^\varphi_{f^+}(e)$ is the large edge of $r^\varphi(f^+) \in F_w^-$.
\end{lemma}
\begin{proof}
Suppose that $g^\varphi(f) = (g_1, \ldots, g_k)$. Then $g_1 = \varphi(L(f))$, $g_{i+1} = \varphi(\ab(g_i))$ for \mbox{$i =1, \ldots, k-1$}, and $r^{\varphi}(f^+)$ is the triangle $f'^-\in F_w^-$ such that $U(f') = g_k$. Furthermore, $r(\varphi)$ maps $e$ to $\left( \left(\sigma_{f'}^U\right)^{-1}\circ  \left(\varphi_{i}\circ \delta_i\right)_{i=1}^{k-1} \circ \varphi_{\check{f}} \circ \sigma_f^L\right)(e)$. We refer the reader to Subsections  \ref{subsec:surface:tri} and \ref{subsec:mutating} to recall the notation.

The large edge of $L(f)$ is given by $\sigma_f^L(e)$. Since $\varphi \in \Aut^+(\QVw \ | \ \tau_{\V, w})$ we get that $\left(\varphi_{\check{f}} \circ \sigma_f^L\right)(e)$ is the large edge of $g_1$. By the definition of $\delta_i$,  $\left(\delta_1 \circ \varphi_{\check{f}} \circ \sigma_f^L\right)(e)$ is the large edge of $\ab(g_1)$. Again, the assumption that $\varphi \in \Aut^+(\QVw \ | \ \tau_{\V, w})$ implies that $\left(\varphi_1 \circ \delta_1 \circ \varphi_{\check{f}} \circ \sigma_f^L\right)(e)$ is the large edge of $g_2$. Continuing this way, we get that $\left( \left(\varphi_{i}\circ \delta_i\right)_{i=1}^{k-1} \circ \varphi_{\check{f}} \circ \sigma_f^L\right)(e)$ is the large edge of $g_k$, and thus $\left( \left(\sigma_{f'}^U\right)^{-1}\circ  \left(\varphi_{i}\circ \delta_i\right)_{i=1}^{k-1} \circ \varphi_{\check{f}} \circ \sigma_f^L\right)(e)$ is the large edge of $r^\varphi(f^+) = f'^-$.
\end{proof}

The above lemma gives a sufficient condition for when the dual spine of $\T^\varphi$ admits a smoothening into a branched surface. Combining it with Proposition \ref{prop:taut:sufficient:necessary} gives sufficient conditions for the existence of a veering structure on $\T^\varphi$. 
\begin{theorem}\label{thm:mutant:veering}
Let $S_w$ be a surface carried by a veering triangulation $\V = (\T, \alpha, \B)$ of $M$.  Suppose that $\varphi \in \Aut^+(\QVw)$ misaligns edge product disks. If additionally \linebreak $\varphi \in \Aut^+(\QVw \ | \ \tau_{\V, w})$ then $(\T^\varphi, \alpha^\varphi)$ admits a veering structure.
\end{theorem}
\begin{proof}
By Proposition \ref{prop:taut:sufficient:necessary},  the assumption that $\varphi$ misaligns edge product disks implies the existence of a taut structure $\alpha^\varphi$ on $\T^\varphi$. If $\varphi \in \Aut^+(\QVw \ | \ \tau_{\V, w})$ then, by Lemma~\ref{lem:train:track}, for every $f^+\in F_w^+$ the regluing map $r(\varphi)$ maps the large edge of $f^+$ to the large edge of $r^\varphi(f^+)$.
Thus the structure of a branched surface on $\B|F_w$ gives a structure of a branched surface on the dual spine $\D^\varphi$ of $(\T^\varphi, \alpha^\varphi)$. We denote this branched surface by ~$\B^\varphi$. For every tetrahedron $t$ of $\T^\varphi$ the branched surface $\B^\varphi_t = \B^\varphi \cap t$ looks  as in Figure \ref{fig:veering_branched_surface}, because there is a tetrahedron $t'$ of $\V$ such with $\B_{t'} = \B\cap t' = B^\varphi_t$ and $(\T, \alpha, \B)$ is veering. Thus $\B^\varphi$ satisfies Definition \ref{def:veering} and $\V^\varphi = (\T^\varphi, \alpha^\varphi, \B^\varphi)$ is veering.\end{proof}

We say that $\V^\varphi = (\T^\varphi, \alpha^\varphi, \B^\varphi)$ is obtained from $\V = (\T, \alpha, \B)$ by a \emph{veering mutation} or that $\V^\varphi$, $\V$ are \emph{veering mutants}. For instance, the first two veering triangulations in the Veering Census, the veering triangulation \texttt{cPcbbbdxm\_10} of the figure eight knot sister (manifold m003 in the SnapPy Census \cite{snappea}) and the veering triangulation \texttt{cPcbbbiht\_12} of the figure eight knot (m004), are veering mutants.

\begin{remark}\label{remark:no:bijection:needed}
By Lemma \ref{lem:train:track}, a combinatorial automorphism $\varphi \in \Aut^+(\QVw \ | \ \tau_{\V, w})$ is uniquely determined by the associated bijection $\varphi: F_{\V,w} \rightarrow F_{\V,w}$. For this reason, when discussing examples of veering mutants in Section \ref{sec:faces} we will not  label the vertices of tetrahedra nor talk about bijections between vertices of identified triangles.
\end{remark}

Theorem \ref{thm:mutant:veering} gives a sufficient condition for veeringness of a taut mutant, but this condition is not necessary. It is possible that $(\T^\varphi, \alpha^\varphi)$ admits a veering structure even though $\varphi \notin \Aut^+(\QVw \ | \ \tau_{\V, w})$. We discuss this possibility briefly, and give an example of this phenomenon, in the next subsection.
\subsection{Generalizations}\label{subsec:generalization}
We say that $(\T|F_w, \alpha|F_w)$ admits a veering structure if it is possible to smoothen its dual spine $\D|F_w$ into a branched surface which locally around every vertex looks either as in Figure \ref{fig:veering_branched_surface}(a) or as in Figure \ref{fig:veering_branched_surface}(b). Suppose that $(\T|F_w, \alpha|F_w)$ admits a veering structure $\B^\ast|F_w$. Let $t$ be a tetrahedron of $(\T|F_w, \alpha|F_w)$. Let $\B^\ast_t = \B^\ast|F_w \cap t$. By $-\B_t^\ast$ we denote the other possible veering structure on $t$; see Figure \ref{fig:veering_branched_surface} to see the two options.
Lemma \ref{lem:large:edges} implies that if $t$ has a top face $f$ which is a bottom face of some tetrahedron of $(\T|F_w, \alpha|F_w)$ then we cannot change the veering structure on $t$ from $\B^\ast_t$ to $-\B^\ast_t$ without destroying veeringness. On the other hand, if both top faces of $t$ are in $F_w^+$ then we can freely change $\B^\ast_t$ to $-\B^\ast_t$ and the resulting branched surface still defines a veering structure on $(\T|F_w, \alpha|F_w)$. 

Let $\tau_{\V, w}^{\ast \ +}$, $\tau_{\V, w}^{\ast \ -}$ be the train tracks in $\QVw^+, \QVw^-$, respectively, induced by~$\B^\ast|F_w$. Using the same arguments as in the proof of Theorem \ref{thm:mutant:veering} we can show that if \mbox{$\varphi \in \Aut^+(\QVw)$} misaligns edge  product disks and sends $\tau_{\V, w}^{\ast \ +}$ to $\tau_{\V, w}^{\ast \ -}$ then the veering structure $\B^\ast|F_w$ on $(\T|F_w, \alpha|F_w)$ glues up into a veering structure on $(\T^\varphi, \alpha^\varphi)$. The advantage of considering this more general setup is that now we can derive both sufficient and necessary conditions for veeringness of a taut mutant.

\begin{theorem}\label{thm:generalized}
Let $S_w$ be a surface carried by a veering triangulation $\V = (\T, \alpha, \B)$ of~$M$.  Suppose that $\varphi \in \Aut^+(\QVw)$ misaligns edge product disks. The taut triangulation $(\T^\varphi, \alpha^\varphi)$ admits a veering structure if and only if there is a veering structure $\B^\ast|F_w$ on $(\T|F_w, \alpha|F_w)$ such that the isomorphism $\varphi: \QVw^+ \rightarrow \QVw^-$ sends $\tau_{\V, w}^{\ast \ +}$ to $\tau_{\V, w}^{\ast \ -}$.
\end{theorem}
\begin{proof}
The backward direction can be proved exactly as Theorem \ref{thm:mutant:veering}. If $(\T^\varphi, \alpha^\varphi)$ has a veering structure $\B^\ast$ then $(\T|F_w, \alpha|F_w)$ must have a veering structure $\B^\ast|F_w$ such that $r(\varphi)$ sends the train track on $F_w^+$ induced by $\B^\ast|F_w$ to the  train track induced by $\B^\ast|F_w$ on $F_w^-$. 
Since $\varphi$ misaligns edge product disks, for every $g \in \QVw$ we have a trichotomy: $g \in L(F_w)$, $g$ appears in $g^\varphi(f)$ for some $f \in F_w$, or $\ab(g)$ appears in $g^\varphi(f)$ for some $f \in F_w$. Equivalently,  when constructing $r(\varphi)$ from $\varphi$ we have passed through every triangular prism of $M_w$.  Therefore Lemma \ref{lem:correct:gluing} implies that $\varphi$ sends $\tau_{\V, w}^{\ast \ +}$ to $\tau_{\V, w}^{\ast \ -}$.
\end{proof}

The more general setup of Theorem \ref{thm:generalized} is not just theoretical.  There are veering triangulations $(\T, \alpha, \B)$, $(\T^\varphi, \alpha^\varphi, \B')$ which are taut mutants but not veering mutants. One such pair is given by the veering triangulation \texttt{gLMzQbcdefffhhhqxdu\_122100} of the manifold s463 and the veering  triangulation \texttt{gLMzQbcdefffhhhqxti\_122100} of the manifold~s639.

Another generalization we might consider is a \emph{veering mutation with insertion}. Let $(\T|F_w, \alpha|F_w, \B|F_w)$ be a veering cut triangulation.  If there are two triangles \mbox{$f_1^+, f_2^+ \in F_w^+$} which are adjacent  along an edge which is large in both $f_1^+$ and $f_2^+$, we might stack another veering tetrahedron on top of $f_1^+\cup f_2^+$. We then obtain another cut triangulation with a veering structure which has more tetrahedra than $\T|F_w$. We can also add a new veering tetrahedron on top of two faces $f_1^+, f_2^+ \in F_w^+$ whose large edges ore mixed. 

Suppose that $(\T^\ast|F_{w_\ast}, \alpha^\ast|F_{w_\ast}, \B^\ast|F_{w_\ast})$ is obtained from $(\T|F_w, \alpha|F_w, \B|F_w)$ by additing finitely many veering tetrahedra on top of $F_w^+$. If there is a map $r: F_{w_\ast}^+ \rightarrow F_{w_\ast}^-$ such that identifying $F_{w_\ast}^+$ with $F_{w_\ast}^-$ yields a veering triangulation $\V^r = (\T^r, \alpha^r, \B^r)$ then we say that $\V^r$ is obtained from $\V$ by a \emph{veering mutation with insertion}. For instance, the veering triangulation \texttt{dLQbccchhfo\_122} of the manifold m009 is obtained from the veering triangulation \texttt{cPcbbbiht\_12} of the manifold m004 (figure eight knot complement) by a veering mutation with insertion.

\section{Homeomorphic veering mutants}\label{sec:faces}
Manifold $M$ and its mutant $M^\varphi$ can be homeomorphic. This can happen for instance for many graph manifolds mutated along one of their decomposing tori.
If a pair of veering mutants $\V, \V^\varphi$ live on the same manifold, they might be  combinatorially isomorphic or combinatorially distinct. For instance, the veering triangulation \texttt{eLMkbcddddedde\_2100} of the $6^2_2$ link complement carries a four times punctured sphere such that mutating the triangulation along it via an involution yields \texttt{eLMkbcddddedde\_2100} back.  We discuss a few examples of combinatorially distinct veering mutants of the same manifold in Subsections \ref{subsec:same:face}, \ref{subsubsec:higher:betti}, and \ref{sec:higher:genus}.

Recall from Theorem \ref{thm:LMT:faces} that veering triangulations combinatorially represent faces of the Thurston norm ball. 
A pair of veering mutants on a 3-manifold $M$ may represent either the same face or different faces of the Thurston norm ball in $H_2(M, \partial M;\rr)$. 
The main  obstacle to finding examples of measurable veering mutants which represent the same face of the Thurston norm ball is that when $b_1(M) > 1$ there are infinitely many distinct bases for $H_2(M, \partial M;\zz)$. 
While it is relatively easy to find the cones $\mathcal{C}(\V)$, $\C(\V^\varphi)$ of homology classes of surfaces carried by $\V, \V^\varphi$, respectively (this is explained in \cite[Section 11.2]{parlak-thesis}), it is not always straightforward to figure out whether they are the same up to a change of basis.

The above problem does not appear in the $b_1(M) = 1$ case. Then, up to $\eta \mapsto -\eta$, there is only one (0-dimensional) face of the Thurston norm ball in $H_2(M, \partial M;\rr)$. If~$M$ admits a pair of veering mutants $\V$, $\V^\varphi$ then neither $\mathcal{C}(\V)$ nor $\C(\V^\varphi)$ is empty. Hence $\V, \V^\varphi$ must combinatorially represent the same face of the Thurston norm ball; see Remark \ref{remark:not:opposite}. When $b_1(M)>1$ it is sometimes possible to verify if $\C(\V) = \C(\V^\varphi)$ using the combinatorics of the Thurston norm ball. We do this in Subsections \ref{subsubsec:higher:betti} (where the faces are the same) and \ref{sec:higher:genus} (where the faces are different).

\begin{remark}\label{remark:not:opposite}
Recall that  veering triangulations  come in pairs $\V, -\V$ having the same taut signature and representing opposite faces of the Thurston norm ball; see Remarks~\ref{remark:two:veering:tris} and \ref{remark:faces:and:automorphisms}. Since  the veering mutant $\V^\varphi = (\T^\varphi, \alpha^\varphi, \B^\varphi)$ inherits coorientations on faces from $\V = (\T, \alpha, \B)$, and these coorientations determine  orientations of carried surfaces, we cannot have $\C(\V^\varphi) = -\C(\V)$. 
\end{remark}

In this section we will establish the following facts connecting veering mutations and faces of the Thurston norm ball.
\pagebreak
\begin{fact}\label{thm:mutation:properties} \emph{(Veering mutations and faces of the Thurston norm ball)}\nopagebreak 
\begin{enumerate}
	\item A non-fibered face $\face{F}$ of the Thurston norm ball of a compact, oriented, hyperbolic 3-manifold with boundary can be represented by two combinatorially non-isomorphic veering mutants.
	\item Performing a veering mutation along a surface representing a class lying at the boundary of a fibered face may yield a veering triangulation representing a non-fibered face of the Thurston norm ball of the mutant manifold.
\end{enumerate}
\end{fact}
\begin{proof}
In Subsection  \ref{subsec:same:face} we discuss four veering mutants $\V, \V^\varrho, \V^\sigma, \V^{\varrho\sigma}$ such that~$\V$ and $\V^{\varrho\sigma}$ represent the same non-fibered face of the Thurston norm ball in a certain manifold~$M$ with $b_1(M)=1$, and $\V^\varrho, \V^\sigma$ represent  adjacent fibered faces of \mbox{$M^\ro \cong M^\sigma$} with $b_1(M^\ro) = 2$. Triangulations $\V, \V^{\varrho\sigma}$ prove (1) in $b_1(M) = 1$ case, while triangulations $\V^\varrho, \V^{\sigma\varrho}$ prove~(2); see also Proposition \ref{prop:mutant:boundary}. Veering mutants proving (1) in the case $b_1(M)>1$ are discussed in Subsection~\ref{subsubsec:higher:betti}. 
\end{proof}

\subsection{Two veering mutants representing the same face of the Thurston norm ball when $b_1(M) =1$}\label{subsec:same:face}
Let $M$ be the manifold t12488 from the SnapPy census. This manifold is not fibered and $H_1(M;\zz) = \zz \oplus \zz/8$. It also admits a pair of distinct measurable veering triangulations which, since $b_1(M) = 1$, must represent the same face of the Thurston norm ball. We will show that they differ by a veering mutation.

Let $\V$ be a veering triangulation of $M$ with the taut signature \[\texttt{iLLLPQccdgefhhghqrqqssvof\_02221000}. \]
We present the tetrahedra of $\V$ in Figure \ref{fig:svof}. 
\begin{figure}[h]
\includegraphics[scale=0.7]{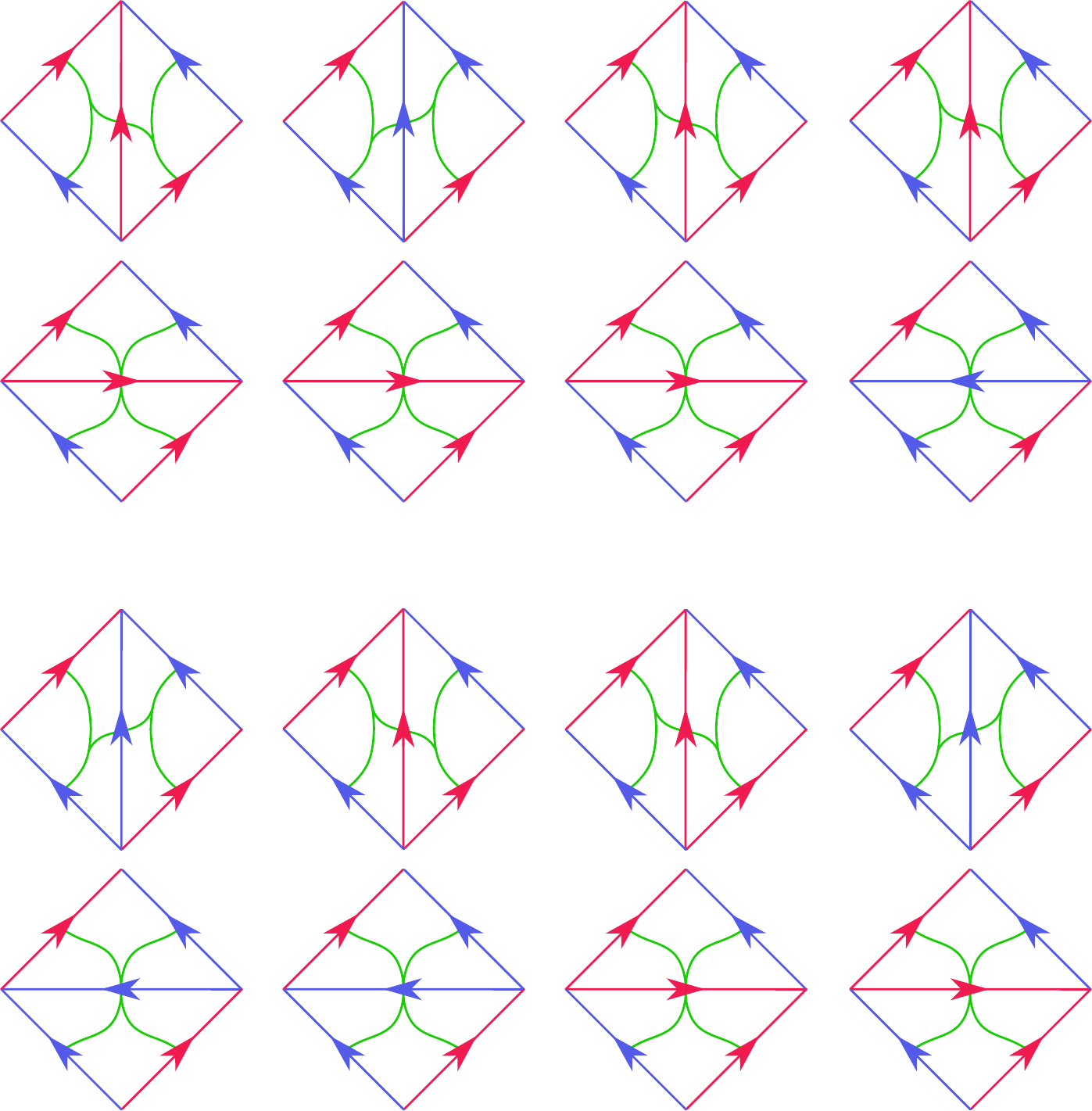}
\put(-254,295){0}
\put(-181, 295){1}
\put(-108, 295){2}
\put(-35, 295){3}
\put(-254,138){4}
\put(-181, 138){5}
\put(-108, 138){6}
\put(-35, 138){7}
\put(-275, 253){$f_3$}
\put(-238, 253){$f_2$}
\put(-202, 253){$f_4$}
\put(-165, 253){$f_5$}
\put(-128, 253){$f_1$}
\put(-92, 253){$f_8$}
\put(-55, 253){$f_9$}
\put(-18, 253){$f_0$}
\put(-256, 205){$f_0$}
\put(-256, 168){$f_1$}
\put(-183, 205){$f_6$}
\put(-183, 168){$f_3$}
\put(-110, 205){$f_2$}
\put(-110, 168){$f_7$}
\put(-37, 205){$f_4$}
\put(-38, 168){$f_{10}$}
\put(-277, 97){$f_{12}$}
\put(-240, 97){$f_{10}$}
\put(-204, 97){$f_{13}$}
\put(-165, 97){$f_6$}
\put(-128, 97){$f_7$}
\put(-94, 97){$f_{15}$}
\put(-57, 97){$f_{11}$}
\put(-20, 97){$f_{14}$}
\put(-258, 47){$f_{11}$}
\put(-256, 11){$f_5$}
\put(-185, 47){$f_{12}$}
\put(-185, 11){$f_{14}$}
\put(-110, 47){$f_8$}
\put(-112, 11){$f_{13}$}
\put(-39, 47){$f_{15}$}
\put(-38, 11){$f_{9}$}

\caption{Veering triangulation \texttt{iLLLPQccdgefhhghqrqqssvof\_02221000} of the manifold t12488.}
\label{fig:svof}
\end{figure}

By solving the system of branch equations associated to $\V$ one can verify that~$\V$ carries four surfaces that can be expressed as the following (relative) 2-cycles (which we identify with the induced triangulations):
\begin{align*}
\Q_0 &= f_2 + f_5 + f_7 + f_{11}, \\
\Q_1 &= f_1 + f_5  + f_8 + f_{11}, \\
\Q_2 &= f_2 + f_7 + f_{10} + f_{12},\\
\Q_3 &= f_1 + f_8 + f_{10} + f_{12}.
\end{align*}
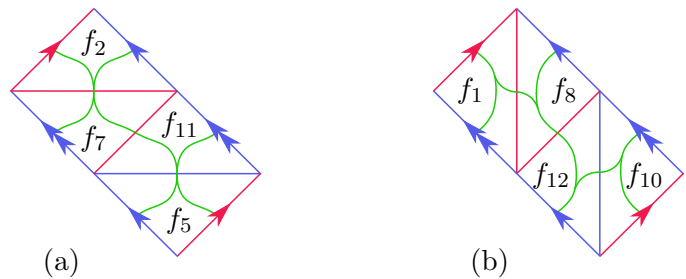
\begin{figure}[h]
\includegraphics[scale=0.7]{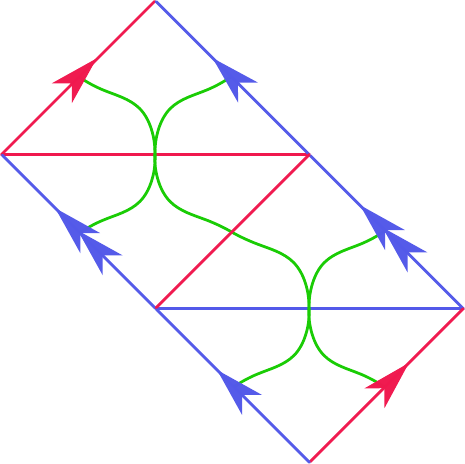}
\hspace{2cm}
\includegraphics[scale=0.7]{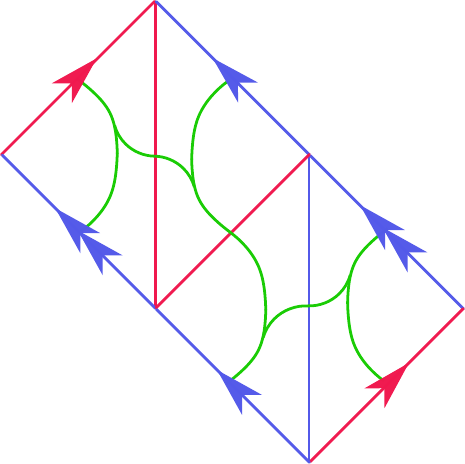}
\put(-240,-5){(a)}
\put(-80,-5){(b)}
\put(-226, 78){$f_2$}
\put(-226, 42){$f_7$}
\put(-196, 47){$f_{11}$}
\put(-194, 11){$f_5$}
\put(-85, 60){$f_1$}
\put(-50, 60){$f_{8}$}
\put(-57, 28){$f_{12}$}
\put(-22, 28){$f_{10}$}

\caption{(a) Triangulation $\Q_0 = f_2 + f_5 + f_7 + f_{11}$ and its dual stable track $\tau_0$. The group $\mathrm{Aut}^+(\Q_0 \ | \  \tau_0)$ is generated by a rotation $\ro$ by $\pi$ about the center of the red edge between faces $f_7$ and $f_{11}$ and a `shift by one square' map $\sigma$. (b) Triangulation $\Q_3= f_1 + f_8  + f_{10} + f_{12}$ and its dual stable track~$\tau_3$. The group $\mathrm{Aut}^+(\Q_3 \ | \  \tau_3)$ is generated by $\rho\sigma$ only.}
\label{fig:svof_surfaces}
\end{figure}

All these surfaces are twice punctured tori.  Triangulations $\Q_0$ and $\Q_3$ are presented in Figure \ref{fig:svof_surfaces}. Since the first Betti number of $M$ is equal to one, these punctured tori are homologous. In fact, it is easy to see that they are all homotopic. $\Q_0$ consists of two bottom faces of tetrahedron 2 and two bottom faces of tetrahedron~4. By performing the diagonal exchange corresponding to tetrahedron 2, one obtains triangulation $\Q_1$. By performing the diagonal exchange corresponding to tetrahedron 4, one obtains triangulation $\Q_2$. Triangulation $\Q_3$ can be obtained from $\Q_0$ by performing diagonal exchanges through both tetrahedra 2 and 4. 

Let $\tau_i$ be the train track dual to $\Q_i$ induced by the stable train track of $\V$. That is, $\tau_i = \tau_{\V, w_i}$ where $w_i$ is the weight system on~$\V$ determining $\Q_i$. Let $\mathrm{Aut}^+(\Q_i \ | \  \tau_i)$ be the group of orientation-preserving combinatorial automorphisms of $\Q_i$ which preserve~$\tau_i$.
Then
\begin{align*}
\mathrm{Aut}^+(\Q_0 \ | \  \tau_0) &= \zz/2 \oplus \zz/2, \\
\mathrm{Aut}^+(\Q_1 \ | \  \tau_1) &= \zz/2, \\
\mathrm{Aut}^+(\Q_2 \ | \  \tau_2)&= \zz/2,\\
\mathrm{Aut}^+(\Q_3 \ | \  \tau_3)&= \zz/2.
\end{align*}
The group $\mathrm{Aut}^+(\Q_0 \ | \  \tau_0)$ is generated by a rotation by $\pi$ about the center of the red edge between faces $f_7$ and $f_{11}$ and a `shift by one square' map; see Figure \ref{fig:svof_surfaces} (a). We denote these combinatorial isomorphisms of $\Q_0$ by $\varrho$ and~$\sigma$, respectively. 
For $i=1,2,3$ the group $\mathrm{Aut}^+(\Q_i \ | \  \tau_i)$ is generated by $\ro\sigma$; see Figure \ref{fig:svof_surfaces}~(b) for the $i=3$ case. 
The fact that $\mathrm{Aut}^+(\Q_0 \ | \  \tau_0)$ is the largest is not surprising; the surface underlying $\Q_0$ is the lowermost carried representative of the generator of $H_2(M, \partial M;\zz)$, and thus the stable train track $\tau_0$ has the most large branches (is minimally splitted).

Since $\Q_0$ does not traverse any edge of $\V$ more than once we  automatically get that neither of $\varrho, \sigma, \varrho\sigma$ aligns edge product disks. Thus, by Theorem \ref{thm:mutant:veering}, we get three veering mutants $\V^\varrho$, $\V^\sigma$, $\V^{\varrho\sigma}$ of $\V$.   Information about the regluing maps $r(\varrho), r(\sigma), r(\varrho\sigma)$ is presented in Table \ref{tab:svof:regluing}. Recall from Remark \ref{remark:no:bijection:needed} that since we are looking only at elements of $\Aut^+(\Q_0 \ | \ \tau_0)$ the regluing maps are uniquely determined by their associated bijections $\left\lbrace f_2^+, f_5^+, f_7^+, f_{11}^+\right\rbrace \rightarrow \left\lbrace f_2^-, f_5^-, f_7^-, f_{11}^-\right\rbrace$. In Figure \ref{fig:4mutants} we present taut signatures of the four mutants as well some additional information about their underlying manifolds.

\begin{table}[h]
\begin{tabular}{|c||c|c|c|c|}
	\hline
	&&&& \\[-1em]
	$f^+$ & $f_2^+$ & $f_5^+$ & $f_7^+$ & $f_{11}^+$ \\ 
	&&&& \\[-1em] 
	\hline
	&&&& \\[-1em]
	$r^\varrho(f^+)$ & $f_5^-$ & $f_2^-$ & $f_{11}^-$ & $f_7^-$\\
	&&&& \\[-1em] 
	\hline
	&&&& \\[-1em]
	$r^\sigma(f^+)$ & $f_{11}^-$ & $f_7^-$ & $f_{5}^-$ & $f_2^-$\\
	&&&& \\[-1em] 
	\hline
	&&&& \\[-1em]
	$r^{\varrho\sigma}(f^+)$ & $f_7^-$ & $f_{11}^-$ & $f_2^-$ & $f_5^-$\\
	\hline
\end{tabular}
\vspace{0.2cm}
\caption{The regluing maps determined by $\varrho, \sigma$ and $\varrho\sigma$.}
\label{tab:svof:regluing}\end{table}

\begin{figure}[h]
\begin{tikzcd}[column sep=large, row sep = large]
	\begin{matrix}\texttt{iLLLPQccdgefhhghqrqqssvof}\\
		\texttt{02221000}\\ t12488\\
		\text{edge-orientable}\\
		\text{measurable}\\
		\zz \oplus \zz/8\end{matrix} \arrow[r,"\varrho"] \arrow[d, "\sigma", swap] &\begin{matrix}\texttt{ivLLQQccfhfeghghwadiwadrv} \\\texttt{20110220} \\t12487\\
		\text{not edge-orientable}\\
		\text{layered}\\
		\zz \oplus \zz\end{matrix} \arrow[d, "\sigma", swap]\\
	\begin{matrix}\texttt{ivLLQQccdhghgfhggrqlipigb}\\\texttt{12020011}\\ t12487\\
		\text{edge-orientable}\\
		\text{layered}\\
		\zz \oplus \zz\end{matrix} \arrow[r,"\varrho"] &\begin{matrix}\texttt{iLLLPQccdgefhhghhrhajsvss}\\
		\texttt{02221000}\\ t12488\\
		\text{not edge-orientable}\\
		\text{measurable}\\
		\zz \oplus \zz/8\end{matrix} \\
\end{tikzcd}
\caption{Four veering mutants. Each data set consists of the isomorphism signature of the triangulation (first row), the taut angle structure (second row),  the name of the underlying manifold in the SnapPy's census (third row), information about edge-orientability (fourth row), the type of the triangulation (fifth row), and the first homology group with integer coefficients of the underlying manifold (sixth row).}
\label{fig:4mutants}
\end{figure}

\begin{proposition}\label{prop:two:tris:on:face}
A non-fibered face $\face{F}$ of the Thurston norm ball can be combinatorially represented by two combinatorially distinct veering mutants.
\end{proposition}
\begin{proof}
Triangulations $\V$, $\V^{\varrho\sigma}$
are two combinatorially distinct measurable veering mutants. Since there is a sequence of Pachner moves from $\V$ to $\V^{\varrho\sigma}$, their underlying manifolds are homeomorphic. (The shortest such path has length four and consists of two 2-3 moves and two 3-2 moves.) Both these triangulations carry a twice punctured torus, which in particular means that the cones $\C(\V)$, $\C(\V^{\varrho\sigma})$ are nonempty. 
Since the first Betti number of $M$ is equal to 1, up to $\eta \mapsto -\eta$ there is only one face $\face{F}$ of the Thurston norm ball in $H^1(M;\rr)$. Remark \ref{remark:not:opposite} then implies that $\C(\V) = \C(\V^{\varrho\sigma}) = \cone(\face{F})$. 
\end{proof}


Another conclusion that we can draw from Figure \ref{fig:4mutants} is that a mutant of a measurable veering triangulation does not have to be measurable.
Observe however, that if a layered veering triangulation $\V$ admits a measurable veering mutant $\V^\varphi$ then the homology class of the mutating surface must lie in the boundary of the fibered cone represented by $\V$.
\begin{proposition}\label{prop:mutant:boundary}Let $\V$ be a finite layered veering triangulation of a 3-manifold $M$.
Suppose that $\V \rightarrow \V^\varphi$ is a veering mutation such that $\V^\varphi$ is measurable. Then the homology class of the mutating surface lies in the boundary of the fibered cone in $H_2(M, \partial M; \rr)$ represented by $\V$. 
\end{proposition}
\begin{proof}
Since $\V$ is layered, the face $\face{F}$ of the Thurston norm ball represented by $\V$ is fibered \cite[Theorem 5.15 (stated here as Theorem \ref{thm:LMT:faces})]{LMT}.
Denote by $S_w$ the surface carried by $\V$ that can be used to mutate $\V$ into $\V^\varphi$, and by $\Semb$ the embedded surface obtained from $S_w$ by slightly pulling apart overlapping regions of $S_w$. The surface $\Semb$ is a Thurston norm minimizing representative of its homology class \cite[Theorem~3]{Lack_taut}. If that homology class lies in the interior of $\cone(\face{F})$ then $M|\Semb$ is a product sutured manifold \cite[Theorem~3]{Thur_norm}. Therefore the mutant manifold $M^\varphi$ is fibered over the circle with the mutating surface being the fiber. The assumption that $\V \rightarrow \V^\varphi$ is a veering mutation implies that $\varphi$ misaligns edge  product disks, and therefore, by Theorems~\ref{thm:mutant:correct:manifold} and \ref{thm:mutant:veering}, $\V^\varphi$ is a veering triangulation of $M^\varphi$. We therefore get that $\V^\varphi$ carries a fiber of a fibration of $M^\varphi$ over the circle. But a veering triangulation that carries fibers of fibrations over the circle is layered \cite[Theorem~5.15]{LMT}. This is  a contradiction with the assumption that $\V^\varphi$ is measurable. 
\end{proof}

The other two mutants, $\V^\varrho$ and $\V^\sigma$, both live on the same 3-manifold t12487, which is the L11n222 link complement.  Since we cannot have  two combinatorially distinct veering triangulations representing the same fibered face of the Thurston norm ball \cite[Proposition 2.7]{MinskyTaylor}, we deduce that $\V^\varrho$ and $\V^\sigma$ represent different faces of the Thurston norm ball. In particular, it is possible that two different fibered faces of the Thurston norm ball of the same manifold are related by a veering mutation.

Let $\face{F}^\varrho$, $\face{F}^\sigma$ be the fibered  face  represented by $\V^\varrho$,~$\V^\sigma$, respectively. The Thurston norm ball of the L11n222 link complement is a quadrilateral with two pairs of fibered faces. Therefore the mutating twice punctured torus $S$ represents the primitive integral class  lying either on the ray $\cone{\face{F}^\varrho} \cap \cone{\face{F}^\sigma}$ or on the ray $\cone{\face{F}^\varrho} \cap \left(-\cone{\face{F}^\sigma}\right)$. Since the stable train tracks $\tau^\varrho$, $\tau^\sigma$ on~$S$ induced from $\V^\varrho, \V^\sigma$, respectively, are equal, we get that under the same mutation (eg. by~$\varrho$) two different veering triangulations on t12487 mutate into two different veering triangulations on t12488. It is the fact that the first Betti number of t12488 is equal to one that makes these two distinct veering triangulations represent the same top-dimensional face of the Thurston norm ball. 
In other words, in this example the phenomenon of a top-dimensional non-fibered face of the Thurston norm ball represented by multiple distinct veering triangulations arises from mutating a fibered 3-manifold with a higher first Betti number along a surface representing a class lying at the intersection of multiple fibered faces.

\begin{remark}\label{remark:many-betti1}
There are 110 manifolds with the first Betti number equal to one which admit 2 measurable veering triangulations. Among those, 87 differ by a veering mutation along a connected surface. The mutating surface is either a four times punctures sphere (8 cases), a twice punctured torus (75 cases), or a four times punctured torus (4 cases). 
\end{remark}

\subsection{Two veering mutants representing the same face of the Thurston norm ball when $b_1 (M)= 2$}\label{subsubsec:higher:betti}
As explained at the beginning of this section, finding examples of two different measurable veering triangulations representing the same face of the Thurston norm ball is harder when $b_1(M)>1$ because  then $H_2(M, \partial M;\rr)$ admits infinitely many distinct bases. A possible approach to overcome this problem is to focus on manifolds for which any two non-fibered non-opposite faces of the Thurston norm ball have different combinatorics, or which have only one pair of opposite non-fibered faces. For instance, we searched for a cusped hyperbolic \mbox{3-manifold} $M$ such that
\begin{itemize}
\item $b_1(M) = 2$,
\item the Thurston norm ball in $H_2(M, \partial M;\rr)$ is a quadrilateral,
\item $M$ is fibered,
\item $M$ admits at least two measurable veering triangulations $\V, \V'$ with different taut signatures and such that $\C(\V), \C(\V')$ are 2-dimensional.
\end{itemize}
When the first three conditions are satisfied, the Thurston norm ball of $M$ admits only one pair of top-dimensional non-fibered faces.  Therefore if additionally $M$ admits at least two measurable veering triangulations whose cones of homology classes of carried surfaces are 2-dimensional, they either represent the same non-fibered face or opposite non-fibered faces; see Remark \ref{remark:not:opposite}. If they represent opposite faces then switching coorientations on faces of one of the triangulations make them represent the same face.
In the Veering Census \cite{VeeringCensus} there is a 3-manifold $M$ which satisfies all these conditions. It admits (at least) three veering triangulations  $\V_1$, $\V_2$, $\V_3$ with the following taut signatures:
\begin{gather*}
\texttt{qLLLzvQMQLMkbeeekljjlmljonppphhhhaaahhahhaahha\_0111022221111001}\\
\texttt{qLLLzvQMQLMkbeeekljjlmljonppphhhhaaahhahhaahha\_1200111112020112}\\
\texttt{qLLLzvQMQLMkbeeekljjlmljonppphhhhaaahhahhaahha\_2111200001111221},
\end{gather*}
respectively. Observe that $\V_1, \V_2, \V_3$ are combinatorially isomorphic as triangulations, but they have different taut structures.
Triangulations $\V_1$ and $\V_3$ are measurable, and~$\V_2$ is layered.  Using \texttt{tnorm}~\cite{tnorm} we can verify that $M$ indeed has only two pairs of faces of the Thurston norm ball. One of this pair has to be fibered because $M$ admits a layered veering triangulation. Thus, after possibly replacing $\V_1$ by $-\V_1$, we must have that $\C(\V_1) = \C(\V_3)$.

Triangulations $\V_1$ and $\V_3$ not only represent the same non-fibered face of the Thurston norm ball, but they are also each other's mutants. Triangulation $\V_1$ carries a four times punctured torus which, using the same labels as in the Veering Census, can be represented by the 2-cycle 
\[S_w = f_3 + f_5 + f_{14} + f_{15} + f_{18} + f_{22} + f_{28} + f_{31}.\]
To save some space, we do not include the picture of tetrahedra of $\V_1$ (this is a triangulation with 16 tetrahedra).
We do, however,  present the induced triangulation $\Q_{\V_1,w}$ and the induced train track $\tau_{\V_1,w}$ in Figure~\ref{fig:S_1_4}. We have
\[\Aut^+(\Q_{\V_1, w} \ | \ \tau_{\V_1, w})= \zz/2 \oplus \zz/2 \oplus \zz/2.\]
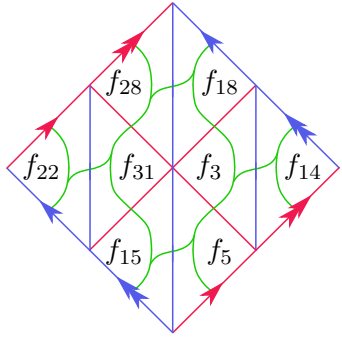
\begin{figure}[h]
\includegraphics[scale=0.7]{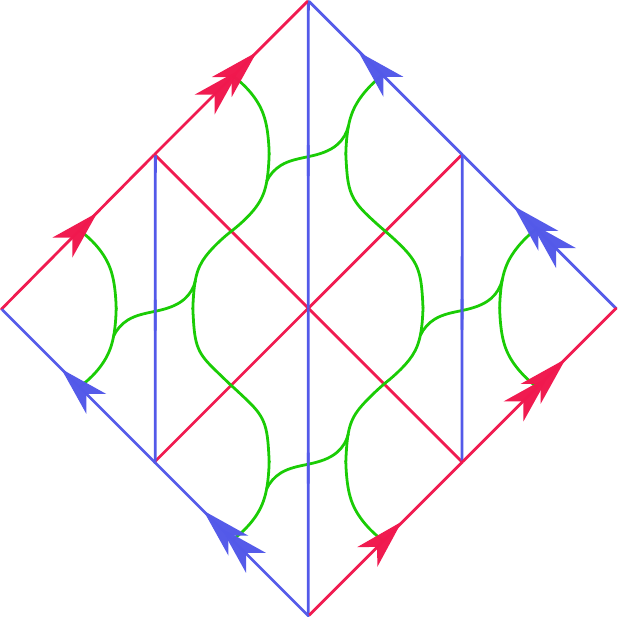}
\put(-119,60){$f_{22}$}
\put(-83,60){$f_{31}$}
\put(-54,60){$f_{3}$}
\put(-21,60){$f_{14}$}
\put(-88,28){$f_{15}$}
\put(-50,28){$f_{5}$}
\put(-88,92){$f_{28}$}
\put(-52,92){$f_{18}$}
\caption{Ideal triangulation $\Q_{\V_1, w}$ and the stable train track $\tau_{\V_1, w}$ of a four times punctured torus carried by $\V_1$.}
\label{fig:S_1_4}
\end{figure}

The group $\Aut^+(\Q_{\V_1, w} \ | \ \tau_{\V_1, w})$ is generated by the rotation by $\pi$ around the center  of Figure \ref{fig:S_1_4}, which we denote by $\varrho$, shift by one `layer' in the north-east direction $\sigma_+$, and shift by one `layer' in the north-west direction $\sigma_-$. Surface $S_w$ does not traverse any edge of $\V_1$ more than once, and thus there are no edge product disks in $M|\Semb$. Consequently, we can construct 8 veering mutants of $\V_1$. They do not have pairwise distinct taut signatures. In particular, $\V_3$ has the same taut signature as both $\V_1^\varrho$ and~$\V_1^{\varrho\sigma_-}$. Data on the remaining mutants of~$\V_1$ is available in Figure \ref{fig:6_mutants}. Observe that in each column we have veering triangulations with the same isomorphism signature but different taut structure. Thus in each column we have two veering triangulations of the same manifold. The manifold in the right column is the complement of the L14n62847 link.

\begin{figure}[h]
\hspace*{-1.7cm}\begin{tikzcd}[column sep=large, row sep = large]
	\begin{matrix}\texttt{qvvLPQwAPLQkfghffjiklnm}\\\texttt{popopoqaaaqaaqqqqaaqaqa}\\
		\texttt{1202011111002211}\\ 
		\text{not edge-orientable}\\
		\text{layered}\\
		\zz \oplus \zz \oplus \zz/4 \oplus \zz/4\end{matrix} &
	\begin{matrix}\texttt{qLLLzvQMQLMkbeeekljjlml}\\\texttt{jonppphhhhaaahhahhaahha}\\
		\texttt{0111022221111001}\\ 
		\text{not edge-orientable}\\
		\text{measurable}\\
		\zz \oplus \zz \oplus \zz/2 \oplus \zz/4\end{matrix} \arrow[l,"\sigma_+", swap] \arrow[r,"\sigma_- \circ \sigma_+"] \arrow[ld, "\varrho \circ \sigma_+", swap] \arrow[rd, "\varrho \circ \sigma_- \circ \sigma_+"] \arrow[d,"\varrho", shift right=1.5ex, swap] \arrow[d, "\varrho \circ \sigma_- ",  shift left=1.5ex] \arrow[loop, distance=3em, "\sigma_-", swap]&
	\begin{matrix}\texttt{qvvLPQwAPLQkfghffjilkmn}\\\texttt{popopoaaaaaaaaaaaaaaaaa}\\
		\texttt{1020211111220011}\\ 
		\text{not edge-orientable}\\
		\text{layered}\\
		\zz \oplus \zz \oplus \zz \oplus \zz\end{matrix} \\	
	\begin{matrix}\texttt{qvvLPQwAPLQkfghffjiklnm}\\\texttt{popopoqaaaqaaqqqqaaqaqa}\\
		\texttt{1020211111220011}\\ 
		\text{edge-orientable}\\
		\text{layered}\\
		\zz \oplus \zz \oplus \zz/4 \oplus \zz/4\end{matrix} &
	\begin{matrix}\texttt{qLLLzvQMQLMkbeeekljjlml}\\\texttt{jonppphhhhaaahhahhaahha}\\
		\texttt{2111200001111221}\\ 
		\text{edge-orientable}\\
		\text{measurable}\\
		\zz \oplus \zz \oplus \zz/2 \oplus \zz/4\end{matrix} &
	\begin{matrix}\texttt{qvvLPQwAPLQkfghffjilkmn}\\\texttt{popopoaaaaaaaaaaaaaaaaa}\\
		\texttt{1202011111002211}\\ 
		\text{edge-orientable}\\
		\text{layered}\\
		\zz \oplus \zz \oplus \zz \oplus \zz\end{matrix}
\end{tikzcd}
\caption{Six veering mutants of $\V_1$. Each data set consists of the isomorphism signature of the triangulation (split into the first and second row), the taut angle structure (third row), information about edge-orientability (fourth row), the type of the triangulation (fifth row), and the first homology group with integer coefficients of the underlying manifold (sixth row).}
\label{fig:6_mutants}
\end{figure}
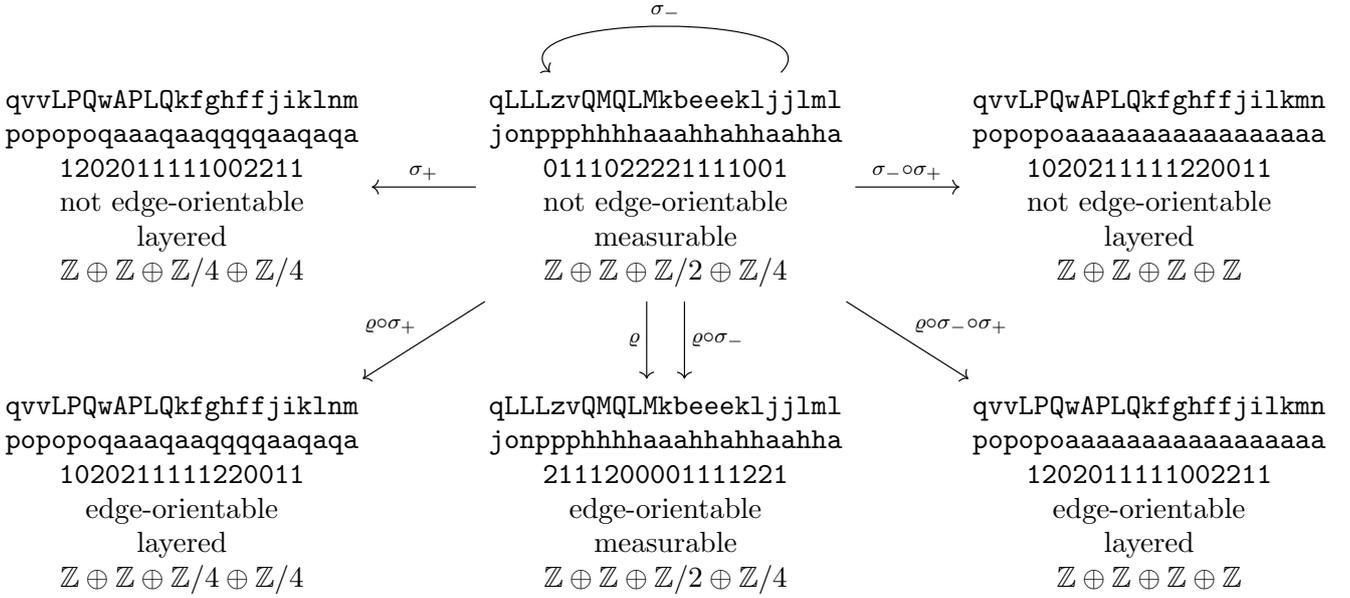

It is worth mentioning that in this example the mutating surface represents a homology class that lies in the interior of the cone $\C(\V_1) = \C(\V_3)$, and thus in the interior of the cone over a face of the Thurston norm ball, but  over a vertex of the Alexander norm ball. See \cite{McMullen_Alex} for the relationship between the Thurston and Alexander norms on $H_2(M, \partial M;\rr)$. 
The Alexander polynomial of $M$ is equal to
\[\Delta_M = a^2b^2 + 2a^2b + 4ab + 4a + 2b + 6 + 2b^{-1} + 4a^{-1} + 4a^{-1}b^{-1}+2a^{-2}b^{-1}+a^{-2}b^{-2}.\]
Using this we present the Alexander norm ball  of $M$ in  Figure \ref{fig:cones}. We also marked the cones $\mathcal{C}(\V_1), \mathcal{C}(\V_2), \mathcal{C}(\V_3)$ of homology classes carried by $\V_1, \V_2, \V_3$, respectively.  We can see that $\mathcal{C}(\V_1) = \mathcal{C}(\V_3)$, and that this cone is a cone on two adjacent faces of the Alexander norm ball, but one face of the Thurston norm ball. The homology class of the mutating surface lies over a vertex of the Alexander norm ball.
\begin{figure}[h]
\includegraphics[scale=2]{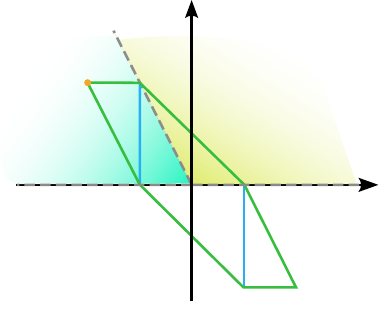}
\put(-215,87){$\C(\V_1) = \C(\V_3)$}
\put(-85, 117){$\C(\V_2)$}
\caption{The unit norm ball of the Alexander norm on $H_2(M, \partial M;\rr)$ has vertices at $\left(\pm \frac{1}{2},  \mp\frac{1}{2}\right), \left(\pm \frac{1}{4}, \mp \frac{1}{2}\right)$ and $\left(\pm \frac{1}{4},0\right)$. Its boundary is marked green. The unit norm ball of the Thurston norm has vertices at $\left(\pm \frac{1}{4}, \mp \frac{1}{2}\right)$ and $\left(\pm \frac{1}{4},0\right)$. The part of its boundary which does not overlap with the boundary of the Alexander norm ball is marked blue. The cones of homology classes carried by $\V_1$ and $\V_3$ are equal to the cone on two adjacent faces of the Alexander norm ball. The surface mutating $\V_1$ to $\V_3$ represents the class $(-1,1)$ lying over the vertex $\left(-\frac{1}{2}, \frac{1}{2}\right)$ of the Alexander norm ball (marked orange).}
\label{fig:cones}
\end{figure}

				\subsection{Mutating along a higher genus surface}\label{sec:higher:genus}
				
				In Subsections \ref{subsec:same:face} and \ref{subsubsec:higher:betti} we discussed pairs of homeomorphic veering mutants for which the mutating surface was of genus one. In Remark \ref{remark:not:opposite} we mentioned also homeomorphic mutants with mutating surface of genus zero. In this subsection we discuss a pair of veering triangulations of the same manifold which differ by a mutation along a surface of genus two. This example differs from the previous ones not only by the genus of the mutating surface, but also by the fact that this surface is a fiber of a fibration over the circle.  Moreover, the sutured manifold $M|\Semb$ has edge product disks.
				
				Let $\V, \V'$ be veering triangulations with taut signatures
				\begin{center}
					\texttt{jLLAvQQcedehihiihiinasmkutn\_011220000},
					
					\texttt{jvLLAQQdfghhfgiiijttmtltrcr\_201102102},
				\end{center}
				respectively. These are two veering triangulations of the $10^3_{12}$ link complement. We present tetrahedra of $\V$ in Figure \ref{fig:smkutn}.
				
				\begin{figure}[h]
					\includegraphics[scale=0.7]{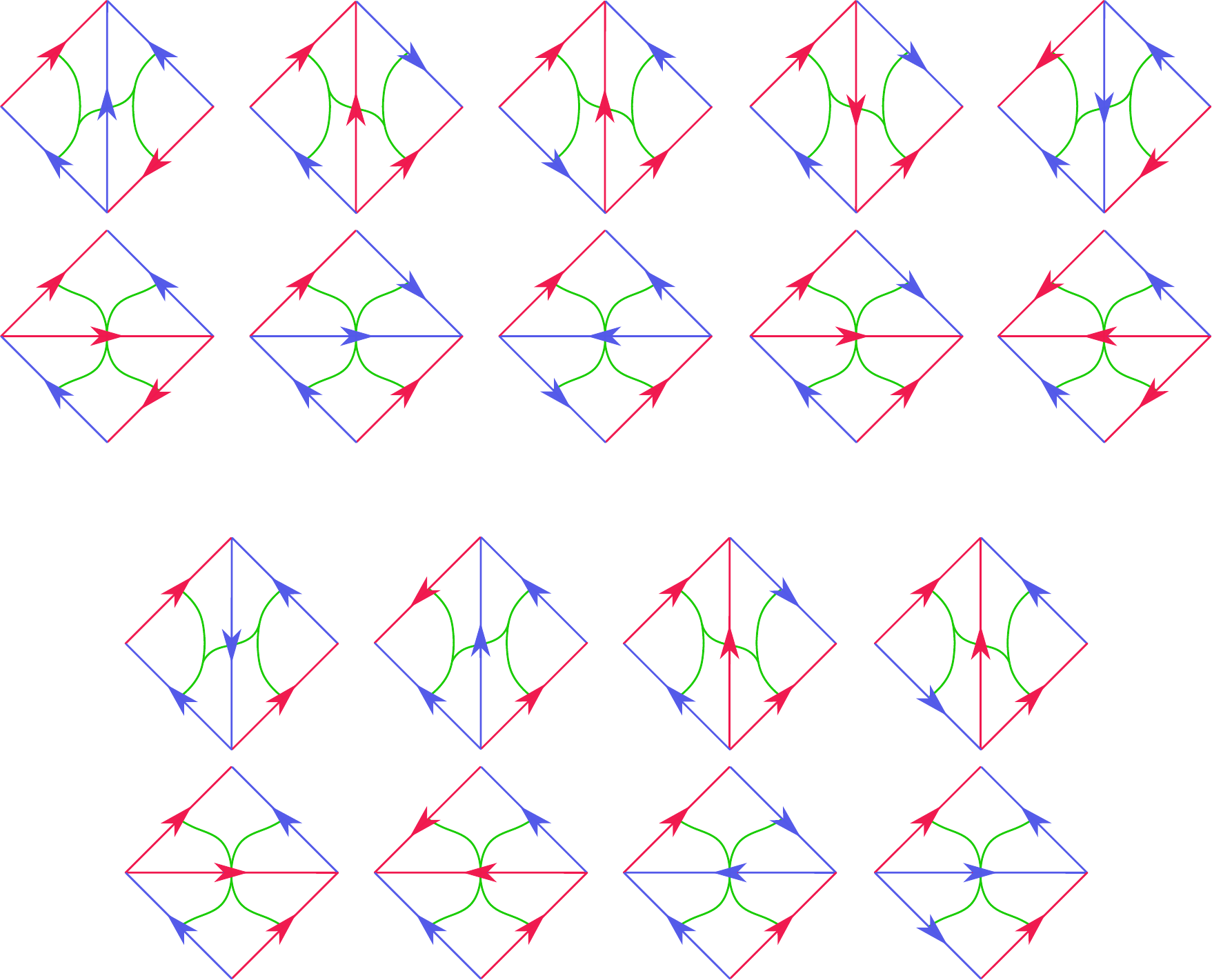}
					\put(-327,295){0}
					\put(-254,295){1}
					\put(-181, 295){2}
					\put(-108, 295){3}
					\put(-35, 295){4}
					\put(-291,138){5}
					\put(-217, 138){6}
					\put(-144, 138){7}
					\put(-70, 138){8}
					\put(-348, 253){$f_2$}
					\put(-312, 253){$f_3$}
					\put(-275, 253){$f_4$}
					\put(-238, 253){$f_5$}
					\put(-202, 253){$f_8$}
					\put(-165, 253){$f_1$}
					\put(-128, 253){$f_0$}
					\put(-92, 253){$f_9$}
					\put(-55, 253){$f_6$}
					\put(-20, 253){$f_{11}$}
					\put(-329, 205){$f_1$}
					\put(-329, 168){$f_0$}
					\put(-256, 205){$f_3$}
					\put(-256, 168){$f_6$}
					\put(-183, 205){$f_2$}
					\put(-183, 168){$f_7$}
					\put(-111, 205){$f_{10}$}
					\put(-110, 168){$f_4$}
					\put(-37, 205){$f_8$}
					\put(-37, 168){$f_5$}
					\put(-314, 97){$f_{12}$}
					\put(-275, 97){$f_7$}
					\put(-240, 97){$f_{16}$}
					\put(-204, 97){$f_{15}$}
					\put(-166, 97){$f_{14}$}
					\put(-129, 97){$f_{10}$}
					\put(-93, 97){$f_{17}$}
					\put(-57, 97){$f_{13}$}
					\put(-294, 47){$f_{13}$}
					\put(-293, 11){$f_{14}$}
					\put(-221, 47){$f_{17}$}
					\put(-220, 11){$f_9$}
					\put(-148, 47){$f_{12}$}
					\put(-148, 11){$f_{15}$}
					\put(-75, 47){$f_{11}$}
					\put(-75, 11){$f_{16}$}
					\caption{Veering triangulation \texttt{jLLAvQQcedehihiihiinasmkutn\_011220000}.}
					\label{fig:smkutn}
				\end{figure}
				
				Let $S_w = 2f_0 + f_2 + f_6 + 2f_7 + 2f_9 + 2f_{11} + f_{12} + f_{16}$. This is a genus two surface with four punctures such that $\Aut^+(\Q_{\V, w} \ | \ \tau_{\V, w}) = \zz/2$; see Figure \ref{fig:smkutn_fiber}. Let $\varrho$ be the generator of this group. Information on the bijection $r^\ro: F_w^+ \rightarrow F_w^-$ determined by~$\ro$ is included in Table \ref{tab:smkutn:regluing}. In this example $M|\Semb$ admits edge product disks --- in Figure 
				\ref{fig:smkutn_fiber} their top bases are shaded yellow and their bottom bases are shaded purple. We can directly check that no edge which is the top base of some edge product disk in $M|\Semb$ is mapped by $\ro$ to an edge which is the bottom base of some edge product disk in $M|\Semb$. This means that $\ro$ misaligns edge product disks. Therefore, by Theorems~\ref{thm:mutant:correct:manifold} and \ref{thm:mutant:veering},~$\V^\ro$ is a veering triangulation of $M^\ro$.  	Using Regina \cite{regina} we can verify that $\V^\ro$ is combinatorially isomorphic to $\V'$, and thus $M^\ro$ is homeomorphic to $M$.
				\begin{table}[h]
					\begin{tabular}{|c||c|c|c|c|c|c|c|c|}
						\hline
						&&&&&&&& \\[-1em]
						$f^+$ & $f_0^+$ & $f_2^+$ & $f_6^+$ & $f_7^+$ & $f_9^+$ & $f_{11}^+$ & $f_{12}^+$ & $f_{16}^+$\\ 
						&&&&&&&& \\[-1em] 
						\hline
						&&&&&&&& \\[-1em]
						$r^\sigma(f^+)$ & $f_0^-$ & $f_{16}^-$ & $f_{12}^-$ & $f_7^-$ & $f_9^-$ & $f_{11}^-$ & $f_6^-$ & $f_2^-$\\
						\hline
					\end{tabular}
					\vspace{0.2cm}
					\caption{The regluing map determined by $\sigma$.}
					\label{tab:smkutn:regluing}\end{table}

				\begin{figure}[h]
					\includegraphics[scale=0.7]{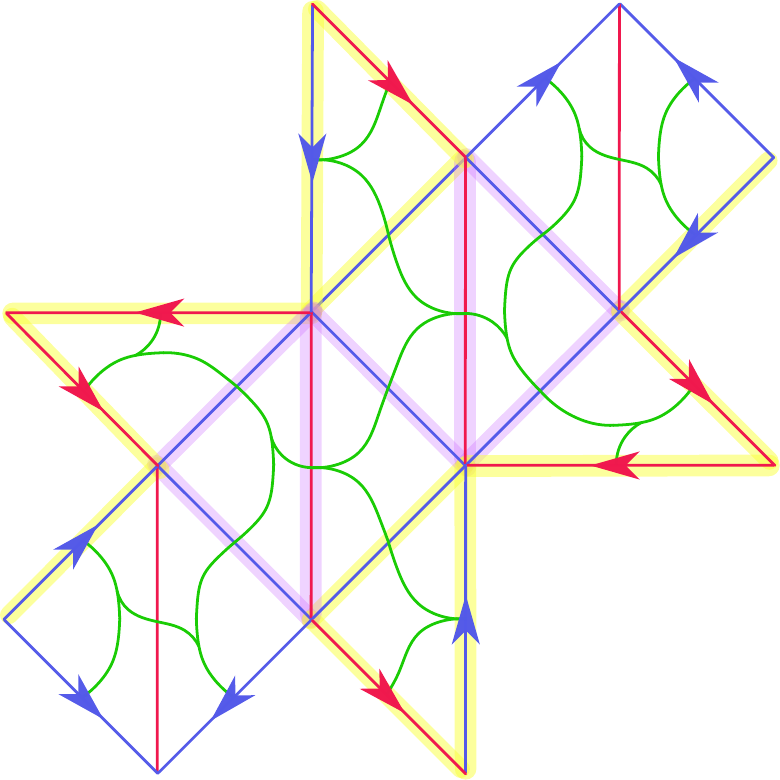}
					\put(-149,29){$f_2$}
					\put(-113,29){$f_6$}
					\put(-119,61){$\check{f_{11}}$}
					\put(-81,61){$\check{f_9}$}
					\put(-87,92){$\check{f_{0}}$}
					\put(-51,92){$\check{f_7}$}
					\put(-57,123){$f_{12}$}
					\put(-22,123){$f_{16}$}
					\put(-132,75){$f_9$}
					\put(-83,123){$f_{11}$}
					\put(-86,29){$f_7$}
					\put(-38,78){$f_0$}
					\put(-15,83){$a$}
					\put(-15,99){$b$}
					\put(-15,146){$c$}
					\put(-54,146){$c$}
					\put(-77,146){$d$}
					\put(-105,123){$b$}
					\put(-128,99){$e$}
					\put(-152,69){$d$}
					\put(-152,49){$f$}
					\put(-151,9){$g$}
					\put(-108,9){$g$}
					\put(-89,9){$a$}
					\put(-59,29){$f$}
					\put(-35,52){$e$}
					\caption{Ideal triangulation $\Q_{\V, w}$ and the stable train track $\tau_{\V, w}$ of a four times punctured genus two surface $S_w$ carried by $\V$. An edge $e$ of $\QVw$ is shaded yellow (respectively, purple) if the edge $e^+$ of $\Q_{\V,w}^+$ (respectively, edge~$e^-$ of $\Q_{\V,w}^-$) is the top (respectively, the bottom) base of an edge product disk in $M|\Semb$. To distinguish between the two copies of $f_i \in F_w$ in $\QVw$ when $w_{f_i}>1$ we denote the lowermost copy of $f_i$ by $\check{f_i}$. Letters $a,b,c,d,e,f, g$ indicate side identifications. The only nontrivial element of $\Aut^+(\Q_{\V, w} \ | \ \tau_{\V, w})$ is the rotation by $\pi$ around the center of the edge between $\check{f_9}$ and $\check{f_0}$. It misaligns edge product disks, because no edge shaded yellow is mapped to an edge shaded purple.}
					\label{fig:smkutn_fiber}
				\end{figure}

				With some choice of basis for $H_2(M, \partial M;\zz)$ the cone $\C(\V)$ is spanned by  $(0, 0, 1), (0, 1, -1),$ $(1, 0, 0)$ and the homology class of $S_w$ is then given by $(1,2,-1)$. In particular, $\Semb$ is a fiber of a fibration of $M$ over the circle. The taut polynomial of $\V$ and its specialization at $\lbrack \Semb \rbrack$ are respectively equal to
				\begin{gather*}\Theta(a,b,c) = a^2b^3c^2 - a^2b^2c^2 - ab^3c^2 - a^2b^2c + ab^2c + abc - bc - a - b + 1\\
					\Theta^{(1,2,-1)}(z) = \Theta(z^1,z^2, z^{-1}) = z^6 - 2z^5 - 2z + 1.\end{gather*}
				It follows from \cite[Theorem 4.2]{McMullen_Teich} and \cite[Theorem 7.1]{LMT} that the stretch factor $\lambda$ of the monodromy $f$ of the fibration with fiber $\Semb$ is equal to the largest real root of $\Theta^{(1,2,-1)}(z)$, that is 
				\[\lambda = \frac{1 + \sqrt{17} + \sqrt{2 (1 + \sqrt{17})}}{4}\approx 2.081.\]
				
				The mutating surface in $\V^\ro$ is a fiber of a fibration of $M$ over the circle with monodromy $\ro f$ and thus the same stretch factor. The fibered faces $\face{F}, \face{F}^\ro$ represented by $\V, \V^\ro$, respectively, must be different because there is at most one veering triangulation associated to a fibered face $\face{F}$ of the Thurston norm ball (zero if the associated circular flow has singular orbits) \cite[Proposition 2.7]{MinskyTaylor}. In this case we can actually deduce a stronger statement, that there is no automorphism $\Phi$ of  $H_2(M, \partial M;\rr)$ that sends $\face{F}$ to~$\face{F}^\ro$. This follows from the fact that $\face{F}$ is a triangle, while $\face{F}^\ro$ is a pentagon.  Thus $\ro f$ and $f$ are not conjugate in the mapping class group of a genus two surface with four punctures.
				
				\pagebreak
				\begin{fact}\label{fact:monodromies:not:conjugate}
					Let $M$ be the complement of the $10^3_{12}$ link.  \nopagebreak
					\begin{itemize}
						\item $M$ fibers in two different ways with fiber being  a genus two surface with four punctures and such that the monodromy of one fibration is obtained from the monodromy of the other fibration by postcomposing it with an involution $\ro$. 
						\item The two fibrations lie over different faces of the Thurston norm ball, and no automorphism of $H_2(M,\partial M;\rr)$ sends one face to the other. Thus the monodromies are not conjugate in the mapping class group of a genus two surface with four punctures.
					\end{itemize}
				\end{fact}	
				\section{Flows representing the same face of the Thurston norm ball}\label{sec:distinct:flows}
				Veering mutants discussed in Sections \ref{subsec:same:face} and \ref{subsubsec:higher:betti}  are combinatorially distinct. 
				If, as in Remark \ref{remark:inverse}, we assume a bijective correspondence between pseudo-Anosov flows, up to topological equivalence, and pairs (veering triangulation, appropriate Dehn filling data) we can immediately deduce that the flows built from the pairs of homeomorphic veering mutants using Theorem \ref{thm:AT} have to be topologically inequivalent. However, since  there is no written proof of this  correspondence, we will prove that the flows are inequivalent using the following lemma.
				
				\begin{lemma}\label{lem:inequivalent:flows}
					Let $\Psi_1, \Psi_2$ be two pseudo-Anosov flows on a closed 3-manifold $N$. If the stable lamination of $\Psi_1$ is transversely orientable, and the stable lamination of $\Psi_2$ is not then $\Psi_1, \Psi_2$ are not topologically equivalent.  Analogous statement holds for the blown-up flows $\Psi_1^\circ$, $\Psi_2^\circ$.
				\end{lemma}
				\begin{proof}
					If $\Psi_1, \Psi_2$ are topologically equivalent there is a homeomorphism \mbox{$h: M \rightarrow M$} taking oriented orbits of $\Psi_1$ to oriented orbits of $\Psi_2$ (see Definition \ref{defn:top:equiv}). This homeomorphism must take leaves of the stable lamination of $\Psi_1$ to the leaves of the stable lamination of~$\Psi_2$. 
					But the stable lamination of $\Psi_2$ admits M\"obius band leaves, while the stable lamination of $\Psi_1$ have only planar and annular leaves. Therefore $\Psi_1, \Psi_2$ cannot be topologically equivalent.
				\end{proof}
				Thus if a face $\face{F}$ of the Thurston norm ball is combinatorially represented by two veering triangulations one of which is edge-orientable and the other is not, Lemma~\ref{lem:inequivalent:flows} together with Corollary \ref{cor:EO:TO} imply that the flows built out of these veering triangulations have to be topologically inequivalent. 
				\begin{theorem}\label{thm:many:flows}
					A non-fibered face $\face{F}$ of the Thurston norm ball can be dynamically represented by two topologically inequivalent flows.
				\end{theorem}
				\begin{proof}
					For the cusped case: apply Lemma \ref{lem:inequivalent:flows} and Corollary \ref{cor:EO:TO} to the blown-up flows built from the pair of homeomorphic veering mutants discussed in Subsection~\ref{subsec:same:face} (for $b_1 = 1$ case) or \ref{subsubsec:higher:betti} (for $b_1 = 2$ case) using Theorem \ref{thm:AT}.
					
					Both these pairs of veering mutants differ by a mutation along a punctured torus. Thus Dehn filling their underlying 3-manifolds along the slopes determined by the boundaries of  these tori yields toroidal  3-manifolds. Consequently, these veering triangulations cannot be used to construct two distinct pseudo-Anosov flows on a closed hyperbolic 3-manifold which represent the same face of the Thurston norm ball.

					To prove the proposition in the closed case we will use a different pair of veering triangulations that represent the same face of the Thurston norm ball. 
					Let $\V_1, \V_2$ denote veering triangulations with taut signatures
					\begin{center}
						\texttt{mLvLQLzQQcghefihljkkllhahxxqnxwbbfj\_022211002221}, \\
						\texttt{mvLAvMMQQecfiikkjjilllfrtdrjfgbbhhd\_201102222211},
					\end{center}
					respectively. Then
					\begin{enumerate}[label=(\alph*)]
						\item $\V_1$, $\V_2$ are two measurable veering triangulations of the same manifold $M$ with  \mbox{$H_1(M;\zz) = \zz \oplus \zz/4$}.
						\item $\V_2$ is edge-orientable, while $\V_1$ is not.
						\item For $i=1,2$ every surface $S_w$ carried by $\V_i$ which represents the generator of $H_2(M, \partial M;\zz)$ has genus 3 and 4 punctures, all complementary regions of $\tau_{\V_i, w}$ are once punctured 4-gons, and $\Aut^+(\Q_{\V_i, w} \ | \ \tau_{\V_i, w}) = 1$.  
					\end{enumerate}
					
					It follows from (a) that, after possibly replacing $\V_1$ by $-\V_1$, $\V_1$ and $\V_2$ represent the same face of the Thurston norm ball in $H_2(M, \partial M;\rr)$. Part (c) and Remark \ref{remark:same:carried:surfaces} imply that both $\V_1$ and $\V_2$ carry a surface $S$ of genus 3 with four punctures such that the complementary regions of the stable train track in~$S$ are all once punctured \mbox{4-gons}. This means that the ladderpole curves of $\V_i$ intersect each boundary component of~$S$ four times. Thus we can use Theorem \ref{thm:L:cones} to deduce that the 3-manifold $N$ obtained from~$M$ by Dehn filling it along the slope determined by the boundary components of~$S$ is hyperbolic.
					
					By Theorem~\ref{thm:AT}, for $i=1,2$ we can use $\V_i$ to construct a transitive pseudo-Anosov flow~$\Psi_i$ on~$N$. Since triangulations $\V_1, \V_2$ represent the same face of the Thurston norm ball in $H_2(M, \partial M;\rr)$,  the flows $\Psi_1, \Psi_2$ represent the same face of the Thurston norm ball in  $H_2(N;\rr)$; see Theorem \ref{thm:L:cones}. The fact that they are not topologically equivalent follows from (b), Lemma \ref{lem:inequivalent:flows} and Corollary \ref{cor:EO:TO}.		
				\end{proof}
			\begin{remark}\label{remark:not:mutants?}
				Part (c) implies that the triangulations $\V_1$, $\V_2$  used to prove Theorem~\ref{thm:many:flows} in the closed case do not not differ by a veering mutation along a carried surface representing the generator of $H_2(M, \partial M;\zz)$. We suspect that they are not veering mutants at all.
				\end{remark}
			It is perhaps important to note that it is also possible that the same face of the Thurston norm ball is combinatorially represented by two distinct veering triangulations with \emph{different} number of tetrahedra. One such pair is given by veering triangulations $\V_1, \V_2$ with the folowing taut signatures: 
			\begin{center} \texttt{lLLvLMQQccdjgkihhijkkqrwsdcfkfjdq\_02221000012}, 
				\\ \texttt{pvLLALLAPQQcdhehlkjmonmoonnwrawwaewaamgwwvn\_122221111122002},
			\end{center}
			respectively. These are two measurable veering triangulations on a 3-manifold $M$ with $b_1(M) = 1$. In their case it is even possible to easily show that $\V_1$, $\V_2$ do not differ by a single mutation with insertion (even a taut one). If they did, then the dual graph of~$\V_2$ would have a subgraph isomorphic to the graph obtained from the dual graph of $\V_1$ by deleting all its edges dual to the faces which have a nonzero weight for some weight system on $\V_1$. It can be directly checked that it does not. 
			
			\begin{fact}\label{fact:not:mutant}
				A non-fibered face of the Thurston norm ball can be combinatorially represented by two distinct veering triangulations which do not differ by a veering mutation or a veering mutation with insertion.\qed
			\end{fact}
		
		\begin{remark}\label{remark:counterexamples:to:Barbot}
			In Subsection \ref{subsec:same:face} we constructed a pair of veering mutants $\V$, $\V^{\ro\sigma}$ on the manifold t12488. 
			Let $N$ denote the manifold obtained from t12488 by Dehn filling it along the boundary of the mutating surface. Using Regina \cite{regina} it is possible to verify that $N$ is a graph manifold obtained from the orientable circle bundle~$N_0$ over a \mbox{2-holed} $\rr P^2$ by identifying its two toroidal boundary components. Thus $N$ is a so-called \emph{BL-manifold}, as defined by Barbot in \cite{Barbot-BL}. 
			
			Langevin-Bonatti constructed an Anosov flow on one BL-manifold in  \cite{Bonatti-Langevin}; a description of this flow written in English can be found in \cite{BL-flow-english}. Barbot generalized the construction to most other  BL-manifolds \cite[Theorem A]{Barbot-BL} and called the resulting Anosov flows  \emph{BL-flows}. If a BL-manifold is not a circle bundle then the constructed flow is not $\rr$-covered, because it is not circular but is transverse to a torus.
			
			In Theorem B(2) Barbot claims that all non $\rr$-covered Anosov flows on a fixed BL-manifold which is not a circle bundle are topologically equivalent. This is in contradiction with our results. It follows from Theorem \ref{thm:AT} and Lemma \ref{lem:inequivalent:flows} that~$N$ admits a pair of topologically inequivalent \mbox{BL-flows} --- one constructed from $\V$ and the other constructed from $\V^{\ro\sigma}$. We denote them by $\Psi$ and $\Psi^{\ro \sigma}$, respectively. The flows $\Psi$, $\Psi^{\ro\sigma}$ are constructed from the same semiflow $\Phi_0$ on $N_0$, but --- unsurprisingly, given that $\V$, $\V^{\ro\sigma}$ are mutants --- by gluing the two boundary tori in a different way. More specifically,  if for $i=1,2$ we  choose a basis $(o_i, f_i)$ on the boundary torus $T_i$ of $N_0$ so that $f_i$ is a fiber of a Seifert fibration while $o_i$ corresponds to the boundary of the 2-holed $\rr P^2$ contained in $T_i$, then 
			one gluing $T_1 \rightarrow T_2$ can be represented by a matrix
			\[A = \begin{bmatrix}
				1 & 1\\
				1&0
			\end{bmatrix},\]
		while the other by $-A$. These two gluings result in the same manifold $N$ because~$N_0$ admits an involution which fixes $(o_1, f_1)$ and sends $(o_2, f_2)$ to $(-o_2, -f_2)$. This involution can be obtained as the composition of the reflection across the stable leaf through the periodic orbit of $\Phi_0$ missing $T_i$ and the reflection which fixes every fiber of the Seifert fibration, but reverses their orientation. 
		
		More generally, given any possible gluing $A: T_1 \rightarrow T_2$ the manifolds $N_0/A$ and $N_0/(-A)$ are homeomorphic. However, the homeomorphism does not send the flow lines of one BL-flow to the flow lines of the other. When one gluing produces transversely orientable foliations then the other does not.
			Consequently, contrary to Theorem B(2) of~\cite{Barbot-BL}, on any BL-manifold which is not a circle bundle there are two BL-flows: one whose stable/unstable foliations are transversely orientable and another whose stable/unstable foliations are not transversely orientable. It seems that Barbot erroneously assumed that the stable/unstable foliations of BL-flows are never transversely orientable \cite[p. 786]{Barbot-BL}.
		
		In Remark \ref{remark:many-betti1} we mentioned 79 pairs of veering mutants on manifolds with first Betti number equal to one for which the mutating surface is a punctured torus. In most cases Regina recognizes their appropriate Dehn fillings as  BL-manifolds.
		\end{remark}
			
				\section{Polynomial invariants of veering triangulations}\label{sec:polys}
				In \cite{McMullen_Teich} McMullen introduced a polynomial invariant of fibered faces of the Thurston norm ball called the \emph{Teichm\"uller polynomial}. Recall that associated to a fibered face~$\face{F}$ there is a unique circular flow $\Psi$ \cite[Theorem~7 (stated here as Theorem \ref{thm:F})]{Fried_suspension}. The Teichm\"uller polynomial of $\face{F}$ is a certain polynomial invariant of the module of transversals to the preimage of the stable lamination of $\Psi$ in the maximal free abelian cover of the manifold \cite[Section 3]{McMullen_Teich}. Its main feature is that it can be used to compute the stretch factors of monodromies of all fibrations lying over $\face{F}$ \cite[Theorem 4.2]{McMullen_Teich}.
				
				McMullen asked whether it is possible to define a similar invariant for non-fibered faces. If a non-fibered face is dynamically represented by a pseudo-Anosov flow~$\Psi$, one could try to replicate the definition of the Teichm\"uller polynomial using the stable lamination of $\Psi$.
				Landry-Minsky-Taylor used veering triangulations to devise such a polynomial invariant \cite{LMT}. In fact, they defined  two polynomial invariants of veering triangulations: the \emph{taut polynomial} and the \emph{veering polynomial}. Furthermore, they showed that if a face $\face{F}$ of the Thurston norm ball represented by a veering triangulation~$\V$ is fibered, then the taut polynomial of~$\V$ is equal to the Teichm\"uller polynomial of $\face{F}$ \cite[Theorem~7.1]{LMT}. 
				Therefore the taut polynomial (and its specializations under Dehn fillings) can be seen as a generalization of the Teichm\"uller polynomial to (some) non-fibered faces.

				However, in Section \ref{sec:faces} we showed that a veering triangulation representing a non-fibered face of the Thurston norm ball is not necessarily unique. This means that the taut and veering polynomials of a veering triangulation might actually not be invariants of the face represented by the triangulation.
				
				An algorithm to compute the taut and veering polynomials of a veering triangulation is explained in \cite{Parlak-computation}. A much faster algorithm for the computation of the taut polynomial follows from the fact that it is equal to the Alexander polynomial of the underlying manifold twisted by a certain representation $\omega: \pi_1(M) \rightarrow \zz/2$ \cite[Proposition 5.7]{taut_alex} and can be therefore computed using \emph{Fox calculus}.
				Both algorithms have been implemented by the author, Saul Schleimer, and Henry Segerman; see Veering GitHub \cite{VeeringGitHub}.
				Using  this software we  computed the taut and veering polynomials of the pairs of veering triangulations representing the same face of the Thurston norm ball discussed in Sections \ref{sec:faces} and \ref{sec:distinct:flows}.
				We include this data in Table \ref{tab:polynomials}.  
				\begin{table}[h] 
					\begin{tabular}{ |c|c|}
						\hline
						& \texttt{iLLLPQccdgefhhghqrqqssvof\_02221000} \\\hline  & \\[-1em]
						$\Theta$ & $4(a+1)$\\ 
						$\mathbb{V}$ &$4(a-1)^3(a+1)$\\ \hline
						& \texttt{iLLLPQccdgefhhghhrhajsvss\_02221000} \\ \hline & \\[-1em]
						$\Theta$ & $4(a-1)$\\ 
						$\mathbb{V}$ &$4(a-1)^3(a+1) $\\ \hline
					%
						\hline
						& \texttt{qLLLzvQMQLMkbeeekljjlmljonppphhhhaaahhahhaahha\_0111022221111001}\\ \hline  & \\[-1em]
						$\Theta$ & $a^2b^2 - 2a^2b + 4ab - 4a - 2b + 6 - 2b^{-1} - 4a^{-1} + 4a^{-1}b^{-1}-2a^{-2}b^{-1}+a^{-2}b^{-2}$\\ 
						$\mathbb{V}$ &$0$\\ \hline 
						& \texttt{qLLLzvQMQLMkbeeekljjlmljonppphhhhaaahhahhaahha\_2111200001111221} \\ \hline & \\[-1em]
						$\Theta$ & $a^2b^2 + 2a^2b + 4ab + 4a + 2b + 6 + 2b^{-1} + 4a^{-1} + 4a^{-1}b^{-1}+2a^{-2}b^{-1}+a^{-2}b^{-2}$\\ 
						$\mathbb{V}$ &$0$\\ \hline
					%
						\hline
						& \texttt{mLvLQLzQQcghefihljkkllhahxxqnxwbbfj\_022211002221}\\ \hline & \\[-1em]
						$\Theta$ & $a^7 + a^6 + 2a^5 - 2a^2 - a - 1$\\ 
						$\mathbb{V}$ &$0$\\ \hline
						& \texttt{mvLAvMMQQecfiikkjjilllfrtdrjfgbbhhd\_201102222211} \\ \hline & \\[-1em]
						$\Theta$ & $a^7 - a^6 + 2a^5 + 2a^2 - a + 1$\\ 
						$\mathbb{V}$ &$a^{22} - a^{21} + 2a^{20} + 2a^{17} - a^{16} + a^{15} - a^7 + a^6 - 2a^5 - 2a^2 + a - 1$\\ \hline
					\end{tabular}
					\vspace{0.2cm}
					\caption{The taut and veering polynomials of pairs of veering triangulations representing the same face of the Thurston norm ball discussed in Sections \ref{sec:faces} and \ref{sec:distinct:flows}. $\Theta$ denotes the taut polynomial, $\mathbb{V}$ denotes the veering polynomial.}
					\label{tab:polynomials}
				\end{table}

									\begin{fact}\label{fact:different:polys}
										A non-fibered face of the Thurston norm ball can be combinatorially represented by two distinct veering triangulations with different taut polynomials, and different veering polynomials. 
									\end{fact}
									\begin{proof}
										See Table \ref{tab:polynomials}.\end{proof}
									
									The only pair of veering triangulations from Table \ref{tab:polynomials} which have different both  taut and  veering polynomials consists of triangulations which probably are not veering mutants; see Remark \ref{remark:not:mutants?}. Hence the question still remains whether two homeomorphic veering mutants representing the same face of the Thurston norm ball can have different both taut and veering polynomials. The answer to this question is positive. One such pair consists of veering triangulations
									\begin{center}
										\texttt{mvLLMvQQQegffhijkllkklreuegggvvrggr\_120200111111}\\
										\texttt{mvLLMvQQQegffhjikllkklreuegrrvvrwwr\_120200111111}.
									\end{center}
									They differ by a veering mutation along a four times punctured torus. Their taut polynomials are $8(a+1)$, $8(a-1)$, respectively, and their veering polynomials are $8(a-1)(a+1)^3$, $8(a-1)^3(a+1)$, respectively. 
				
				\section{Further questions}\label{sec:questions}
				\subsection{Operations on flows underlying veering mutations}\label{subsec:speculation}
				Throughout the paper we worked combinatorially with veering triangulations and used existing literature \cite{Tsang-Agol, LMT_flow, LMT} to deduce statements about pseudo-Anosov flows on closed manifolds or their blow-ups on manifolds with toroidal boundary. We have intentionally avoided discussing how are the flows underlying veering mutants related.  A naive expectation would be that the flows differ by a \emph{mutation of (blown-up) pseudo-Anosov flows}. In the closed case this would be a mutation along a surface transverse to a pseudo-Anosov flow whose intersections with the stable and unstable foliations of the flow are invariant under some nontrivial symmetry. Mutating these foliations via this symmetry gives a pair of 2-dimensional singular foliations intersecting along `recombined flow lines'. It remains to find sufficient conditions for a flow along recombined flow lines (a \emph{mutant flow}) to admit a parametrization which makes it pseudo-Anosov. When $\partial M \neq \emptyset$ one could expect a similar operation performed on the stable and unstable laminations of a blown-up pseudo-Anosov flow.
				
				However, examples presented in Subsection \ref{subsec:same:face} suggest that the problem may be more complicated. Namely, let $\V$ be a veering triangulation with taut signature \texttt{iLLLPQccdgefhhghqrqqssvof\_02221000}. We showed that $\V$ admits four veering mutants: $\V, \V^\ro, \V^\sigma, \V^{\ro\sigma}$. Denote by $M$ the manifold underlying $\V$. By \cite[Theorem~5.1]{Tsang-Agol}, there is a transitive Anosov flow $\Psi$ on the manifold obtained from  $M$ by Dehn filling it along the boundary of the mutating surface, and the blown-up  flow $\Psi^\circ$ on $M$. It can  be deduced from Figure \ref{fig:svof_surfaces}(b) that the intersection $\mathcal{L}_{\Psi^\circ, S}$ of the mutating twice punctured torus $S$ with the stable lamination of $\Psi^\circ$ has two closed leaves and the remaining leaves spiral into them; we approximate this lamination in Figure~\ref{fig:svof_lamination}. It is clear from Figure~\ref{fig:svof_lamination}  that $\mathcal{L}_{\Psi^\circ, S}$ is invariant only under the identity and~$\ro\sigma$.
				This means that even though we can mutate the veering branched surface carrying the stable lamination of~$\Psi^\circ$ in four different ways, the lamination itself can only be mutated in two different ways. 
				
				\begin{figure}[h]
					\includegraphics[scale=0.75]{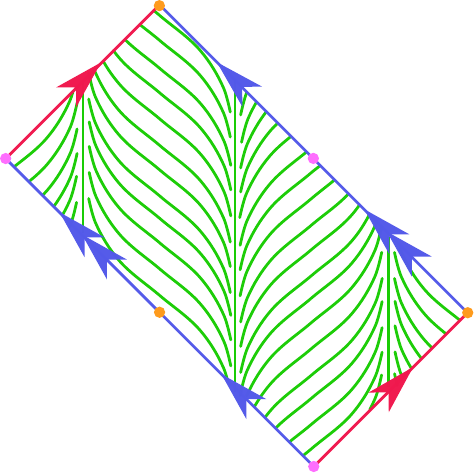}
					\caption{Intersection of the mutating twice punctured torus carried by \texttt{iLLLPQccdgefhhghqrqqssvof\_02221000} with the stable lamination of the underlying blown-up Anosov flow.} 
					\label{fig:svof_lamination}
				\end{figure}

				Working with veering triangulations as opposed to working directly with flows has both advantages and disadvantanges. On one hand, it allowed us to find explicit examples of topologically inequivalent flows on the same manifold which differ by a veering mutation and represent the same face of the Thurston norm ball (Theorem~\ref{thm:many:flows}).  On the other hand, the fact that veering triangulations exist only on hyperbolic 3-manifolds means that if a mutation along a  surface transverse to some (blown-up) pseudo-Anosov flow yields a (blown-up) pseudo-Anosov flow on a non-hyperbolic manifold, there will not be a corresponding mutation on the level of triangulations. This can happen  when we mutate along \mbox{$\varphi \in \Aut^+(\QVw \ | \ \tau_{\V,w})$} which aligns edge product disks; see Proposition \ref{prop:not:hyperbolic}.
				Another obstruction for a veering mutation that would not be an obstruction for a mutation of flows is the `no perfect fits' condition. It is possible that a flow which is without perfect fits relative to a finite collection~$\Lambda$ of closed orbits mutates into a flow which does have perfect fits relative to the recombined collection of orbits $\Lambda^\varphi$. In this case again we do not have a corresponding veering mutation.
				For these reasons, it is still of interest to properly define and study mutations of pseudo-Anosov flows (and possibly other operations underlying veering mutations) without referring to veering triangulations. This would fit into a more general framework of constructing new flows out of old, similarly to the Goodman-Fried surgery \cite{Fried_drill, Goodman}, and Handel-Thurston shearing along tori \cite{HandelThurston}. 
				
				\subsection{The orbit spaces of mutant flows and recognizing mutative flows}
				Associated to a pseudo-Anosov flow $\Psi$ there is a bifoliated plane called the \emph{orbit space} of~$\Psi$ \cite[Proposition 4.1]{Fenley_Mosher}. Suppose that flows $\Psi, \Psi^\varphi$ differ by a mutation along a transverse surface $S$ in the sense introduced in Subsection \ref{subsec:speculation}. If $S$ is a fiber of a fibration over the circle, there is a homeomorphism from the orbit space of $\Psi$ to the orbit space of~$\Psi^\varphi$ which sends foliations of one to the foliations of the other; this follows from the fact these orbit spaces are the universal covers of $S, S^\varphi$ equipped with  the invariant foliations lifted from $S, S^\varphi$, respectively.
				It is not immediately clear how do the orbit spaces differ when $S$ is not a virtual fiber. More generally, 
				it would be advantageous to have an invariant which is equal for flows which are \emph{mutative}, that is differ by a finite number of mutations, and distinguishes flows which are not mutative.

				\subsection{General result on the relationship between two flows representing the same face of the Thurston norm ball}
				We showed that two blown-up Anosov 
				flows representing the same face of the Thurston norm ball can differ by a veering mutation (Subsections~\ref{subsec:same:face} and \ref{subsubsec:higher:betti}). However, we also noted that there are examples of veering triangulations that represent the same face of the Thurston norm ball and do not differ by a veering mutation or even a veering mutation with insertion (Fact \ref{fact:not:mutant}). We have not explained how these veering triangulations, or their underlying flows, are related.
				Ideally, we would like to have a theorem that describes all possible ways in which two distinct flows can represent the same face of the Thurston norm ball. 
				
				\subsection{Homology classes versus free homotopy classes of closed orbits of  flows}
				In recent work Barthelm\'e, Frankel, and Mann found an invariant which distinguishes distinct transitive pseudo-Anosov flows, provided that their orbit spaces satisfy a technical condition called \emph{no tree of scalloped regions}; see \cite[Definition~3.21]{pA-classification}. More precisely, they showed that two such flows $\Psi_1$, $\Psi_2$ on $N$ are isotopically equivalent if and only if the sets $\mathcal{P}(\Psi_1), \mathcal{P}(\Psi_2)$ of unoriented free homotopy classes of their closed orbits are equal, and topologically equivalent if these sets differ by an automorphism of $\pi_1(N)$ \cite[Theorem 1.1]{pA-classification}.
				
				In the proof of Theorem \ref{thm:many:flows} we discussed veering triangulations which, after appropriate Dehn filling, yield topologically inequivalent transitive pseudo-Anosov flows $\Psi_1$, $\Psi_2$ on a closed hyperbolic 3-manifold $N$  representing the same face of the Thurston norm ball in $H_2(N, \rr)$. Using \cite[Proposition 1.2]{pA-classification} it is possible to show that the orbit spaces of $\Psi_1, \Psi_2$ do not have trees of scalloped regions. Thus it follows from \cite[Theorem~1.1]{pA-classification} that $\mathcal{P}(\Psi_1) \neq \Phi \mathcal{P}(\Psi_2)$ for any $\Phi \in \Aut(\pi_1(N))$. On the other hand, we know that the homology classes of closed orbits of $\Psi_1, \Psi_2$ span the same rational cone in $H_1(N;\rr)$. This motivates the question of how exactly do the sets $\mathcal{P}(\Psi_1), \mathcal{P}(\Psi_2)$ differ, not just for these particular flows from the proof of Theorem \ref{thm:many:flows}, but for any two flows which represent the same face of the Thurston norm ball.
				

				\subsection{Many distinct flows representing the same face of the Thurston norm ball in the $b_1(M)=1$ case}
				Suppose that $S$ is a Thurston norm minimizing surface representing a primitive integral class lying at the intersection of two fibered cones $\cone{F_1}$, $\cone{F_2}$. Let $\Psi_1, \Psi_2$ be the circular flows associated to $\face{F_1}$, $\face{F_2}$ as in Theorem~\ref{thm:F}. If the intersections of $S$ with the stable and unstable foliations of $\Psi_1$, $\Psi_2$ are isotopic, we may be able to perform a mutation along $S$ which yields two distinct non-circular flows $\Psi_1^\varphi$, $\Psi_2^\varphi$ representing the same top-dimensional non-fibered face in the mutant manifold. A combinatorial version of this phenomenon occurs for manifolds t12487 and t12488; see Subsection \ref{subsec:same:face}. This leads to a question: given $k>2$ is there a 3-manifold~$M$ with $b_1(M)>2$  such that
				\begin{itemize}
					\item $M$ admits $k$ fibered faces, intersecting at a point $\alpha$, dynamically represented by topologically inequivalent circular flows $\Psi_1, \Psi_2, \ldots, \Psi_k$.
					\item The primitive integral class on $\rr_+ \cdot \alpha$  can be represented by a Thurston norm minimizing surface $S$ such that mutating $M$ along $S$ gives a non-fibered \mbox{3-manifold} $M^\varphi$ with $b_1(M^\varphi) = 1$ and $\lbrack S^\varphi \rbrack \in \C(\Psi_i^\varphi)$ for $i=1,2, \ldots, k$?
				\end{itemize} 
				
			More generally, can a face of the Thurston norm ball be dynamically represented by more than two topologically inequivalent flows? Can it be represented by infinitely many flows?


				\subsection{Veering mutants and hyperbolic geometry}
				Recall that if $\V, \V^\varphi$ are veering mutants then they are both hyperbolic \cite[Theorem 1.5]{veer_strict-angles}.  However, they do not always have the same hyperbolic volume. For instance, the veering triangulation \texttt{gLLPQccdfeffhggaagb\_201022} of the~$6^2_3$ link complement  carries a twice punctured torus with the induced triangulation $\Q_{\V, w}$ and the stable train track $\tau_{\V, w}$ satisfying $\Aut^+(\Q_{\V, w} \ | \ \tau_{\V, w}) \cong \zz/2 \oplus \zz/2$, and such that  elements of $\Aut^+(\Q_{\V, w} \ | \ \tau_{\V, w})$ can be used to construct  veering mutants of the following volumes:
				\begin{table}[h]
					\begin{tabular}{|c|c|}
						\hline
						taut signature & volume \\\hline
						\texttt{gLLPQccdfeffhggaagb\_201022}&5.33348956689812\\\hline 
						\texttt{gLLPQccdfeffhwraarw\_201022}& 5.33348956689812\\\hline
						\texttt{gLLPQbefefefhhxhqhh\_211120}&5.07470803204827\\ \hline
						\texttt{gLLPQbefefefhhhhhha\_011102} &5.07470803204827\\ \hline 		
					\end{tabular}
				\end{table}
				
				Ruberman studied mutations of hyperbolic 3-manifolds and found sufficient conditions for a mutant of a hyperbolic 3-manifold to be a hyperbolic 3-manifold of the same volume. One of his results concerns only mutating via very special types of involutions of certain surfaces \cite[Theorem 1.3]{Ruberman-mutations}, another concerns only mutating along surfaces which are not virtual fibers \cite[Theorem 4.4]{Ruberman-mutations}. Conditions of neither of these theorems are satisfied  when mutating the first triangulation from the table to either the third or the fourth one; the mutating involutions are not of the type required by \cite[Theorem~1.3]{Ruberman-mutations}, and the mutating surface is a fiber of a fibration over the circle. 
				Nonetheless, there are plenty of veering mutants with the same hyperbolic volume. In particular, all non-homeomorphic veering mutants discussed in Section 
				\ref{sec:faces} have the same volume. It would be interesting to know if it is possible to figure out purely combinatorially when does a veering mutation along a carried surface result in a hyperbolic 3-manifold of the same volume. The  relationship between a veering structure and a hyperbolic structure is still not well understood, and perhaps analyzing it in this fairly narrow setup of mutations would give some new insight on the matter.
				

				\color{black}
				\bibliographystyle{abbrv}
				\bibliography{mybib}
				%
			\end{document}